\documentclass[12pt,a4paper]{article}

\usepackage[british]{babel}

\usepackage[a4paper,top=2cm,bottom=2cm,left=2.5cm,right=2.5cm,marginparwidth=1.75cm]{geometry}

\usepackage[backend=biber, style=numeric,  sorting=nyt, giveninits=true, maxnames=99,
  minnames=1]{biblatex}
\DeclareNameAlias{default}{family-given}

\usepackage{graphicx}
\usepackage{amsfonts}
\usepackage[fleqn]{amsmath}

\usepackage{amsthm}
\usepackage{mathtools}
\usepackage{multirow}
\usepackage{float}
\usepackage{xcolor,soul}
\usepackage{bm}
\usepackage{enumerate}
\usepackage{yhmath}
\usepackage{indentfirst}
\usepackage{amsmath}
\usepackage{stmaryrd}
\usepackage{algorithm}
\usepackage{algpseudocode}
\usepackage{amssymb}

\usepackage{booktabs}

\usepackage{subcaption}
\usepackage{adjustbox}
 \usepackage{comment} 
 \usepackage{enumerate}
\usepackage{enumitem}  
\usepackage{caption}

 \usepackage{tikz}
\usepackage{pgfplots}
\pgfplotsset{compat=1.18}
\usetikzlibrary{calc,positioning}
\pgfmathsetseed{\pdfuniformdeviate 10000000}
\usetikzlibrary{angles,quotes}
\usetikzlibrary{arrows,patterns}

\tikzset{
mydot/.style={
  fill,
  circle,
  inner sep=1pt
  }
}

\usepackage{commath}
\usepackage{stmaryrd}

\usepackage{etoolbox}
\usepackage{booktabs}

\usepackage{empheq}

\newcommand{\resizeboxlarger}[1]{
\resizebox{\ifdim\width>\textwidth\textwidth\else\width\fi}{!}{#1}}

\newcommand{\diff}{\mathop{}\!\mathrm{d}}
\newcommand{\mycomment}[1]{}

\makeatletter
\newcommand{\opnorm}{\@ifstar\@opnorms\@opnorm}
\newcommand{\@opnorms}[1]{%
$\left|\mkern-1.5mu\left|\mkern-1.5mu\left|
   #1
  \right|\mkern-1.5mu\right|\mkern-1.5mu\right|
$}
\newcommand{\@opnorm}[2][]{%
  \mathopen{#1|\mkern-1.5mu#1|\mkern-1.5mu#1|}
  #2
  \mathclose{#1|\mkern-1.5mu#1|\mkern-1.5mu#1|}
}
\makeatother

\usepackage[justification=raggedright, singlelinecheck=false]{caption}
\usepackage{helvet}  
\usepackage{lmodern}  

\usepackage{setspace}
\usepackage{commath}
\usepackage{stmaryrd}
\newtheorem{theorem}{Theorem}[section]
\newtheorem{lemma}{Lemma}[section]

\usepackage{enumitem}   

\newtheorem*{remark}{Remark}

\usepackage{titlesec}
\titleformat{\section} 
  {\normalfont\Large\bfseries}{\thesection.}{1em}{}
  
\usepackage{float}   
\usepackage{caption}


\begin{document}


\begin{center}
{\Large Optimal error estimates for a discontinuous Galerkin method on curved boundaries with polygonal meshes} \\[5ex]

{\large Ad\'erito Ara\'ujo and Milene Santos } \\
{\small
 Centre for Mathematics, University of Coimbra, 3000-143 Coimbra, Portugal\\
alma@mat.uc.pt, milene@mat.uc.pt \\
}
\end{center}

\begin{abstract}
We consider a discontinuous Galerkin method for the numerical solution of boundary value problems in two-dimensional domains with curved boundaries. A key challenge in this setting is the potential loss of convergence order due to approximating the physical domain by a polygonal mesh. Unless boundary conditions can be accurately transferred from the true boundary to the computational one, such geometric approximation errors generally lead to suboptimal convergence. 

To overcome this limitation, a higher-order strategy based on polynomial reconstruction of boundary data was introduced for classical finite element methods in \cite{RuaS2017Var,ruas2019OptimalLA} and in the finite volume context in \cite{BOULARAS2017401,ROD2018}. More recently, this approach was extended to discontinuous Galerkin methods in \cite{DGROD}, leading to the DG--ROD method, which restores optimal convergence rates on polygonal approximations of domains with curved boundaries. 

In this work, we provide a rigorous theoretical analysis of the DG--ROD method, establishing existence and uniqueness of the discrete solution and deriving error estimates for a two-dimensional linear advection-diffusion-reaction problem with homogeneous Dirichlet boundary conditions on both convex and non-convex domains. Following and extending techniques from classical finite element methods \cite{ruas2019OptimalLA}, we prove that, under suitable regularity assumptions on the exact solution, the DG--ROD method achieves optimal convergence despite polygonal approximations. Finally, we illustrate and confirm the theoretical results with a numerical benchmark considering triangular meshes.
\end{abstract}

\textbf{Keywords}: Arbitrary curved boundaries, Discontinuous Galerkin method, Reconstruction for off-site data method, Error estimate.

\section{Introduction}\label{intro}
In this work, we study an approach for solving boundary value problems posed in a curved boundary domain of arbitrary shape in the context of discontinuous Galerkin (DG) methods. The study of boundary value problems in curved boundary domains is a subject of growing interest in the numerical analysis community. One of the major problems is the reduction in the order of convergence of numerical methods when considering the approximation of the domain by a polygonal mesh. In particular, the DG solutions are highly sensitive to the accuracy of approximations of the curved boundaries \cite{bassi}. It has been shown that, given homogeneous Dirichlet boundary conditions on a curved boundary domain $\Omega$, if these conditions are imposed on the polygonal domain $\Omega_h$, any finite element method will be at most second-order accurate \cite{BergerPolygonal}. This highlights the importance of the boundary condition treatment since the errors in the boundary may pollute the solution inside the domain.

Over the past few decades, several techniques have been developed to remedy this loss of accuracy. There are two main strategies to address this issue. The isoparametric finite element method \cite{bassi} and the isogeometric analysis \cite{CAD_NURBS} aim to reduce the geometric error without modifying the variational form. Therefore, this technique requires the construction of a mesh with curved elements on the boundary, which is a challenging geometric problem where ineligible cells can be produced. Moreover, this approach also raises some numerical challenges since it considers non-constant Jacobian transformations from the reference element.

Another strategy considers a polygonal approximation domain $\Omega_h$ and focuses on a modified variational formulation. There has been a growing body of research focused on correcting the error that results from the approximation of the physical boundary $\partial \Omega$ by a polygonal boundary $\partial \Omega_h$, by modifying the boundary condition. In \cite{Krivodonova}, the authors consider a computational polygonal domain in place of the physical domain and modify the normal vector involved in the wall boundary condition. However, this method can only be formulated for slip-wall boundary conditions, and the work is limited to 2D geometries. In \cite{Zhang}, the author proposes a modified DG scheme defined on polygonal meshes that avoids integrals inside curved elements. However, integrations along boundary curve segments are still necessary. This approach was extended to solving three-dimensional Euler equations and it was simplified by considering the relation between the normal vector of the computational domain and the surface Jacobian \cite{Euler2020}. In the Shifted Boundary Method (SBM), the location where the boundary conditions are applied is shifted from the true boundary to an approximate (surrogate) boundary. The value of boundary conditions is modified by means of Taylor expansion, in order to reflect this displacement (see \cite{ATALLAH2022114885} and the references therein). In \cite{doi:10.1137/17M1154266}, the authors design and analyse a hybrid high-order (HHO) method on unfitted meshes to approximate elliptic interface problems. Sharp geometric assumptions are introduced for $C^2$ curved interfaces, establishing discrete inverse trace estimates and continuous trace inequalities for shape-regular elements cut by the interface. Within the framework of HHO methods, in \cite{10.1007/978-3-031-20432-6_15}, the authors introduced a formulation that combines Nitsche's technique to weakly impose boundary conditions on curved and complicated Lipschitz domains.  Another method proposed in the literature is the so-called Extensions from Subdomains \cite{Solano}. The main idea of this approach consists of determining a new Dirichlet boundary condition on the polygonal computational boundary from the one evaluated on the physical boundary.

In \cite{hpDG2021}, the domain is discretised with meshes comprising general, essentially arbitrarily-shaped element shapes. In this framework, the authors establish the stability and error analysis of the interior-penalty DG method through new extensions of trace inverse estimates for more general curved element shapes.
More recently, \cite{BERTOLUZZA2022115454} presents a virtual element discretisation framework for weakly imposing non-homogeneous Dirichlet boundary conditions on two- and three-dimensional domains with curved boundaries. In particular, the authors provide an analysis of the geometric approximation error when using polytopal domains to approximate curved domains in both 2D and 3D.

The development of the method considered in this work began in the finite element (FE) context with \cite{RuaS2017Var}, which introduced polygonal meshes for curved domains while preserving high-order accuracy. Subsequent works \cite{ruas2019OptimalLA, RUAS2021113523, RUAS2025} refined the approach and provided a rigorous theoretical analysis in the FE setting. Independently, a similar strategy was proposed for finite volume (FV) methods, first in \cite{BOULARAS2017401,ROD2018} and later extended in \cite{COSTA2022115064, COSTA2023116274, COSTA2019112560, COSTA2021110604}, where it became known as the ROD (Reconstruction for Off-site Data) method. The approach was also adapted to finite difference (FD) schemes on Cartesian grids \cite{CLAIN2021110217}. Building on these developments, we extended this technique for DG schemes in \cite{DGROD} considering triangular meshes, thereby introducing the so-called DG--ROD method.

Although numerical evidence demonstrates the effectiveness of the ROD approach in FV, FD, and DG settings, a rigorous theoretical analysis has so far been limited to FE. In this work, we extend the theoretical framework to DG by establishing error estimates for a two-dimensional linear advection–diffusion–reaction problem with homogeneous Dirichlet boundary conditions, in both the DG norm and the $L^2$ norm.

The DG--ROD method can be implemented either iteratively -- alternating between the DG solver and the polynomial reconstruction -- or via a global formulation solved in a single step. Here, we adopt the latter approach, providing an efficient and robust strategy for accurately handling curved boundaries in DG simulation

This document is organised as follows. Section \ref{DGROD_method} introduces the basic concepts, including mesh notations, the space of discontinuous functions, and the formulation of the problem under consideration. In Section \ref{existenceUniq}, we analyse fundamental properties of the method, establishing the boundedness of the bilinear form and proving an inf-sup condition. We also study the existence and uniqueness of the solution for the advection–diffusion–reaction problem with homogeneous Dirichlet boundary conditions.

The core of this work is presented in Section \ref{errorEstimates}, where we derive error estimates for the DG--ROD method introduced in Section \ref{DGROD_method}, considering both convex and non-convex domains. This analysis follows the theoretical framework developed for the classical finite element method \cite{ruas2019OptimalLA}. A key aspect of the analysis concerns the Galerkin orthogonality. The convex and non-convex cases lead to fundamentally different arguments, since in the non-convex setting, the Galerkin orthogonality no longer holds due to the geometric approximation of the domain. This loss constitutes the main additional technical challenge in the non-convex case. For convex domains, we prove that the DG--ROD solution achieves the optimal $\mathcal{O}(h^{N+1})$ convergence rate in the $L^2$ norm when $N$-degree piecewise polynomials are used, under appropriate regularity assumptions on the solution. 

Finally, Section \ref{NumResults} reports numerical experiments and results for triangular meshes, and Section \ref{conclusions} summarises the main findings, providing final comments and perspectives for future work.

\section{The DG Variational Formulation}\label{DGROD_method}
This section addresses a DG formulation for a two-dimensional linear boundary value problem on a curved boundary domain, which is discretised with piecewise linear elements. This method has the advantage of overcoming the difficulties inherent to curved mesh approaches by discretising the physical domain with polygonal meshes constructed from the conventional meshing algorithms, where piecewise linear elements approximate the arbitrary curved boundary. The main idea of the DG method is based on the use of discontinuous functions to obtain an approximate solution. Additionally, the DG--ROD method employs specific polynomial reconstructions for the prescribed boundary conditions on the physical boundary. From a methodological point of view, this approach constitutes a DG extension of the method proposed in \cite{RuaS2017Var,ruas2019OptimalLA}. Since the method studied in this paper is formulated within a similar variational setting as in \cite{ruas2019OptimalLA}, the description provided in this section should be understood as an adaptation of the corresponding methodology to the DG context. The first step is the definition of the mesh and the broken polynomial spaces. After providing an overview of some basic ideas related to computational meshes, we present the primal formulation of the method, which incorporates the modification derived from the polynomial reconstruction of the boundary conditions. 

\subsection{Model problem}
The methodology for dealing with curved boundary domains studied in this work can be applied to different equations. However, to avoid non-essential technical details, we consider the advection-diffusion-reaction equation in a two-dimensional physical domain $\Omega$ with arbitrary smooth curved physical boundary $\partial \Omega$,
considering the Cartesian coordinate system $\boldsymbol{x}=\left(x,y\right)$. We seek function $u = u(\boldsymbol{x})$, solution of the advection-diffusion-reaction problem
\begin{align}
-\Delta u\left({\boldsymbol x}\right) + \nabla \cdot (\boldsymbol{b}(\boldsymbol{x})u(\boldsymbol{x})) + c\left({\boldsymbol x}\right)u\left({\boldsymbol x}\right) &= f\left({\boldsymbol x}\right), \quad {\boldsymbol x} \in \Omega,\label{problem_E}\\
   u\left({\boldsymbol x}\right)&=0, \quad {\boldsymbol x} \in \partial \Omega,\label{problem_B}
\end{align} where $f \in L^2(\Omega)$,  $c \in L^{\infty}(\Omega)$ and $\boldsymbol{b} = \boldsymbol{b}(\boldsymbol{x}) = \left(b_1(\boldsymbol{x}), b_2(\boldsymbol{x})\right)^{\textrm{T}}$,  with $b_1, b_2 \in W^1_{\infty}(\Omega)$. Assume that
\begin{equation}\label{cx_bx_cond}
   c\left({\boldsymbol x}\right) + \frac{1}{2} \nabla \cdot \boldsymbol{b}(\boldsymbol{x}) \geq 0, \quad \mbox{for $\boldsymbol x \in \overline{\Omega}$.} 
\end{equation}
The Lebesgue space $L^2(\Omega)$ is defined as a space of mensurable functions $u: \Omega \to \mathbb{R}$ such that $||u ||^2_{L^2(\Omega)} < +\infty,$ equipped with norm $||u ||^2_{L^2(\Omega)}=(u,u)_{L^2(\Omega)}$ and inner product 
\[ (u,w)_{L^2(\Omega)} = \int_{\Omega} u(\boldsymbol{x})w(\boldsymbol{x}) \dif\boldsymbol{x}.
\]

For simplicity of presentation, we restrict our analysis to homogeneous Dirichlet boundary conditions. This assumption is not essential, since nonhomogeneous Dirichlet data can be treated by a standard lifting argument.

Throughout this work, we assume that $\partial\Omega$ is piecewise of class $C^2$, namely, each smooth boundary segment is of class $C^2$. Additional regularity assumptions on $\partial\Omega$ may be required for higher-order error estimates. In particular, the derivation of the $L^2$-error estimates in Theorems~\ref{error_L2} and~\ref{errorNonConvex_L2} requires stronger boundary regularity.


\subsection{Definition of the mesh}
The physical domain $\Omega$ is meshed with $K$ non-overlapping straight-sided triangles $T^k, k=1,\ldots,K$, leading to an approximate computational domain $\Omega_h$ given as
\begin{equation}\label{eqmalha}
\Omega_h= \bigcup_{k=1}^K T^k.
\end{equation}
The triangulation $\mathcal{T}_h = \left\{T^k, k=1,\ldots,K\right\}$ is assumed to be conformed, where the intersection of two elements is either a complete edge, a vertex, or the empty set. We assume that no element $T^k$ has more than one edge on $\partial \Omega_h $ and all the vertices of the polygon lie on $\partial \Omega$. The space parameter $h$ represents the maximum element diameter, namely
\[h= \max_{T^k \in \mathcal{T}_h } \{h_k\}, \qquad h_k= \sup_{P_1,P_2 \in T^k} \left\Vert P_1-P_2\right\Vert.
\]
The triangulation is also assumed to be regular~\cite{Ghalati2017} in the sense that there is a constant $\rho >0$ such that
\begin{equation}\label{shape_regularity}
    \forall T^k \in \mathcal{T}_h, \quad \frac{h_k}{\rho_k} \leq \rho,
\end{equation} where $\rho_k$ denotes the maximum radius of a ball inscribed in $T^k$. In addition to the mesh regularity requirements in \eqref{shape_regularity}, we assume that all boundary vertices lie on $\partial\Omega$ and that any corners (points where $\partial\Omega$ is not $C^2$) are resolved as mesh vertices. Furthermore, assuming that $h$ is sufficiently small to resolve the local curvature of $\partial\Omega$, the boundary $\partial\Omega_h$ acts as a piecewise linear interpolation of the $C^2$ boundary segments. This directly implies the geometric error estimate
\begin{equation}\label{distance_boundaries}
    \mathrm{dist}_H(\partial\Omega, \partial\Omega_h) \leq C_{\partial \Omega} h^2,
\end{equation}
where $\mathrm{dist}_H$ denotes the Hausdorff distance and $C_{\partial \Omega}$ depends on $\partial\Omega$.

Let $\mathcal{E}_h$ denote all edges of elements in $\mathcal{T}_h$ and $\mathcal{E}_0$ denote the subset of $\mathcal{E}_h$ consisting of the edges not contained in $\partial\Omega_h$. We assume that exists a positive constant $\mu$ such that for every element $T^k \in \mathcal{T}_h$ and $e \in \mathcal{E}_h \cap \partial T^k$, we have \cite{MU2014432}
\begin{equation}\label{meshIneq_hk_he}
    \mu h_k \leq h_e,
\end{equation} where $h_e$ denotes the length of the edge $e$. 
%

Given an element $T^k$, denote as $I^k$ the index set of the elements $T^\ell$ that share a common edge $e^{k\ell}$ and by $I^B$ the index set of elements which have an edge on the boundary, $e^{kB}$. Normal vector $\boldsymbol{n}^{k \ell}$, $\ell \in I^k$, pointing outward of element $T^k$ and $\boldsymbol{n}^{ \ell k} = - \boldsymbol{n}^{k \ell}$. In order to avoid non-essential technical details, assume that the mesh is designed such that convex and concave portions of $\partial \Omega$ correspond to convex and concave portions of  $\partial \Omega_h$ (see Figure \ref{BoundElem}). This property is guaranteed when the mesh is sufficiently fine, and the points separating such portions of $\partial \Omega$ are vertices of the polygon $\Omega_h$. For each element $T^k, k \in I^B$, denote as $\Delta_{k}$ the closed set delimited by $\partial \Omega$ and the edge $e^{kB}$ with the smallest area (see Figure \ref{BoundElem}). Consider $\mathcal{Q}^B$ a subset of $I^B$ such that $\mathcal{Q}^B$ denotes the index set of elements that have an edge on the boundary and $T^k \setminus \Omega$ is not restricted to a pair of vertices on $\partial \Omega$.

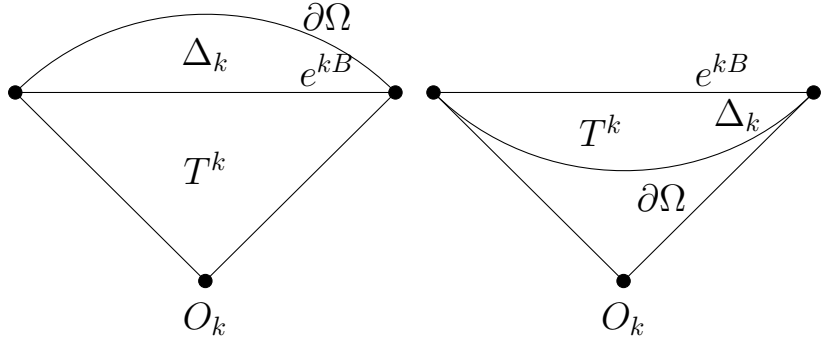
\begin{figure}
 \centering
 \begin{tikzpicture}[> = stealth,  shorten > = 1pt,   auto,   node distance = 1.7cm]
\coordinate (O) at (0,0);
\coordinate (A) at (2.5,2.5);
\coordinate (B) at (-2.5,2.5);
\coordinate (O1) at (5.5,0);
\coordinate (A1) at (8,2.5);
\coordinate (B1) at (3,2.5);

\draw[yshift=0cm]  (-2.5,2.5) to [bend left=45] (2.5,2.5);
\draw[yshift=0cm]  (3,2.5) to [bend right=45] (8,2.5);

\draw (O) -- (A) -- (B) -- cycle;
\draw (O1) -- (A1) -- (B1) -- cycle;

\node[circle,draw=black, fill=black, inner sep=0pt,minimum size=5pt] at (A) {};
\node[circle,draw=black, fill=black, inner sep=0pt,minimum size=5pt] at (B) {};
\node(teste)[circle,draw=black, fill=black, inner sep=0pt,minimum size=5pt] at (O) {};
\node[circle,draw=black, fill=black, inner sep=0pt,minimum size=5pt] at (A1) {};
\node[circle,draw=black, fill=black, inner sep=0pt,minimum size=5pt] at (B1) {};
\node(teste)[circle,draw=black, fill=black, inner sep=0pt,minimum size=5pt] at (O1) {};

\node at (0,1.5) {\large $T^k$}; 
\node at (1.6,2.8) {\large $e^{kB}$}; 
\node at (1.6,3.5) {\large $\partial \Omega$};
\node at (0,3) {\large $\Delta_k$}; 
\node at (0,-0.5) {\large $O_k$}; 

\node at (5.2,2) {\large $T^k$}; 
\node at (6.8,2.8) {\large $e^{kB}$}; 
\node at (6,1.1) {\large $\partial \Omega$};
\node at (7,2.2) {\large $\Delta_k$}; 
\node at (5.5,-0.5) {\large $O_k$};

\end{tikzpicture}
\caption{Element $T^k$ with an edge $e^{kB}$ on the computational boundary $\partial \Omega_h$, for the convex case where $T^k \subset \Omega$ (left panel) and for the concave case, where $T^k \not\subset \Omega$ (right panel).}
\label{BoundElem}
\end{figure}

\subsection{Space of discontinuous functions}
\label{subsec:spacediscfunc}
The discontinuous Galerkin method is based on the use of discontinuous approximations. Thus, we introduce the broken Sobolev spaces 
\begin{align*}
    H^m(\mathcal{T}_h) &= \{w \in L^2(\Omega_h): w_{\vert_{T^k}} \in H^m(T^k) \; \forall T^k \in \mathcal{T}_h\}, \\
    W^m_p(\mathcal{T}_h) &= \{w \in L^p(\Omega_h): w_{\vert_{T^k}} \in W^m_p(T^k) \; \forall T^k \in \mathcal{T}_h\},
\end{align*} where $m \geq 0$ is an integer and $1\leq p \leq \infty$ a real number.
Note that %
\[(w,v)_{L^2(\Omega_h)} = \sum_{k=1}^K (w,v)_{L^2(T^k)}, \; \forall w, v \in {L^2(\Omega_h)},\] where $(u,v)_{L^2(T^k)}$ denotes the usual inner product on $L^2(T^k)$. For $w \in H^{m}(\mathcal{T}_h)$,  with $m\geq 0$ an integer, we define the norm
\[\norm{w}^2_{H^{m}(\mathcal{T}_h)} = \sum_{k=1}^K \norm{w}^2_{H^{m}(T^k)},\] and the seminorm
\[\abs{w}^2_{H^{m}(\mathcal{T}_h)} = \sum_{k=1}^K \abs{w}^2_{H^{m}(T^k)},\]
where $\norm{w}_{H^{m}(T^k)}$ and $\abs{w}_{H^{m}(T^k)}$ denote the usual $H^{m}$-norm and $H^{m}$-seminorm on the element $T^k$, respectively. We also introduce the following function space
\begin{equation*}
    H(\mbox{div};\Omega) = \{ \boldsymbol{q} \in [L^2(\Omega)]^2: \nabla \cdot \boldsymbol{q} \in L^2(\Omega) \}.
\end{equation*} For all $T^k \in \mathcal{T}_h$, we define the function space $H(\mbox{div}; T^k)$ by replacing $\Omega$ by $T^k$ in  $H(\mbox{div}; \Omega)$. We then introduce the broken space
\begin{equation*}
    H(\mbox{div};\mathcal{T}_h) = \{ \boldsymbol{q} \in [L^2(\Omega)]^2: \forall T^k \in \mathcal{T}_h,  \boldsymbol{q}_{\vert_{T^k}} \in H(\mbox{div};T^k) \}.
\end{equation*}
We define the space of discontinuous  piecewise polynomial functions 
\[S_{hN} = \left\{ v \in L^2(\Omega_h): v_{\vert_{T^k}} \in \mathcal{P}_N\left(T^k\right) \forall T^k \in \mathcal{T}_h\right\},\]
with $\mathcal{P}_N\left(T^k\right)$ denoting the space of polynomials of degree less than or equal to $N$ in element $T^k$. We also introduce broken operators by restriction to each element $T^k \in \mathcal{T}_h$ as follows:
\begin{itemize}
    \item The broken gradient operator $\nabla_h: W^1_p(\mathcal{T}_h) \to [L^p(\Omega_h)]^2$ is defined by $(\nabla_hv)_{\vert_{T^k}} = \nabla(v_{\vert_{T^k}})$, for $T^k \in \mathcal{T}_h$, $v \in W^1_p(\mathcal{T}_h)$.
    \item The broken divergence operator $\nabla_h \cdot: H(\mbox{div}; \mathcal{T}_h) \to L^2(\Omega_h)$ is defined by $(\nabla_h \cdot \boldsymbol{q})_{\vert_{T^k}} = \nabla \cdot (\boldsymbol{q}_{\vert_{T^k}})$, for $T^k \in \mathcal{T}_h$, $\boldsymbol{q} \in H(\mbox{div}; \mathcal{T}_h)$.
\end{itemize}

Let $\Gamma = \cup_{T^k \in \mathcal{T}_h} \partial T^k$ and $\Gamma_0 = \Gamma \setminus \partial \Omega_h$, the traces of functions in $H^1(\mathcal{T}_h)$ belong to $T(\Gamma) = \Pi_{T^k \in \mathcal{T}_h} L^2(\partial T^k)$. Note that $v \in T(\Gamma)$ may be double-valued on $\Gamma_0$ and is single-valued on $\partial \Omega_h$.

We introduce some operators that will be useful for manipulating the numerical fluxes and obtaining the primal formulation. Let $e^{k \ell}$ be an edge shared by the elements $T^k$ and $T^\ell$. For $\boldsymbol{q} \in  [T(\Gamma)]^2 $ and  $v \in  T(\Gamma)$, we define the averages $\{\!\{\boldsymbol{q}\}\!\}^{ k\ell}$ and $\{\!\{ v\}\!\}^{ k\ell}$ and the jumps $\llbracket \boldsymbol{q}\rrbracket^{k\ell}$ and $\llbracket v\rrbracket^{k\ell}$ as follows:
\begin{align*}
&\{\!\{\boldsymbol{q}\}\!\}^{ k\ell} = \frac{\boldsymbol{q}^k + \boldsymbol{q}^\ell}{2},
\qquad \llbracket \boldsymbol{q}\rrbracket^{k\ell} = \boldsymbol{n}^{k\ell}\cdot (\boldsymbol{q}^{k} - \boldsymbol{q}^\ell),
 \\
&\{\!\{ v\}\!\}^{ k\ell} = \frac{ v^k + v^\ell}{2}, \qquad \llbracket v\rrbracket^{k\ell}= \boldsymbol{n}^{k\ell}(v^{k} -v^\ell).
\end{align*}
For a boundary edge $e^{kB}$, we define
\begin{align*}
&\{\!\{\boldsymbol{q}\}\!\}^{ kB} = \boldsymbol{q}^k,
\;\quad
\{\!\{ v\}\!\}^{ kB} = v^k,
\;\quad \llbracket \boldsymbol{q}\rrbracket^{kB} = \boldsymbol{n}^{kB}\cdot \boldsymbol{q}^{k}, \;\quad \llbracket v\rrbracket^{kB} = \boldsymbol{n}^{kB} v^{k}.
\end{align*} When it is clear which edge we are referring to, we usually omit the superscript $k\ell$ and simply write $\{\!\{ \cdot \}\!\}$ and $\llbracket \cdot \rrbracket$. A convenient norm with which to carry out the analysis of the method is the following \cite{ArnoldBrezziCockburnMarini, Ern_Pietro}
\begin{align}\label{DG_norm}
    \opnorm{u}^2 &=\sum_{k=1}^K \left(\norm{u}^2_{H^1(T^k)}+h_k^2 \abs{u}^2_{H^2(T^k)} \right)+\sum_{e \in \mathcal{E}_h} h_e^{-1} \norm{\llbracket u\rrbracket}^2_{L^2(e)} \nonumber\\
    &\quad+\sum_{e \in \mathcal{E}_h}\frac{1}{2} \norm{ \abs{\boldsymbol{b}\cdot \boldsymbol{n}_e}^{1/2}\llbracket u\rrbracket}^2_{L^2(e)}, 
\end{align} for $u \in H^2(\mathcal{T}_h)$. For notational convenience, let
\begin{equation}\label{semiNormDG}
    \abs{v}^2_{\ast} = \sum_{e \in \mathcal{E}_h} h_e^{-1} \norm{\llbracket v \rrbracket}^2_{L^2(e)} \quad \mbox{and} \quad \abs{v}^2_{b} = \sum_{e \in \mathcal{E}_h} \frac{1}{2} \norm{ \abs{\boldsymbol{b}\cdot \boldsymbol{n}_e}^{1/2}\llbracket v\rrbracket}^2_{L^2(e)},
\end{equation} for $v \in T(\Gamma)$. Using a inequality (\cite{Brenner}, Lemma $4.5.3$), we may prove that, for $v \in \mathcal{P}_N(T^k)$, with $T^k \in \mathcal{T}_h$,
\begin{equation}\label{Des_H2_H1_lemma4.5.3}
   h_k \abs{v}_{H^2(T^k)} \leq C \abs{v}_{H^1(T^k)}.
\end{equation} Thus, for $v \in S_{hN}$
\begin{align}\label{DG_norm_simpl}
   \opnorm{v}^2 &=\sum_{k=1}^K \left(\norm{v}^2_{L^2(T^k)} + \abs{v}^2_{H^1(T^k)} +  h_k^2 \abs{v}^2_{H^2(T^k)} \right) +\abs{v}^2_{\ast} + \abs{v}^2_{b}\nonumber\\
    &\leq (1+C^2) \left( \norm{v}^2_{H^1(\mathcal{T}_h)} +  \abs{v}^2_{\ast} + \abs{v}^2_{b}\right).
\end{align}

For each element $T^k$, with $k \in I^B$, let $I^{kB}$ be the index set of the discontinuous Galerkin nodes different from the vertices that belong to the boundary edge $e^{kB}$ (see left panel of Figure \ref{NodalSet}).

Now, we introduce two spaces $\mathcal{V}_h$ and $\mathcal{W}_h$ associated to $\mathcal{T}_h$. The space $\mathcal{V}_h$ is defined by
\[ \mathbin{\mathcal{V}_h = \left\{v \in L^2(\Omega_h):  v\vert_{\partial \Omega_h}=0, \;  v_{\vert_{T^{k}}} \in \mathcal{P}_N\left(T^k\right),\forall T^k \in \mathcal{T}_h \right\}}.
\] For convenience, we extend by 0 every function $v \in \mathcal{V}_h$ to $\Omega \setminus \Omega_h$.
 $\mathcal{W}_h$ is the space defined in $\Omega_h$ that satisfies the following properties for $w \in \mathcal{W}_h$:
\begin{enumerate}[label=\textnormal{(\arabic*)}]
    \item $ w_{\vert_{T^{k}}} \in \mathcal{P}_N\left(T^k\right), \forall T^k \in \mathcal{T}_h$;
    \item $w \in L^2(\Omega_h)$;
    \item The expression of $w$ is extended to $\Omega \setminus \Omega_h$ in such a way that its polynomial expression in $T^k, k \in I^B$, also applies in $\Delta_{k}$;
    \item $w$ vanishes at the vertices of $\partial \Omega_h$ and $w(P^k_r) = 0$, $r= 1,\ldots, N-1$, $\forall T^k \in \mathcal{T}_h, k \in I^B$ (where each point $P^k_r$ is chosen to be the nearest intersection with the physical boundary $\partial \Omega$ of the line passing through the vertex $O_k$ of $T^k$ not belonging to $\partial \Omega$ and one of $N-1$ discontinuous Galerkin nodes $\boldsymbol{x}_i^k, i \in I^{kB}$, lying on the associated boundary edge, $e^{kB}$). Thus, $w$ vanishes at $N+1$ points on $\partial \Omega$. \label{item4Def_Wh}
\end{enumerate}

For notational purposes, assume that the vertices of the element $T^k$, $k\in I^B$, on $\partial \Omega$ are denoted by $P^k_N$ and $P^k_{N+1}$. Thus, according to property \ref{item4Def_Wh}, we may write $w(P^k_r) = 0$, $r= 1,\ldots, N+1$, $\forall T^k \in \mathcal{T}_h, k \in I^B$.  An example of the nodes associated with $\mathcal{W}_h$ is reported in Figure \ref{NodalSet}. Namely, the discontinuous Galerkin nodal set and the points $P_r^k$, $r=1, \ldots, N-1$, resulting from a projection of the nodal points lying on the boundary edge $e^{kB}$. For the non-convex case, the points $P_r^k$ are obtained using the same approach.

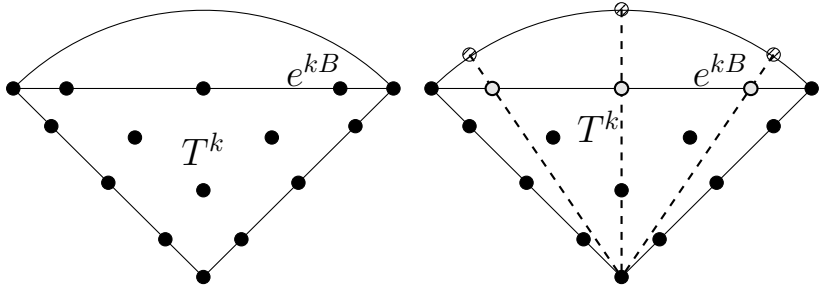
\begin{figure}
\centering
\begin{tikzpicture}[> = stealth,  shorten > = 1pt,   auto,   node distance = 1.7cm]
\coordinate (O) at (0,0);
\coordinate (A) at (2.5,2.5);
\coordinate (B) at (-2.5,2.5);

\coordinate (C1) at (1.8,2.5);
\coordinate (C2) at (-1.8,2.5);
\coordinate (C3) at (0,2.5);

\coordinate (D1) at (0.5,0.5);
\coordinate (D2) at (-0.5,0.5);
\coordinate (D3) at (1.25,1.25);
\coordinate (D4) at (-1.25,1.25);
\coordinate (D5) at (2,2);
\coordinate (D6) at (-2,2);

\coordinate (E1) at (0.9,1.85);
\coordinate (E2) at (-0.9,1.85);
\coordinate (E3) at (0,1.15);

\coordinate (O1) at (5.5,0);
\coordinate (A1) at (8,2.5);
\coordinate (B1) at (3,2.5);

\coordinate (C4) at (7.2,2.5);
\coordinate (C5) at (3.8,2.5);
\coordinate (C6) at (5.5,2.5);

\coordinate (M1) at (7.5,2.95);
\coordinate (M2) at (3.5,2.95);
\coordinate (M3) at (5.5,3.55);

\coordinate (D7) at (6,0.5);
\coordinate (D8) at (5,0.5);
\coordinate (D9) at (6.75,1.25);
\coordinate (D10) at (4.25,1.25);
\coordinate (D11) at (7.5,2);
\coordinate (D12) at (3.5,2);

\coordinate (E4) at (6.4,1.85);
\coordinate (E5) at (4.6,1.85);
\coordinate (E6) at (5.5,1.15);

\draw[yshift=0cm]  (-2.5,2.5) to [bend left=45] (2.5,2.5);
\draw[yshift=0cm]  (3,2.5) to [bend left=45] (8,2.5);

\draw (O) -- (A) -- (B) -- cycle;
\draw (O1) -- (A1) -- (B1) -- cycle;

\draw[thick,black,dashed] (M1) -- (O1);
\draw[thick,black,dashed] (M2) -- (O1);
\draw[thick,black,dashed] (M3) -- (O1);

\node[circle,draw=black, fill=black, inner sep=0pt,minimum size=5pt] at (A) {};
\node[circle,draw=black, fill=black, inner sep=0pt,minimum size=5pt] at (B) {};
\node(teste)[circle,draw=black, fill=black, inner sep=0pt,minimum size=5pt] at (O) {};
\node[circle,draw=black, fill=black, inner sep=0pt,minimum size=5pt] at (A1) {};
\node[circle,draw=black, fill=black, inner sep=0pt,minimum size=5pt] at (B1) {};
\node(teste)[circle,draw=black, fill=black, inner sep=0pt,minimum size=5pt] at (O1) {};

\node at (0,1.7) {\large $T^k$}; 
\node at (1.45,2.75) {\large $e^{kB}$}; 

\node at (5.2,2) {\large $T^k$}; 
\node at (6.8,2.75) {\large $e^{kB}$};

\node[circle,draw=black, fill=black, inner sep=0pt,minimum size=5pt] at (C1) {};
\node[circle,draw=black, fill=black, inner sep=0pt,minimum size=5pt] at (C2) {};
\node[circle,draw=black, fill=black, inner sep=0pt,minimum size=5pt] at (C3) {};
\node[circle,draw=black, fill=black, inner sep=0pt,minimum size=5pt] at (D1) {};
\node[circle,draw=black, fill=black, inner sep=0pt,minimum size=5pt] at (D2) {};
\node[circle,draw=black, fill=black, inner sep=0pt,minimum size=5pt] at (D3) {};
\node[circle,draw=black, fill=black, inner sep=0pt,minimum size=5pt] at (D4) {};
\node[circle,draw=black, fill=black, inner sep=0pt,minimum size=5pt] at (D5) {};
\node[circle,draw=black, fill=black, inner sep=0pt,minimum size=5pt] at (D6) {};
\node[circle,draw=black, fill=black, inner sep=0pt,minimum size=5pt] at (E1) {};
\node[circle,draw=black, fill=black, inner sep=0pt,minimum size=5pt] at (E2) {};
\node[circle,draw=black, fill=black, inner sep=0pt,minimum size=5pt] at (E3) {};

\node[circle,draw=black, fill=black, inner sep=0pt,minimum size=5pt, pattern=north east lines] at (M1) {};
\node[circle,draw=black, fill=black, inner sep=0pt,minimum size=5pt, pattern=north east lines] at (M2) {};
\node[circle,draw=black, fill=black, inner sep=0pt,minimum size=5pt, pattern=north east lines] at (M3) {};

\node[minimum size=5pt, draw,circle,font=\sffamily\Large\bfseries,inner sep=0pt, fill=black!10, thick] at (C4) {};
\node[minimum size=5pt, draw,circle,font=\sffamily\Large\bfseries,inner sep=0pt, fill=black!10, thick] at (C5) {};
\node[minimum size=5pt, draw,circle,font=\sffamily\Large\bfseries,inner sep=0pt, fill=black!10, thick] at (C6) {};

\node[circle,draw=black, fill=black, inner sep=0pt,minimum size=5pt] at (D7) {};
\node[circle,draw=black, fill=black, inner sep=0pt,minimum size=5pt] at (D8) {};
\node[circle,draw=black, fill=black, inner sep=0pt,minimum size=5pt] at (D9) {};
\node[circle,draw=black, fill=black, inner sep=0pt,minimum size=5pt] at (D10) {};
\node[circle,draw=black, fill=black, inner sep=0pt,minimum size=5pt] at (D11) {};
\node[circle,draw=black, fill=black, inner sep=0pt,minimum size=5pt] at (D12) {};
\node[circle,draw=black, fill=black, inner sep=0pt,minimum size=5pt] at (E4) {};
\node[circle,draw=black, fill=black, inner sep=0pt,minimum size=5pt] at (E5) {};
\node[circle,draw=black, fill=black, inner sep=0pt,minimum size=5pt] at (E6) {};

\end{tikzpicture}

\caption{Discontinuous Galerkin nodal set $\{\boldsymbol{x}_i^k\}_{i=1}^{N_p}$ denoted by the black dots (left panel) and points $P^k_r$, $r=1,\ldots. N-1$, denoted by the dots with diagonal lines pattern (right panel).}
\label{NodalSet}
\end{figure}

For each element $T^k$, $k \in I^B$, let $m_N = N(N+1)/2$ be the number of nodal points that do not lie on the edge $e^{kB}$. In other words, $m_N = N_p - (N+1)$, with $N_p=\left(N+1\right)\left(N+2\right)/2$. The following lemma, proved in \cite{ruas2019OptimalLA} establishes that $\mathcal{W}_h$ is a non-empty finite-dimensional space.  

\begin{lemma}
Let  $\mathcal{P}_N\left(T^k\right)$ be the space of polynomials defined in $T^k$, $k \in I^B$, of degree less than or equal to $N$. Provided $h$ small enough  $\forall T^k$, $k \in I^B$, given a set of $m_N$ real values $\gamma_i^k$,  $i=1, \ldots, m_N$, there exists a unique function $w \in \mathcal{P}_N\left(T^k\right)$ that vanishes at both vertex of $T^k$ located on $\partial \Omega$ and at the points $P^k_r$ of $\partial \Omega$, $r=1, \ldots N-1$, and takes value $\gamma_i^k$ respectively at the $m_N$ nodes of $T^k$ not located on $\partial \Omega_h$.
\end{lemma}

\subsection{Variational formulation}

Consider that the numerical solution $u_h$ has the following decomposition \cite{metodoDG}
\begin{equation}\label{approx}
u_h\left(\boldsymbol{x}\right) = \bigoplus_{k=1}^K u_h^k\left(\boldsymbol{x}\right) \in \mathcal{W}_h.
\end{equation} In each element $T^k$, the local solution $u_h^k$ has a polynomial decomposition with the two-dimensional Lagrange polynomials
\begin{equation}\label{u_hkLocal1}
\boldsymbol{x} \in T^k \in \mathcal{T}_h: \quad u_h^k\left(\boldsymbol{x}\right) = \sum_{i=1}^{N_p} u^k_i \ell_i^k\left(\boldsymbol{x}\right),
\end{equation} where $u^k_i=u_h^k\left(\boldsymbol{x}_i^k\right)$ are the nodal values of the Lagrange polynomials basis $\ell_i^k\left(\boldsymbol{x}\right)$ at points $\boldsymbol{x}_i^k\in T^k$, $i=1,\ldots,N_p$. Vector $\boldsymbol{u}^k = \left(u^k_1,\ldots,u^k_{N_p}\right)^{\textrm{T}}$ gathers the $N_p$ nodal values. 

The DG method considered herein combines the Symmetric Interior Penalty (SIP) bilinear scheme to handle the diffusion term and the upwind DG method to handle the advection-reaction terms. In other words, the elliptic numerical flux is defined by considering the internal penalty fluxes and the hyperbolic numerical flux is determined by evaluating the upwind flux. Thus, we may write the DG scheme as: find $u_h \in \mathcal{W}_h$  such that (\cite{Arnold_intPen}, \cite{Ern_Pietro}, \cite{hp-version})
\begin{align}\label{weakForm_2_upwind}
     &(\nabla_h u_h,\!\nabla_h \phi_h)_{L^2(\Omega_h)}\!-\!(\boldsymbol{b} u_h,\!\nabla_h \phi_h)_{L^2(\Omega_h)}\!+\!(c u_h,\!\phi_h)_{L^2(\Omega_h)}\nonumber \\
     &-\int_{\Gamma}\!\llbracket u_h\rrbracket \cdot\{\!\{\nabla_h \phi_h \}\!\} \dif s-\int_{\Gamma}\llbracket \phi_h\!\rrbracket \!\cdot\!\{\!\{\nabla_h u_h\}\!\}\dif s+\int_{\Gamma}\!\llbracket \phi_h\!\rrbracket\cdot\frac{\eta}{h_e} \llbracket u_h\rrbracket \dif s \nonumber \\
     &+\int_{\Gamma_0}\boldsymbol{b}\cdot \llbracket \phi_h\rrbracket \{\!\{u_h\}\!\}\dif s+\int_{\Gamma_0}\frac{\abs{\boldsymbol{b}\cdot \boldsymbol{n}_e}}{2}\llbracket \phi_h\rrbracket \cdot\llbracket u_h\rrbracket \dif s =( f, \phi_h)_{L^2(\Omega_h)}\nonumber \\
     &-\sum_{k \in I^B} \int_{e^{kB}} \left(g_D^k \boldsymbol{n}^{kB} \cdot \nabla_h \phi_h^k\!-\!\frac{\eta}{h_e} g_{D}^k \phi^k_h\!+\!\boldsymbol{b}\cdot \boldsymbol{n}^{kB} \phi^k_h g_{D}^k \right)\dif s, \forall \phi \in \mathcal{V}_h,
\end{align} where {\begin{equation}\label{penaltyConst}
    \eta = C_{\eta} (N+1)^2,
\end{equation} with $C_{\eta}$ a positive constant \cite{metodoDG}, and $\boldsymbol{n}_e$ denote the normal vector associated with the edge $e \in \Gamma_0$, i.e., the normal vector pointing outward of element $T^k$. Moreover, $g_{D}$ defines the boundary condition imposed in $\partial \Omega_h$.

Since the physical domain $\Omega$ features a curved boundary $\partial \Omega$ where the data is prescribed, the boundary values required by the DG formulation on the computational boundary $\partial \Omega_h$ are not directly accessible. To address this issue, we employ a reconstruction procedure based on the ROD method. This approach, detailed by \cite{DGROD}, enables the computation of the boundary data $g_D$ on $\partial \Omega_h$ using the information provided on the physical boundary $\partial \Omega$.
Thus, instead of imposing homogeneous Dirichlet boundary conditions on $\partial \Omega_h$, i.e. $g_D = 0$, we consider a new boundary condition $g_{ROD}$ determined by the ROD method. The polynomial $g_{ROD}$ takes into account the geometrical mismatch between $\partial \Omega$ and $\partial \Omega_h$ and it has the following decomposition
\[g_{ROD}(\boldsymbol{x}; \boldsymbol{a}) = \bigoplus_{k \in I^B} g^k\left(\boldsymbol{x}; \boldsymbol{a}^k\right),\] and in each element $T^k$, with $k \in I^B$, the local polynomial $g^k$ has a polynomial decomposition with the two-dimensional Lagrange polynomials
\begin{equation}\label{u_hLocal1}
\boldsymbol{x} \in T^k \in \mathcal{T}_h: \quad g^k\left(\boldsymbol{x}; \boldsymbol{a}^k\right) = \sum_{i=1}^{N_p} a_i^k \ell_i^k\left(\boldsymbol{x}\right),
\end{equation} where vector $\boldsymbol{a}^k = \left(a_1^k,\ldots,a_{N_p}^k\right)^\textrm{T}$ gathers the $N_p$ nodal values, $a_i^k$, and $\boldsymbol{a}$ gathers all the vectors $\boldsymbol{a}^k$, $k \in   I^B$. For each boundary element $T^k$, with $k \in I^B$, the polynomial $g^k$ is obtained as the solution of a constrained minimisation problem. This polynomial minimises the distance, in a least-squares sense, from the local solution $u^k_{h}$ while fulfilling the boundary conditions at a set of $N+1$ points on $\partial \Omega$. Thus, for each element $T^k$, $k \in I^{B}$, consider a generic polynomial 
\[
\pi\left(\boldsymbol{x};\boldsymbol{a}^k\right)=\sum_{i=1}^{N_p} a_i^k\ell_i^k\left(\boldsymbol{x}\right).
\] Moreover, consider the functional
\begin{equation}\label{def_functional_ROD}
\boldsymbol{a}^k\to \mathcal{E}^k\left(\boldsymbol{a}^k;\boldsymbol{u}^k\right) = \frac{1}{2} \sum\limits_{i=1}^{N_p} \left(a_i^k-u^k_i\right)^2,
\end{equation}
subjected to the linear constraints
\begin{equation}\label{def_constraints_ROD}
\mathcal{B}^k_r\left(\boldsymbol{a}^k\right)=
\pi\left(P_r^k;\boldsymbol{a}^k\right)=0,\quad r=1,\ldots,N+1.
\end{equation}
The polynomial $g^k$ is sought as a minimisation of functional~\eqref{def_functional_ROD} under linear constraints~\eqref{def_constraints_ROD}. The resulting local constrained minimisation problem is solved locally via the Lagrange multipliers method. From \cite{DGROD}, the minimisation problem is stated under the matrix form
%
\begin{equation}{\label{formaMatricial}}
\begin{bmatrix}
\boldsymbol{I}_{N_p} & \boldsymbol{B}^k\\
\left(\boldsymbol{B}^k\right)^\textrm{T} & \boldsymbol{0}_{N+1}
\end{bmatrix}
\begin{bmatrix}
\boldsymbol{a}^k \\
\boldsymbol{\lambda}^k
\end{bmatrix}=
\begin{bmatrix}
\boldsymbol{u}^k \\
\boldsymbol{0}
\end{bmatrix},
\end{equation} where $\boldsymbol{I}_{N_p}$ is the identity matrix in $\mathbb{R}^{N_p \times N_p}$, $\boldsymbol{0}_{N+1}$ is the null matrix in $\mathbb{R}^{(N+1) \times (N+1)}$, $\boldsymbol{0}$ is the null vector in $\mathbb{R}^{(N+1) \times 1}$, and $\boldsymbol{B}^k=\left[\boldsymbol{B}^k_1,\ldots,\boldsymbol{B}^k_{N+1}\right]$ in $\mathbb{R}^{N_p\times (N+1)}$, with
\[
\boldsymbol{B}^k_r=\Big[  \ell_j^k\left(P_r^k\right) \Big]_{j=1}^{N_p},
\]
for $r= 1,\ldots, N+1$, including the vertices of $T^k$ on $\partial \Omega_h$. Thus, 
\begin{equation*}
\begin{bmatrix}
\boldsymbol{a}^k \\
\boldsymbol{\lambda}^k
\end{bmatrix}=
\begin{bmatrix}
\boldsymbol{I}_{N_p} & \boldsymbol{B}^k\\
\left(\boldsymbol{B}^k\right)^\textrm{T} & \boldsymbol{0}_{N+1}
\end{bmatrix}^{-1}
\begin{bmatrix}
\boldsymbol{u}^k \\
\boldsymbol{0}
\end{bmatrix} =
\begin{bmatrix}
C_1 & C_2\\
C_3 & C_4
\end{bmatrix}
\begin{bmatrix}
\boldsymbol{u}^k \\
\boldsymbol{0}
\end{bmatrix},
\end{equation*} where $C_i$ denotes the $i-th$ block of the inverse matrix. Noticing that $\boldsymbol{I}_{N_p}$ is invertible, the inverse of the ${N_p \times N_p}$ matrix $C_1$ is given by \cite{bernstein2009matrix}
\begin{align*}
C_1 &= \boldsymbol{I}_{N_p}^{-1} + \boldsymbol{I}_{N_p}^{-1}\boldsymbol{B}^k \left(\boldsymbol{0}_{N+1} - \left(\boldsymbol{B}^k\right)^\textrm{T} \boldsymbol{I}_{N_p}^{-1} \boldsymbol{B}^k \right)^{-1} \left(\boldsymbol{B}^k\right)^\textrm{T}\boldsymbol{I}_{N_p}^{-1} \\
&=\boldsymbol{I}_{N_p} - \boldsymbol{B}^k \left(\left(\boldsymbol{B}^k\right)^\textrm{T} \boldsymbol{B}^k \right)^{-1} \left(\boldsymbol{B}^k\right)^\textrm{T}.
\end{align*} Thus, we get
\begin{align*}
\boldsymbol{a}^k &= \boldsymbol{u}^k -  \boldsymbol{B}^k \left(\left(\boldsymbol{B}^k\right)^\textrm{T} \boldsymbol{B}^k \right)^{-1} \left(\boldsymbol{B}^k\right)^\textrm{T} \boldsymbol{u}^k.
\end{align*}

Note that
\begin{align*}
    \left(\boldsymbol{B}^k\right)^\textrm{T} \boldsymbol{u}^k &=\begin{bmatrix}
\ell_1^k\left(P_1^k\right) & \ldots & \ell_{N_p}^k\left(P_1^k\right)\\
\ell_1^k\left(P_2^k\right) & \ldots & \ell_{N_p}^k\left(P_2^k\right)\\
\vdots & \ddots & \vdots \\
\ell_1^k\left(P_{N+1}^k\right) & \ldots & \ell_{N_p}^k\left(P_{N+1}^k\right)\\
\end{bmatrix} \begin{bmatrix}
    u_1^k \\
    u_2^k \\
    \vdots \\
    u_{Np}^k \\
\end{bmatrix}= \begin{bmatrix}
 u_h^k\left(P_1^k\right) \\
 u_h^k \left(P_2^k\right)\\
\vdots \\
 u_h^k \left(P_{N+1}^k\right)\\
\end{bmatrix} = \boldsymbol{0}, 
\end{align*} since $u_h^k \in \mathcal{W}_h$. Then $\boldsymbol{a}^k = \boldsymbol{u}^k$ and $g^k = u_h^k$, with $k \in I^B$. Replacing $g_D^k$ by $u_h^k$ in \eqref{weakForm_2_upwind}, we get for $v \in \mathcal{V}_h$ 
\begin{align}\label{weakForm_3}
     &(\nabla_h u_h,\nabla_h v)_{L^2(\Omega_h)}-(\boldsymbol{b} u_h,\nabla_h v)_{L^2(\Omega_h)}+(c u_h,v)_{L^2(\Omega_h)}\nonumber \\
     &-\int_{\Gamma}\!\llbracket u_h\rrbracket \cdot\{\!\{\nabla_h v \}\!\}\!\dif s-\int_{\Gamma}\llbracket v\rrbracket \cdot\{\!\{\nabla_h u_h\}\!\}\dif s+\int_{\Gamma}\llbracket v\rrbracket\cdot\frac{\eta}{h_e} \llbracket u_h\rrbracket \dif s \nonumber \\
     &+\int_{\Gamma_0}\boldsymbol{b}\cdot \llbracket v\rrbracket \{\!\{u_h\}\!\}\dif s+\int_{\Gamma_0}\frac{\abs{\boldsymbol{b}\cdot \boldsymbol{n}_e}}{2}\llbracket v\rrbracket \cdot\llbracket u_h\rrbracket \dif s=( f, v)_{L^2(\Omega_h)} \nonumber \\
     &-\sum_{k \in I^B} \int_{e^{kB}} \left(\!u_h^k \boldsymbol{n}^{kB} \cdot \nabla_h v_{\vert_{T^k}}
     -\frac{\eta}{h_e} u_h^k v_{\vert_{T^k}} + \boldsymbol{b}\cdot \boldsymbol{n}^{kB} v_{\vert_{T^k}} u_{h}^k \right)\!\dif s \nonumber \\
    &=\!( f, v)_{L^2(\Omega_h)}\!-\!\int_{\partial \Omega_h} \left(\!\llbracket u_h \rrbracket\!\cdot\!\{\!\{\nabla_h v \}\!\}\!-\!\frac{\eta}{h_e} \llbracket u_h\rrbracket \llbracket v \rrbracket\!+\!\boldsymbol{b}\cdot\llbracket v\rrbracket \{\!\{u_h\}\!\}\!\right)\!\dif s.
\end{align}

Attending that $v=0$ on $\partial \Omega_h$ and considering \eqref{weakForm_3}, we obtain
\begin{align}\label{weakForm_4}
  &(\nabla_h u_h, \nabla_h v)_{L^2(\Omega_h)} - (\boldsymbol{b} u_h,\!\nabla_h v)_{L^2(\Omega_h)} + (c u_h, v)_{L^2(\Omega_h)} \nonumber \\
  & -  \int_{\Gamma_0}  \llbracket u_h\rrbracket \cdot  \{\!\{ \nabla_h v\}\!\}\dif s- \int_{\Gamma_0} \llbracket v\rrbracket \cdot  \{\!\{ \nabla_h u_h\}\!\} \dif s + \int_{\Gamma} \frac{\eta}{h_e}  \llbracket v\rrbracket \cdot  \llbracket u_h\rrbracket \dif s  \nonumber \\
  &+ \int_{\Gamma_0} \boldsymbol{b}\cdot \llbracket v \rrbracket \{\!\{u_h\}\!\}\dif s+ \int_{\Gamma} \frac{\abs{\boldsymbol{b}\cdot \boldsymbol{n}_e}}{2}\llbracket v \rrbracket \cdot\llbracket u_h\rrbracket \dif s = ( f, v)_{L^2(\Omega_h)}.
\end{align}

Then, the variational problem of \eqref{problem_E}--\eqref{problem_B} can be reformulated as follows: find $u_h \in \mathcal{W}_h$ such that
\begin{equation}\label{formfraca_brokenGrad}
    B^{DG}_h(u_h,v)= (f, v)_{L^2(\Omega_h)}, \quad \forall v \in \mathcal{V}_h, 
\end{equation}
where the bilinear form $B^{DG}_h(\cdot,\cdot)$ is defined as 
\begin{align}
    B^{DG}_h(u_h,v) &= (\nabla_h u_h, \nabla_h v)_{L^2(\Omega_h)} - (\boldsymbol{b} u_h,  \nabla_h v)_{L^2(\Omega_h)} + (c u_h, v)_{L^2(\Omega_h)}\nonumber\\
    &-\!\int_{\Gamma_0}\!\llbracket u_h\rrbracket\!\cdot\!\{\!\{ \nabla_h v\}\!\}\dif s\!-\!\int_{\Gamma_0}\!\llbracket v\rrbracket\!\cdot\!\{\!\{ \nabla_h u_h\}\!\}\dif s\!+\!\int_{\Gamma}\!\frac{\eta}{h_e}  \llbracket v\rrbracket\!\cdot\!\llbracket u_h\rrbracket \dif s \nonumber \\
    &+ \int_{\Gamma_0} \boldsymbol{b}\cdot \llbracket v \rrbracket \{\!\{u_h\}\!\}\dif s + \int_{\Gamma} \frac{\abs{\boldsymbol{b}\cdot \boldsymbol{n}_e}}{2}\llbracket v \rrbracket \cdot\llbracket u_h\rrbracket \dif s.\label{bilinear_form} 
\end{align} We call \eqref{formfraca_brokenGrad}--\eqref{bilinear_form} the primal formulation of the method and the bilinear form $B^{DG}_h(\cdot, \cdot)$ the primal form.

\section{Existence and Uniqueness of the Solution}\label{existenceUniq}

In this section, we prove the existence and uniqueness of the numerical solution. We start by analysing some basic properties of the method, namely, we show the boundedness of the bilinear form $B^{DG}_h(\cdot, \cdot)$ defined by \eqref{bilinear_form}. Then, we prove an inf-sup condition in connection with finite-dimensional subspaces, with $\mbox{dim}(\mathcal{W}_h) = \mbox{dim}(\mathcal{V}_h)$.

\subsection{Boundedness}\label{subsection:boundedness}
We prove that the bilinear form $B^{DG}_h(\cdot,\cdot)$ is continuous on
$H^2(\mathcal{T}_h)\times H^2(\mathcal{T}_h)$ endowed with the norm
$\opnorm{\cdot}$. More precisely, there exists a constant $C_b>0$ such that
\begin{equation}\label{boundedness}
    \abs{B^{DG}_h(u_h,v)} \leq C_b \opnorm{u_h}\,\opnorm{v},
    \quad \forall\, u_h,v \in H^2(\mathcal{T}_h).
\end{equation}
In particular, since $\mathcal{W}_h,\mathcal{V}_h \subset H^2(\mathcal{T}_h)$, the estimate applies to all discrete functions. Recalling the definition of $B^{DG}_h(\cdot,\cdot)$, we have
\begin{align}
    \abs{B^{DG}_h(u_h,v)}
    &\leq \abs{(\nabla_h u_h,\nabla_h v)_{L^2(\Omega_h)}}
    + \abs{(\boldsymbol{b}u_h,\nabla_h v)_{L^2(\Omega_h)}}
    + \abs{(c u_h,v)_{L^2(\Omega_h)}} \nonumber\\
    &\quad + \int_{\Gamma_0}\abs{\llbracket u_h\rrbracket\cdot\{\!\{\nabla_h v\}\!\}}\,\dif s
    + \int_{\Gamma_0}\abs{\llbracket v\rrbracket\cdot\{\!\{\nabla_h u_h\}\!\}}\,\dif s \nonumber\\
    &\quad + \int_{\Gamma}\abs{\frac{\eta}{h_e}\llbracket u_h\rrbracket\cdot\llbracket v\rrbracket}\,\dif s
    + \int_{\Gamma_0}\abs{\boldsymbol{b}\cdot\llbracket v\rrbracket\{\!\{u_h\}\!\}}\,\dif s \nonumber\\
    &\quad + \int_{\Gamma}\frac{\abs{\boldsymbol{b}\cdot\boldsymbol{n}_e}}{2}
    \abs{\llbracket u_h\rrbracket\cdot\llbracket v\rrbracket}\,\dif s.
    \label{abs_bilinearForm}
\end{align}

Each term on the right-hand side of \eqref{abs_bilinearForm} is now bounded in
terms of the norm $\opnorm{\cdot}$.

\begin{itemize}
\item For the volume terms, the Cauchy--Schwarz inequality yields
\begin{align}\label{bound_1}
   &\abs{(\nabla_h u_h,\nabla_h v)_{L^2(\Omega_h)}}
   + \abs{(\boldsymbol{b}u_h,\nabla_h v)_{L^2(\Omega_h)}}
   + \abs{(c u_h,v)_{L^2(\Omega_h)}} \nonumber\\
   &\leq 2\hat{C}\sum_{k=1}^K
   \big(\norm{u_h}_{L^2(T^k)}+\norm{\nabla u_h}_{L^2(T^k)}\big)
   \big(\norm{v}_{L^2(T^k)}+\norm{\nabla v}_{L^2(T^k)}\big) \nonumber\\
   &\leq 4\hat{C}\,\opnorm{u_h}\,\opnorm{v}, \qquad \hat{C}=\max\{\norm{\boldsymbol{b}}_{L^\infty(\Omega_h)},\norm{c}_{L^\infty(\Omega_h)}\},
\end{align}
where
\[
\norm{c}_{L^{\infty}(\Omega_h)}\!=\!\mbox{ess} \sup \{\abs{c(\boldsymbol{x})}: \boldsymbol{x} \in \Omega_h\}, \qquad \norm{\boldsymbol{b}}_{L^{\infty}(\Omega_h)}\!=\!\max_{1\leq i \leq 2}\norm{b_i}_{L^{\infty}(\Omega_h)}.
\]

\item We now bound the interface terms.
Using the trace inequality \cite[(2.4)]{Arnold_intPen}, for
$e \subset \partial T^k$, $T^k \in \mathcal{T}_h$,
\begin{equation}\label{trace_2.4}
\norm{w}_{L^2(e)}^2
\leq
C_T^2\big(
h_e^{-1}\norm{w}_{L^2(T^k)}^2
+
h_e\abs{w}_{H^1(T^k)}^2
\big),
\qquad w\in H^1(T^k),
\end{equation}
with $w=\nabla v$, we obtain
\begin{align}\label{bound_04}
\sum_{e\in\mathcal{E}_0}
\int_e
\abs{
\llbracket u_h\rrbracket\cdot
\{\!\{\nabla_h v\}\!\}
}
\,\dif s
&\leq
\abs{u_h}_{\ast}
\left(
\sum_{e\in\mathcal{E}_h}
h_e
\norm{\{\!\{\nabla_h v\}\!\}}_{L^2(e)}^2
\right)^{1/2}
\nonumber\\
&\leq
\abs{u_h}_{\ast}
\left(
\sum_{k=1}^K
3C_T^2
\big(
\abs{v}_{H^1(T^k)}^2
+
h_k^2
\abs{v}_{H^2(T^k)}^2
\big)
\right)^{1/2}
\nonumber\\
&\leq
\sqrt{3} C_T
\abs{u_h}_{\ast}\,
\opnorm{v}.
\end{align}

An analogous argument gives
\begin{align}\label{bound_05}
    \int_{\Gamma_0}
    \abs{\llbracket v\rrbracket\cdot\{\!\{\nabla_h u_h\}\!\}}\,\dif s
    &\leq \sqrt{3}\,C_T \,\abs{v}_\ast \opnorm{u_h}.
\end{align}

\item For the penalty term,
\begin{equation}\label{bound_06}
    \int_{\Gamma}\abs{\frac{\eta}{h_e}\llbracket u_h\rrbracket\cdot\llbracket v\rrbracket}\,\dif s
    \leq \eta\,\abs{u_h}_\ast\,\abs{v}_\ast.
\end{equation}

\item Proceeding similarly for the convection-related jump term, we find
\begin{equation}\label{bound_07}
    \int_{\Gamma_0}\abs{\boldsymbol{b}\cdot\llbracket v\rrbracket\{\!\{u_h\}\!\}}\,\dif s
    \leq \sqrt{6}\,C_T\norm{\boldsymbol{b}}_{L^\infty(\Omega_h)}
    \abs{v}_\ast\,\opnorm{u_h}.
\end{equation}

\item Finally,
\begin{equation}\label{bound_08}
    \int_{\Gamma}\frac{\abs{\boldsymbol{b}\cdot\boldsymbol{n}_e}}{2}
    \abs{\llbracket u_h\rrbracket\cdot\llbracket v\rrbracket}\,\dif s
    \leq \abs{u_h}_b\,\abs{v}_b.
\end{equation}
\end{itemize}
Collecting \eqref{bound_1}, \eqref{bound_04}--\eqref{bound_08}, we conclude that
\eqref{boundedness} holds with
\[
C_b = 4\hat{C}
+ 2\sqrt{3}C_T 
+ \eta
+ \sqrt{6}C_T\norm{\boldsymbol{b}}_{L^\infty(\Omega_h)}
+ 1.
\]

\begin{remark}
The continuity of the bilinear form is established on $H^2(\mathcal{T}_h) \times H^2(\mathcal{T}_h)$, following the framework of \cite{Arnold_intPen}. From a functional-analytic perspective, this requirement could be relaxed to $H^{3/2+\epsilon}(\mathcal{T}_h)$ for any $\epsilon > 0$, since the trace theorem ensures $\nabla v \cdot \boldsymbol{n} \in L^2(\partial T^k)$ for $v \in H^{3/2+\epsilon}(T^k)$, which suffices for the interface integrals to be well defined. We retain the $H^2(\mathcal{T}_h)$ framework for simplicity and consistency with the $H^2(\Omega)$ regularity of the exact solution.
\end{remark}

\subsection{Inf-sup condition}

The following theorem establishes an inf-sup condition for the DG--ROD formulation \eqref{formfraca_brokenGrad}--\eqref{bilinear_form}. This result implies the well-posedness of the discrete problem \eqref{formfraca_brokenGrad}--\eqref{bilinear_form} based on employing the Banach–Ne\v{c}as–Babu\v{s}ka (BNB) Theorem (see \cite{Ern_Guermond}, Theorem $2.6$).

\begin{theorem}[Inf-sup condition]\label{ProposicaoDesAlpha_v2}
Consider the bilinear form $B^{DG}_h(\cdot, \cdot)$ defined in \eqref{bilinear_form}. Provided $h < h_0$ and $ \eta > \eta_0$, with $h_0$ and $\eta_0$ explicitly defined in the proof of the theorem, there hold:
\begin{enumerate}
    \item there exists a constant $\alpha >0$ such that \begin{equation}\label{DesigAlpha_v2}
        \inf_{w \in \mathcal{W}_h \setminus\{0\}} \sup_{v \in \mathcal{V}_h \setminus \{0\}} \frac{B^{DG}_h(w,v)}{\opnorm{w}\opnorm{v}} \geq \alpha; \tag{BNB1}
        \end{equation}
    \item for any $v \in \mathcal{V}_h$ such that \begin{equation}\label{bnb2}
     B_h^{DG}(w,v) = 0, \quad \mbox{for all $w \in \mathcal{W}_h$}, \tag{BNB2}
    \end{equation} then $v=0$.
\end{enumerate}
\end{theorem}

\begin{proof}
We first show \eqref{DesigAlpha_v2}. Let $w \in \mathcal{W}_h$ be arbitrary. We will construct a test function $v \in \mathcal{V}_h$ such that
\begin{equation}\label{goal_proof}
    B^{DG}_h(w,v) \geq C_\alpha \opnorm{w}^2
    \quad \text{and} \quad
    \opnorm{v} \leq C_v \opnorm{w},
\end{equation}
with constants $C_\alpha, C_v >0$ independent of $h$.
The inf-sup condition \eqref{DesigAlpha_v2} then follows with
$\alpha = C_\alpha/C_v$.

The function $v$ is defined by modifying $w$ at the nodes lying on
$\partial\Omega_h$ that are not mesh vertices.
More precisely, we define the residual
\[
r(w) = w - v = \bigoplus_{k \in I^B} r^k(w),
\]
where, for each boundary element $T^k$,
\[
r^k(w) = \sum_{i \in I^{kB}} w(\boldsymbol{x}_i^k)\,\ell_i^k.
\]
By construction, $r(w)$ vanishes on all elements that do not touch
$\partial\Omega_h$.

Taking advantage of the relation between $w$ and $v$, we may write
\begin{align} \label{Eq1_v2}
 &B^{DG}_h(w,v)=\sum_{k=1}^K \left((\nabla w,\nabla w)_{L^2(T^k)}-(\boldsymbol{b} w,\nabla w)_{L^2(T^k)}+(cw,w)_{L^2(T^k)}\right) \nonumber \\
 &+\!\int_{\Gamma}\frac{\eta}{h_e} \llbracket w\rrbracket\!\cdot\!\llbracket w \rrbracket \dif s+ \int_{\Gamma}\frac{\abs{\boldsymbol{b} \cdot \boldsymbol{n}_e}}{2} \llbracket w\rrbracket\cdot\llbracket w \rrbracket \dif s  -2\int_{\Gamma_0}\llbracket w\rrbracket\cdot\{\!\{\nabla_h w\}\!\} \dif s  \nonumber \\
  &+\int_{\Gamma_0}\boldsymbol{b} \cdot \llbracket w\rrbracket \{\!\{ w\}\!\} \dif s -\sum_{k \in I^B}\left((\nabla w, \nabla r^k(w))_{L^2(T^k)}-(\boldsymbol{b} w,\!\nabla r^k(w))_{L^2(T^k)}\right)\nonumber \\
  &-\sum_{k \in I^B}(cw,r^k(w))_{L^2(T^k)} -\!\sum_{e \in \mathcal{E}^B_h} \int_{e}\!\frac{\eta}{h_e}\!\llbracket r(w)\rrbracket\!\cdot\!\llbracket w\!\rrbracket \dif s\!-\!\int_{\mathcal{E}^B_h}\!\frac{\abs{\boldsymbol{b}\!\cdot \!\boldsymbol{n}_e}}{2} \llbracket r(w)\rrbracket\!\cdot\!\llbracket w\!\rrbracket \dif s \nonumber \\
  &+\sum_{e \in \mathcal{E}_0^B}\!\int_{e}\llbracket w\rrbracket\!\cdot\!\{\!\{\nabla_h r(w)\}\!\} \dif s+\sum_{e \in \mathcal{E}_0^B}\int_{e}\llbracket r(w) \rrbracket\cdot\{\!\{ \nabla_h w\}\!\}\dif s \nonumber \\
  &-\sum_{e \in \mathcal{E}_0^B}\int_{e} \boldsymbol{b}\cdot\llbracket r(w) \rrbracket\{\!\{w\}\!\}\dif s, 
 \end{align} where $\mathcal{E}_h^B = \cup_{k \in I^B} \partial T^k$ and $\mathcal{E}_0^B = \mathcal{E}_h^B  \cap \Gamma_0$. 

\medskip
We now estimate lower bounds term by term individually.

\begin{itemize}
\item 
Using the definition of the seminorms \eqref{semiNormDG}, we directly obtain
\begin{align}
&\sum_{k=1}^K (\nabla w, \nabla w)_{L^2(T^k)}
= |w|^2_{H^1(\mathcal{T}_h)}, \label{upperBound_gradw} \\
&\sum_{e \in \mathcal{E}_h}\int_{e}
\frac{\eta}{h_e}\llbracket w \rrbracket\cdot\llbracket w \rrbracket\,\dif s
= \eta |w|_*^2, \label{upperBound_w*} \\
&\sum_{e \in \mathcal{E}_h}\int_{e}
\frac{|\boldsymbol{b}\cdot\boldsymbol{n}_e|}{2}
\llbracket w \rrbracket\cdot\llbracket w \rrbracket\,\dif s
= |w|_b^2. \label{upperBound_wb}
\end{align}
These terms provide the positive contributions that control the DG norm
and will be used to absorb the remaining terms in the inf-sup analysis.

\item Integrating by parts, the second term in \eqref{Eq1_v2} can be rewritten as
\begin{equation}\label{int_parts_b_gradw}
    -\sum_{k=1}^K (\boldsymbol{b} w, \nabla w)_{L^2(T^k)}
    = -\sum_{k=1}^K \int_{\partial T^k} \frac{w^2}{2} \, \boldsymbol{b} \cdot \boldsymbol{n} \, \mathrm{d}s
      + \sum_{k=1}^K \int_{T^k} \frac{w^2}{2} \, \nabla \cdot \boldsymbol{b} \, \mathrm{d}\boldsymbol{x}.
\end{equation}

The second term on the right-hand side of \eqref{int_parts_b_gradw} will be bounded simultaneously with the reactive term, following the condition stated in \eqref{cx_bx_cond}. 
We now estimate a lower bound for the first term. Using the average and jump operators (see \cite{metodoDG}, Lemma 7.9), we have
\begin{align}\label{formula_jumpsAverage}
    \sum_{T^k \in \mathcal{T}_h} \int_{\partial T^k} w^2 \, \boldsymbol{n} \cdot \boldsymbol{b} \, \mathrm{d}s
    &= \int_{\Gamma} \llbracket w^2 \rrbracket \cdot \{\!\{ \boldsymbol{b} \}\!\} \, \mathrm{d}s
      + \int_{\Gamma_0} \{\!\{ w^2 \}\!\} \llbracket \boldsymbol{b} \rrbracket \, \mathrm{d}s.
\end{align}

Exploiting the continuity of $\boldsymbol{b}$ across interfaces, we get
\begin{equation}\label{int_parts_b_gradw_eq1}
  -\sum_{k=1}^K \int_{\partial T^k} \frac{w^2}{2} \, \boldsymbol{b} \cdot \boldsymbol{n} \, \mathrm{d}s
  = - \frac{1}{2} \int_{\Gamma} \boldsymbol{b} \cdot \llbracket w^2 \rrbracket \, \mathrm{d}s.
\end{equation}

Using the standard relation between jumps and averages on interior edges \cite{Ern_Pietro}, 
$
    \frac{1}{2} \llbracket w^2 \rrbracket = \llbracket w \rrbracket \cdot \{\!\{ w \}\!\},
$
we can rewrite \eqref{int_parts_b_gradw_eq1} as
\begin{align}\label{int_parts_b_gradw_eq1_separ_aux}
    -\sum_{k=1}^K \int_{\partial T^k} \frac{w^2}{2} \, \boldsymbol{b}\!\cdot\!\boldsymbol{n} \, \mathrm{d}s
    &= -\int_{\Gamma_0} \boldsymbol{b} \cdot \llbracket w \rrbracket \{\!\{ w \}\!\} \, \mathrm{d}s\!-\!\frac{1}{2} \int_{\partial \Omega_h} w^2 \, \boldsymbol{b} \cdot \boldsymbol{n} \, \mathrm{d}s.
\end{align}

Adding the seventh term in \eqref{Eq1_v2} gives
\begin{equation}\label{int_parts_b_gradw_eq1_separ}
    -\sum_{k=1}^K \int_{\partial T^k} \frac{w^2}{2} \, \boldsymbol{b} \cdot \boldsymbol{n} \, \mathrm{d}s
    + \int_{\Gamma_0} \boldsymbol{b} \cdot \llbracket w \rrbracket \{\!\{ w \}\!\} \, \mathrm{d}s
    \geq - \frac{1}{2} \int_{\partial \Omega_h} |w^2 \, \boldsymbol{b} \cdot \boldsymbol{n}| \, \mathrm{d}s.
\end{equation}

We now bound the term on the right-hand side. Let $O$ be a vertex of the edge $e^{kB}$, $k \in I^B$, and $M \in e^{kB}$. 
Since $w(O) = 0$, the mean value theorem yields
\begin{equation}
    |w(M)| \leq h_k \, \|\nabla w\|_{L^\infty(T^k)}, \quad \forall M \in e^{kB}.
\end{equation}

Using this, Lemma 3.1 in \cite{ruas2019OptimalLA}, and the fact that $\mathrm{length}(e^{kB}) \leq h_k$, we obtain
\begin{align}\label{w_square_boundary_final}
    \int_{\partial \Omega_h} |w^2 \, \boldsymbol{b}\cdot \boldsymbol{n}| \, \mathrm{d}s
    &\leq C_J^2 \, \|\boldsymbol{b}\|_{L^\infty(\Omega_h)} \, h \, |w|^2_{H^1(\mathcal{T}_h)}.
\end{align}

Substituting \eqref{w_square_boundary_final} into \eqref{int_parts_b_gradw_eq1_separ} gives
\begin{equation}\label{secondTerm_1}
    -\sum_{k=1}^K \int_{\partial T^k} \frac{w^2}{2} \, \boldsymbol{b}\cdot \boldsymbol{n} \, \mathrm{d}s
    + \int_{\Gamma_0} \boldsymbol{b} \cdot \llbracket w \rrbracket \{\!\{ w \}\!\} \, \mathrm{d}s
    \geq - \frac{1}{2} C_J^2 \, \|\boldsymbol{b}\|_{L^\infty(\Omega_h)} \, h \, |w|^2_{H^1(\mathcal{T}_h)}.
\end{equation}

\item For the sixth term in \eqref{Eq1_v2} (edge term involving $\llbracket w \rrbracket$ and $\nabla w$), using \eqref{bound_04} and \eqref{Des_H2_H1_lemma4.5.3}, together with a trace inequality, the Cauchy–Schwarz inequality, and Young's inequality, we obtain, for any $\epsilon_1>0$, the following bound 
\begin{equation}\label{upperBound_w_graw_E}
    -2 \int_{\Gamma_0} \llbracket w \rrbracket \cdot \{\nabla_h w\} \dif s 
    \geq -\sqrt{3}\, C_T \sqrt{1+C^2} \left( \epsilon_1\, |w|_{H^1(\mathcal{T}_h)}^2 + \frac{|w|_*^2}{\epsilon_1} \right).
\end{equation}

\item To bound the terms containing $r(w)$ (eighth to tenth terms), we first note that
\begin{align}
\norm{r^k(w)}_{L^2(T^k)} &\leq \sum_{i \in I^{kB}} \abs{w(\boldsymbol{x}^{k}_i)} \norm{\ell_i^k}_{L^2(T^k)}, \label{maj_residuo}\\
\norm{\nabla r^k(w)}_{L^2(T^k)} &\leq \sum_{i \in I^{kB}} \abs{w(\boldsymbol{x}^{k}_i)} \norm{\nabla \ell_i^k}_{L^2(T^k)}. \label{maj_grad_residuo}
\end{align}

From standard results, there exist mesh-independent constants $C_1$ and $C_2$ such that
\[
\norm{\ell_i^k}_{L^2(T^k)} \leq C_1 h_k, \quad 
\norm{\nabla \ell_i^k}_{L^2(T^k)} \leq C_2.
\]

Recall that $w(P^k_j) = 0$, for $j=1, \ldots, N-1$, where $P^k_j$ is the point on $\partial \Omega$ associated with the node $x_j^k$, for $j \in I^{kB}$. Then, considering a Taylor expansion about $P^k_j$, we get
\begin{align*}
    |w(\boldsymbol{x}^{k}_j)| \leq \textrm{length}(\overline{P^k_j\boldsymbol{x}^{k}_j}) \norm{\nabla w}_{L^{\infty}(T^k \cup \Delta_k)}.
\end{align*} Hence, following \eqref{distance_boundaries}, which establishes the $\mathcal{O}(h^2)$ distance between the physical and computational boundaries, for a suitable constant $C_{\partial \Omega}$ depending only on $\Omega$, we have 
\begin{align*}
    |w(\boldsymbol{x}^{k}_j)| \leq C_{\partial \Omega}h_k^2  \norm{\nabla w}_{L^{\infty}(T^k \cup \Delta_k)}.
\end{align*}
Since $w$ is a polynomial in $T^k$, according to Lemma 3.1 from \cite{ruas2019OptimalLA}, there exist mesh-independent constants $C_{\infty}$ and $C_J$ such that
\begin{align*}
    \norm{\nabla w}_{L^{\infty}(T^k \cup \Delta_k)} \leq C_{\infty} \ \norm{\nabla w}_{L^{\infty}(T^k)}\leq C_{\infty} C_J \frac{1}{h_k}  \| \nabla w\|_{L^{2}(T^k)}.
\end{align*}
Thus, using the above inequalities, we may write
\begin{align*}
    |w(\boldsymbol{x}^{k}_j)| \leq C_{\partial \Omega} C_{\infty} C_J h_k  \|\nabla w\|_{L^{2}(T^k)}.
\end{align*}


Hence, we obtain the following estimates:
\begin{equation}\label{norm_Residuo}
\norm{r^k(w)}_{L^2(T^k)} \leq \Tilde{C}_1 h^2 \norm{\nabla w}_{L^2(T^k)}, \quad
\Tilde{C}_1 = (N-1) C_1 C_{\partial \Omega_h} C_\infty C_J,
\end{equation}
\begin{equation}\label{norm_gradResiduo}
\norm{\nabla r^k(w)}_{L^2(T^k)} \leq \Tilde{C}_2 h \norm{\nabla w}_{L^2(T^k)}, \quad
\Tilde{C}_2 = (N-1) C_2 C_{\partial \Omega_h} C_\infty C_J.
\end{equation}

Using \eqref{norm_gradResiduo}, we bound the eighth term in \eqref{Eq1_v2}:
\begin{align}\label{upperBound_gradw_grar_Eb}
-\sum_{k \in I^B} (\nabla w, \nabla r^k(w))_{L^2(T^k)}
&\geq -\sum_{k \in I^B} \norm{\nabla w}_{L^2(T^k)} \norm{\nabla r^k(w)}_{L^2(T^k)} \nonumber \\
&\geq - \Tilde{C}_2 h \abs{w}^2_{H^1(\mathcal{T}_h)}.
\end{align}

Similarly, for the ninth term and any $\epsilon_2 > 0$,
\begin{align}\label{upperBound_bgradw_r_Eb}
\sum_{k \in I^B} (\boldsymbol{b} w, \nabla &r^k(w))_{L^2(T^k)}
\geq -\sqrt{2} \norm{\boldsymbol{b}}_{L^\infty(\Omega_h)} \sum_{k \in I^B} \norm{w}_{L^2(T^k)} \norm{\nabla r^k(w)}_{L^2(T^k)} \nonumber \\
&\geq -\frac{\sqrt{2} \norm{\boldsymbol{b}}_{L^\infty(\Omega_h)} \Tilde{C}_2}{2} h
\left( \epsilon_2 \norm{w}^2_{L^2(\Omega_h)} + \frac{\abs{w}^2_{H^1(\mathcal{T}_h)}}{\epsilon_2} \right).
\end{align}

Finally, for the tenth term and any $\epsilon_3 > 0$, using \eqref{norm_Residuo}:
\begin{align}\label{upperBound_w_r_Eb}
-\sum_{k \in I^B} (c w, r^k(w))_{L^2(T^k)}
&\geq - \sum_{k \in I^B} \norm{c w}_{L^2(T^k)} \norm{r^k(w)}_{L^2(T^k)} \nonumber \\
&\geq - \Tilde{C}_1 h^2 \norm{c w}_{L^2(\Omega_h)} \abs{w}_{H^1(\mathcal{T}_h)} \nonumber \\
&\geq -\frac{\Tilde{C}_1}{2} h^2 \left(\!\epsilon_3 \norm{c w}^2_{L^2(\Omega_h)} + \frac{\abs{w}^2_{H^1(\mathcal{T}_h)}}{\epsilon_3}\!\right).
\end{align}

\item We now estimate the DG jump terms of $r(w)$, corresponding to the eleventh through last term in \eqref{Eq1_v2}.

For the eleventh term, using $\epsilon_4>0$, we have
\begin{align}\label{des_coercivity}
-\sum_{e \in \mathcal{E}^B_h} \int_{e} \frac{\eta}{h_e} \llbracket r(w) \rrbracket \cdot \llbracket &w \rrbracket \dif s
\geq -\sum_{e \in \mathcal{E}^B_h} \frac{\eta}{h_e} \abs{\llbracket r(w) \rrbracket \cdot \llbracket w \rrbracket} \dif s \nonumber \\
&\geq -\eta \sum_{e \in \mathcal{E}^B_h} \norm{ \frac{\llbracket r(w) \rrbracket}{h_e^{1/2}}}_{L^2(e)} \norm{ \frac{\llbracket w \rrbracket}{h_e^{1/2}}}_{L^2(e)} \nonumber \\
&\geq -\eta \abs{r(w)}_{\ast} \abs{w}_{\ast} \geq -\frac{\eta}{2} \left(\!\epsilon_4 \abs{r(w)}^2_{\ast}\!+\!\frac{\abs{w}^2_{\ast}}{\epsilon_4}\!\right).
\end{align}

Using the trace inequality \eqref{trace_2.4} for $r^k(w)$ and \eqref{meshIneq_hk_he}, we get
\begin{align*}
\abs{r(w)}^2_{\ast} 
&\leq \sum_{k=1}^K \sum_{e \in \mathcal{E}_h \cap \partial T^k} 2 h_e^{-1} \norm{r^k(w)}^2_{L^2(e)} \\
&\leq \sum_{k=1}^K \sum_{e \in \mathcal{E}_h \cap \partial T^k} 2 C_T^2 h_e^{-2} \Big(\norm{r^k(w)}^2_{L^2(T^k)} + h_e^2 \norm{\nabla r^k(w)}^2_{L^2(T^k)} \Big) \\
&\leq \sum_{k=1}^K C^2_{\ast} h^2 \abs{w}^2_{H^1(T^k)}, \quad
C^2_{\ast} = \frac{6 C_T^2 (\Tilde{C}_1^2 + \Tilde{C}_2^2)}{\mu^2}.
\end{align*}

Thus, \eqref{des_coercivity} becomes
\begin{equation}\label{upperBound_rw_graw_Eb}
-\sum_{e \in \mathcal{E}^B_h} \int_{e} \frac{\eta}{h_e} \llbracket r(w) \rrbracket \cdot \llbracket w \rrbracket \dif s
\geq -\frac{\eta}{2} \left( \epsilon_4 C^2_{\ast} h^2 \abs{w}^2_{H^1(\mathcal{T}_h)} + \frac{\abs{w}^2_{\ast}}{\epsilon_4} \right).
\end{equation}

For the twelfth term, using $\epsilon_5>0$, we have
\begin{align}\label{aux_bound_12}
-&\sum_{e \in \mathcal{E}_h^B} \int_e \frac{\abs{\boldsymbol{b}\cdot\boldsymbol{n}_e}}{2} \llbracket w \rrbracket \cdot \llbracket r(w) \rrbracket \dif s \nonumber \\
&\geq - \abs{w}_b \abs{r(w)}_b \geq -\frac{1}{2} \left( \epsilon_5 \abs{w}^2_b + \frac{\abs{r(w)}^2_b}{\epsilon_5} \right).
\end{align}

Using $h_e<1$, we bound $\abs{r(w)}_b$ by $\abs{r(w)}_{\ast}$:
\begin{align}\label{relationship_seminorms}
\abs{r(w)}^2_b &\leq \frac{\norm{\boldsymbol{b}}_{L^\infty(\Omega_h)}}{2} \sum_{e \in \mathcal{E}_h} h_e^{-1} \llbracket r(w) \rrbracket^2 \dif s = \frac{\norm{\boldsymbol{b}}_{L^\infty(\Omega_h)}}{2} \abs{r(w)}^2_{\ast} \nonumber \\
&\leq \frac{\norm{\boldsymbol{b}}_{L^\infty(\Omega_h)}}{2} C_*^2 h^2 \abs{w}^2_{H^1(\mathcal{T}_h)}.
\end{align}

Substituting \eqref{relationship_seminorms} into \eqref{aux_bound_12} gives
\begin{equation}\label{bound_12}
-\sum_{e \in \mathcal{E}_h^B} \int_e \frac{\abs{\boldsymbol{b}\!\cdot\!\boldsymbol{n}_e}}{2} \llbracket w \rrbracket\!\cdot\!\llbracket r(w) \rrbracket \dif s
\geq -\frac{\epsilon_5}{2} \abs{w}^2_b\!- \frac{\norm{\boldsymbol{b}}_{L^\infty(\Omega_h)} C_*^2 h^2}{4 \epsilon_5} \abs{w}^2_{H^1(\mathcal{T}_h)}.
\end{equation}

For the thirteenth term, using $\epsilon_6>0$, we get
\begin{align}\label{upperBound_w_gradr_Eb}
\sum_{e \in \mathcal{E}_0^B} \int_e \llbracket w \rrbracket \cdot \{\!\{\nabla_h r(w)\}\!\} \dif s
\geq -\frac{\epsilon_6}{2} \abs{w}^2_{\ast} - \frac{3}{2 \epsilon_6} C_T^2 (1+C^2) \Tilde{C}_2^2 h^2 \abs{w}^2_{H^1(\mathcal{T}_h)}.
\end{align}

For the fourteenth term, similarly,
\begin{align}\label{upperBound_rw_gradw_Eb}
\sum_{e \in \mathcal{E}_0^B} \int_e \llbracket r(w) \rrbracket \cdot \{\!\{\nabla_h w\}\!\} \dif s
\geq -\sqrt{3} C_T \sqrt{1+C^2} C_* h \abs{w}^2_{H^1(\mathcal{T}_h)}.
\end{align}

Finally, for the last term, using $\epsilon_7>0$, we have
\begin{align}\label{bound_lastTerm}
-&\sum_{e \in \mathcal{E}_0^B} \int_e \boldsymbol{b}\cdot\llbracket r(w)\rrbracket \{\!\{ w\}\!\} \dif s \nonumber \\
&\geq -\frac{\sqrt{6} C_T \norm{\boldsymbol{b}}_{L^\infty(\Omega_h)}}{2} \left(\!\left(\!\epsilon_7\!+\!\frac{C_*^2}{\epsilon_7}\!\right) h^2 \abs{w}^2_{H^1(\mathcal{T}_h)} + \epsilon_7 \norm{w}^2_{L^2(\Omega_h)}\!\right).
\end{align}

\end{itemize}

To organise the terms involving $\norm{w}_{L^2(\Omega_h)}$, we first consider those containing 
the function $c$ in \eqref{Eq1_v2}, namely the third term and the contribution from bounding 
the tenth term (see \eqref{upperBound_w_r_Eb}). Next, we account for terms containing the 
function $\boldsymbol{b}$, which include the last term in \eqref{int_parts_b_gradw}, the first term 
arising from the bound of the ninth term (see \eqref{upperBound_bgradw_r_Eb}), and the last term 
from bounding the fifteenth term (see \eqref{bound_lastTerm}).  

Taking into account \eqref{cx_bx_cond}, we obtain
\begin{align}\label{bound_cx}
&\sum_{k=1}^K (c w, w)_{L^2(T^k)} 
- \frac{\Tilde{C}_1}{2} h^2 \epsilon_3 \norm{c w}^2_{L^2(\Omega_h)}
+ \sum_{k=1}^K \int_{T^k} \frac{w^2}{2} \nabla \cdot \boldsymbol{b} \, d\boldsymbol{x} \nonumber \\
&\quad - \frac{\sqrt{2}\norm{\boldsymbol{b}}_{L^{\infty}(\Omega_h)}}{2} 
\Tilde{C}_2 \epsilon_2 h \norm{w}^2_{L^2(\Omega_h)}
- \frac{\sqrt{6}\norm{\boldsymbol{b}}_{L^{\infty}(\Omega_h)}}{2} C_T \epsilon_7 
\norm{w}^2_{L^2(\Omega_h)} \nonumber \\
&\geq - \widehat{C}_1 \norm{w}^2_{L^2(\Omega_h)}, 
\end{align}
where
\begin{equation}\label{Cl}
\widehat{C}_1 = 
 \frac{\Tilde{C}_1}{2} h^2 \epsilon_3 \norm{c}^2_{L^{\infty}(\Omega_h)}
 + \frac{\sqrt{2}\norm{\boldsymbol{b}}_{L^{\infty}(\Omega_h)}}{2} \Tilde{C}_2 \epsilon_2 h
+ \frac{\sqrt{6}\norm{\boldsymbol{b}}_{L^{\infty}(\Omega_h)}}{2} C_T \epsilon_7.
\end{equation}

Using \eqref{DG_norm_simpl}, we may write
\begin{equation}\label{desig_normDG_1}
\opnorm{w}^2 \leq C^2_{aux} \Big( \norm{w}^2_{L^2(\Omega_h)} 
+ \abs{w}^2_{H^1(\mathcal{T}_h)} 
+ \abs{w}^2_{\ast} 
+ \abs{w}^2_b \Big), \quad C^2_{aux} = 1 + C^2.
\end{equation}

Applying a broken Poincaré–Friedrichs inequality valid for $w \in H^1(\mathcal{T}_h)$ 
(see \cite{Arnold_intPen}, Lemma 2.1), we further obtain
\begin{align}\label{desig_normDG_2}
\opnorm{w}^2 
&\leq C^2_{aux} \Big( C_P \big( \abs{w}^2_{H^1(\mathcal{T}_h)} + \abs{w}^2_{\ast} \big)
+ \abs{w}^2_{H^1(\mathcal{T}_h)} + \abs{w}^2_{\ast} + \abs{w}^2_b \Big) \nonumber \\
&\leq C^2_{aux} (1 + C_P) \Big( \abs{w}^2_{H^1(\mathcal{T}_h)} + \abs{w}^2_{\ast} + \abs{w}^2_b \Big).
\end{align}

Combining the bounds for each term in \eqref{Eq1_v2}, namely 
\eqref{upperBound_gradw}--\eqref{upperBound_wb}, 
\eqref{secondTerm_1}, \eqref{upperBound_w_graw_E}, 
\eqref{upperBound_gradw_grar_Eb}--\eqref{upperBound_w_r_Eb}, 
\eqref{upperBound_rw_graw_Eb}, \eqref{bound_12}--\eqref{bound_cx}, we may write
\begin{equation}\label{infsupa}
B^{DG}_h(w,v) \geq C_{\alpha} \opnorm{w}^2,
\end{equation}
where
\[
C_{\alpha} = \frac{\min\{ \widehat{C}_2 - \widehat{C}_1 C_P, \; \widehat{C}_3 - \widehat{C}_1 C_P, \; \widehat{C}_4 \}}{C^2_{aux} (1 + C_P)}.
\]

The constants are defined as follows: $\widehat{C}_1$ is given in \eqref{Cl}, 
\begin{align*}
\widehat{C}_2 &= \big(1 - \sqrt{3} C_T \sqrt{1+C^2} \, \epsilon_1 \big) \\
& \quad - h \Big( \Tilde{C}_2 + \frac{\norm{\boldsymbol{b}}_{L^{\infty}(\Omega_h)} C^2_J}{2} 
+ \frac{\sqrt{2} \Tilde{C}_2 \norm{\boldsymbol{b}}_{L^{\infty}(\Omega_h)}}{2 \epsilon_2} 
+ \sqrt{3} C_T \sqrt{1 + C^2} C_* \Big) \\
&\quad - h^2 \left( \frac{\Tilde{C}_1}{2 \epsilon_3} + \frac{3}{2 \epsilon_6} C^2_T (1+C^2) \Tilde{C}^2_2
+ \eta \epsilon_4 \frac{C^2_*}{2} 
+ \frac{C_*^2}{4 \epsilon_5} \norm{\boldsymbol{b}}_{L^{\infty}(\Omega_h)} \right.  \\
&\qquad \left.\qquad + \frac{\sqrt{6} \norm{\boldsymbol{b}}_{L^{\infty}(\Omega_h)} C_T}{2} \left(\frac{C_*^2}{\epsilon_7} + \epsilon_7 \right) \right)
\end{align*}
\[
\widehat{C}_3 = \eta\left(1-\frac{1}{2 \epsilon_4}\right) - \frac{\sqrt{3} C_T \sqrt{1+C^2}}{\epsilon_1} - \frac{\epsilon_6}{2}, \quad 
\widehat{C}_4 = 1 - \frac{\epsilon_5}{2}.
\]

To conclude the first inequality in \eqref{goal_proof} from \eqref{infsupa}, we must guarantee positivity of these constants. We impose the following conditions: (i)  for $\widehat{C}_2 - \widehat{C}_1 C_P > 0$, choose 
    $\epsilon_7 < 1 / (\sqrt{6} \norm{\boldsymbol{b}}_{L^\infty(\Omega_h)} C_T C_P)$ 
    (if $\norm{\boldsymbol{b}}_{L^\infty(\Omega_h)} > 0$), $\epsilon_1 < 1/(2 \sqrt{3} C_T \sqrt{1+C^2})$, 
    and $h < h_1$, where $h_1$ is the positive root of $\widehat{C}_2 - \widehat{C}_1 C_P = 0$; (ii) for $\widehat{C}_3 - \widehat{C}_1 C_P > 0$, take $\epsilon_4 > 1/2$, 
\begin{align}
    \eta & >\!\eta_0 =\left(\!\frac{1}{1-\frac{1}{2\epsilon_4}}\!\right) \left(\!\frac{\epsilon_6}{2}\!+\!\frac{ \sqrt{3}C_T\sqrt{1\!+\!C^2}}{\epsilon_1}+\! C_P\frac{\sqrt{6}\norm{\boldsymbol{b}}_{L^{\infty}(\Omega_h)} {C}_T \epsilon_7}{2}\!\right) \label{eta0}
\end{align}
    and $h < h_2$, where $h_2$ is the positive root of $\widehat{C}_3 - \widehat{C}_1 C_P = 0$; (iii) for $\widehat{C}_4 > 0$, take $\epsilon_5 < 2$.
Note that, in particular, $\widehat{C}_2 - \widehat{C}_1 C_P$ and $\widehat{C}_3 - \widehat{C}_1 C_P$ admit positive lower bounds on $(0,h_1)$ and $(0,h_2)$, respectively, that are independent of $h$. Therefore, there exist constants
$\delta_1,\delta_2>0$, independent of $h$, such that
$$
\widehat C_2-\widehat C_1C_P \ge \delta_1,
\qquad
\widehat C_3-\widehat C_1C_P \ge \delta_2,
\qquad \forall\,0< h < h_0 = \min\{h_1, h_2\}.
$$
Consequently,
\[
C_\alpha=
\frac{\min\{\widehat C_2-\widehat C_1C_P,\;\widehat C_3-\widehat C_1C_P,\;\widehat C_4\}}
{C^2_{\mathrm{aux}}(1+C_P)}
\ge
\frac{\min\{\delta_1,\delta_2,\widehat C_4\}}
{C^2_{\mathrm{aux}}(1+C_P)},
\]
where the right-hand side is a positive constant independent of $h$.
Then, for $h$ sufficiently small, $h < h_0$, the bound \eqref{infsupa} holds with $C_\alpha > 0$ and so we conclude the first inequality in \eqref{goal_proof}.

\medskip
To conclude the proof of the inf-sup condition, it remains to prove the second inequality in \eqref{goal_proof}. Using \eqref{norm_Residuo}, we estimate
\begin{align*}
\norm{v}_{L^2(T^k)} &= \|v-w+w\|_{L^2(T^k)}\leq  \|r^k(w)\|_{L^2(T^k)}  + \|w\|_{L^2(T^k)} \\ 
&\leq \Tilde{C}_1 h^2 \norm{\nabla w}_{L^2(T^k)} +\norm{w}_{L^2(T^k)} 
\leq \sqrt{2} \, (1 + \Tilde{C}_1 h^2) \norm{w}_{H^1(T^k)}.    
\end{align*}

Similarly, from \eqref{norm_gradResiduo} we have
\[
\norm{\nabla v}_{L^2(T^k)} \leq (1 + \Tilde{C}_2 h) \norm{\nabla w}_{L^2(T^k)}.
\]
For the jump-related seminorms, we obtain
\[
\abs{v}^2_{\ast} \leq 2 C^2_\ast h^2 \abs{w}^2_{H^1(\mathcal{T}_h)} + 2 \abs{w}^2_{\ast}, \quad
\abs{v}^2_b \leq 2 \norm{\boldsymbol{b}}_{L^\infty(\Omega_h)} C^2_\ast h^2 \abs{w}^2_{H^1(\mathcal{T}_h)} + 2 \abs{w}^2_b.
\]
Combining the above, the DG norm of $v$ satisfies
\[
\opnorm{v}^2 \leq \widehat{C}^2_{aux} \Big( \norm{v}^2_{L^2(\Omega_h)} + \abs{v}^2_{H^1(\mathcal{T}_h)} + \abs{v}^2_{\ast} + \abs{v}^2_b \Big) \leq C_v^2 \opnorm{w}^2,
\]
with
\[
C_v^2 = \widehat{C}^2_{aux} \max \Big\{ 2, \; 2(1 + \Tilde{C}_1 h_0^2)^2 + (1 + \Tilde{C}_2 h_0)^2 + 2 C^2_\ast h_0^2 + 2 \norm{\boldsymbol{b}}_{L^\infty(\Omega_h)} C^2_\ast h_0^2 \Big\}.
\]
which concludes the proof of \eqref{DesigAlpha_v2}.

Next, we observe that since we are in a finite-dimensional setting with trial and test spaces having the same dimension, \eqref{bnb2} is a consequence of \eqref{DesigAlpha_v2} (see \cite{Ern_Guermond}, Proposition 2.21) and the proof
is complete.
\end{proof}

Thus, thanks to BNB Theorem, provided $h$ sufficiently small and $\eta$ sufficiently large, the fact that the inf-sup condition \eqref{DesigAlpha_v2} holds implies that \eqref{formfraca_brokenGrad}--\eqref{bilinear_form} is uniquely solvable.

\begin{remark}
In the inf-sup condition of Theorem~\ref{ProposicaoDesAlpha_v2}, the parameter $\eta$ must satisfy $\eta > \eta_0$, where $\eta_0$ is defined in \eqref{eta0}. Choosing
\[
\epsilon_1=\frac{1}{3\sqrt{3}\,C_T\sqrt{1+C^2}}, 
\quad 
\epsilon_4=\frac34,
\quad 
\epsilon_6=1,
\quad 
\epsilon_7=\frac{1}{2\sqrt{6}\,\|\boldsymbol{b}\|_{L^\infty(\Omega_h)} C_T C_P},
\]
which are admissible according to the constraints imposed on these parameters, we get
\[
\eta_0=\frac94 + 27\,C_T^2(1+C^2).
\]
Hence, one may for instance choose
$\eta=\frac52 + 27\,C_T^2(1+C^2)$.
\end{remark}

\section{Error Estimates}\label{errorEstimates}

In this section, we derive error estimates for the DG--ROD method on both convex and non-convex domains. We first examine whether Galerkin orthogonality holds; if not, we estimate the resulting residual. The error analysis relies on the inf-sup condition \eqref{DesigAlpha_v2} and classical interpolation results. Let $I_h(w) \in \mathcal{W}_h$ denote the $\mathcal{P}_N$-interpolant of $w$ at the nodes of $\mathcal{W}_h$. For elements $k \notin I^B$, $I_h(w)$ is the standard interpolant at the mesh nodes $\boldsymbol{x}_i^k$ (left panel, Figure~\ref{NodalSet}). For boundary elements $k \in I^B$, $I_h(w)$ interpolates $w$ at the $m_N+2$ nodes of $T^k$ outside $e^{kB}$ and the $N-1$ nodes on $\partial \Omega$ corresponding to the interior of $e^{kB}$ (right panel, Figure~\ref{NodalSet}).

We follow the approach for error analysis outlined in \cite{ruas2019OptimalLA}, yet its extension to the discontinuous Galerkin setting introduces significant technical novelties. As discussed in the proof of the inf-sup condition (BNB1), the DG formulation introduces an additional residual term $r(w)$ arising from the non-standard treatment of boundary conditions. This requires a more refined treatment and introduces technical challenges not present in the classical framework of \cite{ruas2019OptimalLA}. 
A fundamental distinction lies in the nature of the approximation: unlike the continuous finite element methods in \cite{ruas2019OptimalLA}, the DG interpolant $I_h(u)$ is inherently discontinuous across element interfaces. Consequently, the DG-norm of the interpolation error must explicitly incorporate jump contributions, reflecting the discontinuity of both the numerical solution and its interpolant. Furthermore, the analysis requires a rigorous element-wise integration by parts, which generates additional interface terms involving jump and average operators. These features represent a substantial departure from the analysis in \cite{ruas2019OptimalLA} and require a dedicated treatment to ensure the convergence of the DG--ROD method.


\subsection{Convex case}

Let $\Omega$ be a convex domain, $u \in H^2(\Omega)$ the exact solution of \eqref{problem_E}--\eqref{problem_B}, and $v \in \mathcal{V}_h$.  Since $\Omega_h \subset \Omega$ and
$
\llbracket u \rrbracket = 0,$ $\llbracket {}\boldsymbol{ \nabla u} \rrbracket = 0,
$
(see \cite{Ern_Pietro}, Lemmas 1.23--1.24), using the jump/average formula \eqref{formula_jumpsAverage}, we have
\begin{align*}
B^{DG}_h(u,v) &= (\nabla_h u, \nabla_h v)_{L^2(\Omega_h)} - (\boldsymbol{b} u, \nabla_h v)_{L^2(\Omega_h)} + (c u, v)_{L^2(\Omega_h)} \\
&\quad - \int_{\Gamma_0} \llbracket u \rrbracket \cdot \{\!\{\nabla_h v\}\!\} \dif s
- \int_{\Gamma_0} \llbracket v \rrbracket \cdot \{\!\{\nabla_h u\}\!\} \dif s
+ \int_{\Gamma} \frac{\eta}{h_e} \llbracket v \rrbracket \cdot \llbracket u \rrbracket \dif s \\
&\quad + \int_{\Gamma_0} \boldsymbol{b} \cdot \llbracket v \rrbracket \{\!\{ u \}\!\} \dif s
+ \int_{\Gamma} \frac{|\boldsymbol{b} \cdot \boldsymbol{n}_e|}{2} \llbracket v \rrbracket \cdot \llbracket u \rrbracket \dif s \\
&= (\nabla_h u, \nabla_h v)_{L^2(\Omega_h)} - (\boldsymbol{b} u, \nabla_h v)_{L^2(\Omega_h)} + (c u, v)_{L^2(\Omega_h)} \\
&\quad - \int_{\Gamma_0} \llbracket v \rrbracket \cdot \{\!\{\nabla_h u\}\!\} \dif s
+ \int_{\Gamma_0} \boldsymbol{b} \cdot \llbracket v \rrbracket \{\!\{ u \}\!\} \dif s.
\end{align*}
Applying elementwise integration by parts and using the fact that $\llbracket u \rrbracket = \llbracket \nabla u \rrbracket = 0$, we get
\begin{align*}
B^{DG}_h(u,v) &= \sum_{k=1}^K \big( -(\Delta u, v)_{L^2(T^k)} + (\nabla \cdot (\boldsymbol{b} u), v)_{L^2(T^k)} + (c u, v)_{L^2(T^k)} \big) \\
&= (-\Delta u + \nabla \cdot (\boldsymbol{b} u) + c u, v)_{L^2(\Omega_h)} = (f, v)_{L^2(\Omega_h)}.
\end{align*}
Hence, the DG--ROD method is consistent, and the Galerkin orthogonality holds
\begin{equation}\label{GalerkinOrthog}
B^{DG}_h(u - u_h, v) = 0, \qquad \forall v \in \mathcal{V}_h.
\end{equation}

We now derive rigorous error estimates for the DG--ROD method in the case of convex domains. We first establish an estimate in the DG norm (also called energy norm), which captures both the elementwise $H^1$-seminorm and the contributions from jumps across element interfaces. Then, leveraging duality arguments, we derive an $L^2$-norm estimate exhibiting optimal convergence rates.

\begin{theorem}[DG-norm estimate]\label{error_DG}
Let $\Omega \subset \mathbb{R}^2$ be a convex domain, and $u \in H^{N+1}(\Omega)$ the solution of the boundary value problem \eqref{problem_E}--\eqref{problem_B}. Then, for $h<h_0$, where $h_0$ is the threshold defined in the proof of Theorem 1, there exists a constant $\mathcal{C} > 0$, independent of $h$ and $u$, such that the DG--ROD solution $u_h \in \mathcal{V}_h$ satisfies
\begin{equation}\label{T41}
\opnorm{u - u_h} \leq \mathcal{C} h^N |u|_{H^{N+1}(\Omega)},
\end{equation}
where $\opnorm{\cdot}$ denotes the DG norm defined in \eqref{DG_norm}.
\end{theorem}

\begin{proof}
The proof proceeds by adapting the arguments presented in Theorem 4.3 of \cite{ruas2019OptimalLA}. Let $I_h(u) \in \mathcal{W}_h$ denote the $\mathcal{P}_N$-interpolant of $u$ at the nodes of the DG space. By the triangle inequality:
\[
\opnorm{u - u_h} \leq \opnorm{u - I_h(u)} + \opnorm{u_h - I_h(u)}.
\]
Using the inf-sup condition (Theorem \ref{ProposicaoDesAlpha_v2}) for the DG bilinear form $B_h^{DG}$, we have
\begin{equation}\label{aux}
\opnorm{u_h - I_h(u)} \leq \frac{1}{\alpha} \sup_{v \in \mathcal{V}_h \setminus \{0\}} \frac{B^{DG}_h(u_h - I_h(u),v)}{\opnorm{v}}.
\end{equation}

Adding and subtracting $u$ in the numerator, using Galerkin orthogonality \eqref{GalerkinOrthog}, and boundedness of $B_h^{DG}$ (see \eqref{boundedness}), we obtain
\[
\opnorm{u - u_h} \leq \left(1 + \frac{C_b}{\alpha}\right) \opnorm{u - I_h(u)}.
\]

Unlike standard continuous finite elements, the DG interpolant $I_h(u)$ is discontinuous across element boundaries. Therefore, the DG norm of the interpolation error includes jump terms:
\begin{align}\label{normDG_interp}
\opnorm{&u\!-\!I_h(u)}^2 = \sum_{k=1}^K \left(\!\norm{u\!-\!I_h(u)}^2_{L^2(T^k)}\!+\!\abs{u\!-\!I_h(u)}^2_{H^1(T^k)}\!+\!h_k^2 \abs{u\!-\!I_h(u)}^2_{H^2(T^k)}\!\right) \nonumber\\
&\quad + \sum_{e \in \mathcal{E}_h} h_e^{-1} \norm{\llbracket u\!-\!I_h(u) \rrbracket}^2_{L^2(e)}\!+\!\sum_{e \in \mathcal{E}_h} \frac{1}{2} \norm{ |\boldsymbol{b}\!\cdot\!\boldsymbol{n}_e|^{1/2} \llbracket u\!-\!I_h(u) \rrbracket }^2_{L^2(e)}.
\end{align}

Using the trace inequality \eqref{trace_2.4} and classical polynomial approximation results (Lemma 4.1 in \cite{ruas2019OptimalLA}), we conclude
\[
\opnorm{u - I_h(u)} \leq \mathcal{C}_a h^N |u|_{H^{N+1}(\Omega)}.
\]

Finally, combining the previous steps gives the desired estimate \eqref{T41} with
\(
\mathcal{C} = \mathcal{C}_a \left(1 + \frac{C_b}{\alpha}\right).
\)
\end{proof}

\begin{theorem}[$L^2$-norm estimate]\label{error_L2}
Let $\Omega$ be a convex domain with piecewise $C^{N+1}$ boundary $\partial \Omega$ and $u$ the solution of \eqref{problem_E}--\eqref{problem_B} with $u \in H^{N+1+r}(\Omega)$, $r = 1/2 + \epsilon$, $\epsilon>0$. Then, for $h<h_0$, where $h_0$ is the threshold defined in the proof of Theorem 1, 
and $N>1$, the DG--ROD solution $u_h$ satisfies
\begin{equation}\label{T42}
\norm{u - u_h}_{L^2(\Omega_h)} \leq \mathcal{C}_0 h^{N+1} \norm{u}_{H^{N+1+r}(\Omega)},
\end{equation}
where $\mathcal{C}_0$ is independent of $h$ and $u$.
\end{theorem}

\begin{proof}
The proof follows the duality argument introduced in \cite{ruas2019OptimalLA}. Recall that every function in $\mathcal{W}_h$ is defined in $\overline{\Omega}\setminus \Omega_h$. Let $z\in H^1_0(\Omega)$ denote the solution of the adjoint problem
 \begin{align*}
-\Delta z\left({\boldsymbol x}\right) - \boldsymbol{b}(\boldsymbol{x}) \cdot \nabla z(\boldsymbol{x})+ c\left({\boldsymbol x}\right) z\left({\boldsymbol x}\right)  &= u\left({\boldsymbol x}\right)  - u_h \left({\boldsymbol x}\right), \quad {\boldsymbol x} \in \Omega, \\
   z \left({\boldsymbol x}\right)&=0, \quad {\boldsymbol x} \in \partial \Omega. 
\end{align*} 
Standard elliptic regularity \cite{Evans} implies $z\in H^2(\Omega)$ and there exists a
constant $C(\Omega)>0$ such that
\begin{equation}
 \|z\|_{H^2(\Omega)} \le C(\Omega)\|u-u_h\|_{L^2(\Omega)} .
\end{equation}
Consequently,
\begin{equation}
 \|u-u_h\|_{L^2(\Omega)}
 \le C(\Omega)\,
 \frac{(u-u_h,-\Delta z-\boldsymbol{b}\cdot\nabla z+cz)_{L^2(\Omega)}}
      {\|z\|_{H^2(\Omega)}} .
\end{equation}

Considering $\Delta h = \Omega \setminus \Omega_h$ and using $z=0$ on $\partial \Omega$, we have
\begin{align*}
    &(u\!-\!u_h, -\Delta z\!-\!\boldsymbol{b}\!\cdot\!\nabla z\!+\!cz)_{L^2(\Omega)}\!=\!(u\!-\!u_h,-\Delta z-\!\boldsymbol{b}\!\cdot\!\nabla z\!+\!cz)_{L^2(\Omega_h)}   \\
    &\quad+  (u - u_h, -\Delta z -\boldsymbol{b}\cdot \nabla z + cz)_{L^2(\Delta h)} \\
    &= -\!\sum_{k=1}^K\!\int_{\partial T^k}\!(\!u\!-\!u_h\!)\frac{\partial z}{\partial n_h}\!\diff s\!+\!(\nabla_h (\!u\!-\!u_h\!), \nabla z)_{L^2(\Omega_h)}\!-\!(\boldsymbol{b}(\!u\!-\!u_h\!), \nabla z)_{L^2(\Omega_h)} \\
     &\quad +\!(u\!-\!u_h, cz)_{L^2(\Omega_h)}\!-\!\int_{\partial \Omega_h} (\!u\!-\!u_h\!)\frac{\partial z}{\partial n_h} \diff s\!+\!b_{1h}(\!u\!-\!u_h,z\!)\!+\!B^{DG}_{\Delta h}(u\!-\!u_h,z)
\end{align*}
with
\begin{align}
    B^{DG}_{\Delta h}(u\!-\!u_h, z) &= \int_{\Delta h} \nabla_h (u\!-\!u_h) \cdot \nabla z\!-\!\boldsymbol{b} (u\!-\!u_h) \cdot \nabla z\!+\!(u\!-\!u_h) c z \dif\boldsymbol{x}, \label{a_deltah_expressao} \\
    b_{1h}(u-u_h, z) &= - \int_{\partial \Omega} (u-u_h) \frac{\partial z}{\partial n} \dif s. \label{equacao_b1h}
\end{align}
Since $\llbracket z \rrbracket = 0$ and $\llbracket \nabla z \rrbracket = 0$, we can rewrite
\begin{align}\label{numerador_1}
     &(u-u_h, -\Delta z-\boldsymbol{b}\cdot \nabla z+cz)_{L^2(\Omega)} \nonumber \\
     &=\!B^{DG}_h(u\!-\!u_h, z)+B^{DG}_{\Delta h}(u\!-\!u_h,z)\!+\!b_{1h}(u\!-\!u_h,z)\!-\!2\int_{\partial \Omega_h} (u\!-\!u_h)\frac{\partial z}{\partial n_h} \diff s \nonumber \\
     &- \int_{\partial \Omega_h} \frac{\eta}{h_e}  \llbracket u-u_h\rrbracket \cdot  \llbracket z\rrbracket \dif s-\int_{\partial \Omega_h} \frac{\abs{\boldsymbol{b}\cdot \boldsymbol{n}_e}}{2}  \llbracket u-u_h\rrbracket \cdot  \llbracket z\rrbracket \dif s.
\end{align} 

Let $\Pi_h(z)$ denote the continuous piecewise linear interpolant of $z$ over the mesh vertices and set $z_h = \Pi_h(z) \in \mathcal{V}_h$ in $\Omega_h$. Then, by definition,
\begin{equation}\label{estimativa_erro2}
    B^{DG}_h(u, z_h) = (f, z_h)_{L^2(\Omega_h)} = B^{DG}_h(u_h, z_h).
\end{equation}
Following \cite{ruas2019OptimalLA}, define $e_h(z) = z - \Pi_h(z)$ and consider
\begin{align}
    b_{2h}(\!u\!-\!u_h, \Pi_h(z)\!) &=\!\sum_{k \in I^B}\!\int_{\Delta_k} -\Delta(\!u\!-\!u_h\!)\Pi_h(z)\!-\!\boldsymbol{b}(\!u\!-\!u_h\!)\!\cdot\!\nabla\Pi_h(z) \dif \boldsymbol{x} \nonumber \\
    &\quad+ \sum_{k \in I^B}  \int_{\Delta_k} c(u-u_h)  \Pi_h(z) \dif \boldsymbol{x}, \label{equacao_b2h} \\
    b_{3h}(u-u_h,\Pi_h(z)) &= \sum_{k \in I^B} \int_{(T^k \cup \Delta_k)\cap \partial \Omega}  \frac{\partial (u-u_h)}{\partial n}  \Pi_h(z) \dif s \label{equacao_b3h},
\end{align} 
\begin{equation}\label{equacao_b4h}
    b_{4h}(u-u_h, e_h(z)) = B^{DG}_{\Delta h}(u-u_h,z - \Pi_h(z)).
\end{equation}
%
Thus,
\begin{equation}\label{estimativa_erro4}
    B^{DG}_{\Delta h}(u\!-\!u_h, z) = b_{2h}(u-u_h, \Pi_h(z))\!+\!b_{3h}(u-u_h, \Pi_h(z))\!+\!b_{4h}(u\!-\!u_h, e_h(z)).
\end{equation}

By Galerkin orthogonality \eqref{GalerkinOrthog},
\begin{equation}\label{estimativa_erro5}
    B^{DG}_h(u-u_h, z) = B^{DG}_h(u-u_h, e_h(z)).
\end{equation}
Defining the remaining boundary terms
\begin{align}
    b_{5h}(u-u_h, z) &= -2 \int_{\partial \Omega_h} (u-u_h) \frac{\partial z}{\partial n_h} \dif s, \label{equacao_b5h_convex} \\
    b_{6h}(u-u_h, z) &= - \int_{\partial \Omega_h} \frac{\eta}{h_e} \llbracket u-u_h \rrbracket \cdot \llbracket z \rrbracket \dif s, \label{equacao_b6h_convex} \\
    b_{7h}(u-u_h, z) &= - \int_{\partial \Omega_h} \frac{|\boldsymbol{b}\cdot \boldsymbol{n}_e|}{2} \llbracket u-u_h \rrbracket \cdot \llbracket z \rrbracket \dif s, \label{equacao_b7h_convex}
\end{align}
and combining \eqref{estimativa_erro4}--\eqref{equacao_b7h_convex} with \eqref{numerador_1}, we obtain
\begin{align}\label{estimativa_erro7}
    \|u\!-\!u_h\|_{L^2(\Omega)}\!&\le\!C(\Omega) \frac{B^{DG}_h(u\!-\!u_h, e_h(z))\!+\!b_{1h}(u\!-\!u_h, z)\!+\!b_{2h}(u\!-\!u_h, \Pi_h(z))}{\|z\|_{H^2(\Omega)}} \nonumber \\
    &\quad +\!C(\Omega)\!\frac{b_{3h}(u\!-\!u_h, \Pi_h(z))\!+\!b_{4h}(u\!-\!u_h, e_h(z))\!+\!b_{5h}(u\!-\!u_h, z)}{\|z\|_{H^2(\Omega)}} \nonumber \\
    &\quad + C(\Omega) \frac{b_{6h}(u-u_h, z) + b_{7h}(u-u_h, z)}{\|z\|_{H^2(\Omega)}}.
\end{align}

Using the boundedness inequality \eqref{boundedness} and applying Theorem \ref{error_DG}, we have
\begin{align}\label{ah_estimate}
    B^{DG}_h(\!u\!-\!u_h, e_h(\!z\!)\!) 
    &\le\!C_b\!\opnorm{u\!-\!u_h} \opnorm{e_h(z)}\!\le\!C_b\mathcal{C} h^N \abs{u}_{H^{N+1}(\Omega)} \opnorm{e_h(z)}.
\end{align}

From Lemma 4.1 in \cite{ruas2019OptimalLA} (with $j=0,1,2$) and since $h<1$, we deduce
\begin{align}\label{ehz_estimation}
    \opnorm{e_h(z)} \le C_{\Omega,z} h \abs{z}_{H^2(\Omega)},
\end{align}
where $C_{\Omega,z}$ is a mesh-independent constant. Substituting \eqref{ehz_estimation} into \eqref{ah_estimate} gives
\begin{align}\label{ah_estimate2}
    B^{DG}_h(u-u_h, e_h(z)) \le C_a h^{N+1} \norm{u}_{H^{N+1+r}(\Omega)} \norm{z}_{H^2(\Omega)},
\end{align}
with $C_a = C_b \mathcal{C} C_{\Omega,z}$.

The estimates for $b_{ih}$, $i=1,2,3,4$, follow similarly to \cite{ruas2019OptimalLA}. For $b_{5h}$, we have
\begin{align*}
    b_{5h}(u-u_h, z) \le 2 \abs{\sum_{e \in \partial \Omega_h} \int_e \llbracket u-u_h \rrbracket \cdot \nabla (z - \Pi_h(z) + \Pi_h(z)) \dif s}.
\end{align*}
Applying the Cauchy–Schwarz inequality,
\begin{align}
    b_{5h}(u-u_h, z) 
    &\le 2 \sum_{e \in \partial \Omega_h} \norm{h_e^{-1/2} \llbracket u-u_h \rrbracket}_{L^2(e)} \norm{h_e^{1/2} \nabla e_h(z)}_{L^2(e)} \label{b5h_aux1} \\
    &\quad + 2 \sum_{e \in \partial \Omega_h} \norm{h_e^{-1/2} \llbracket u-u_h \rrbracket}_{L^2(e)} \norm{h_e^{1/2} \nabla \Pi_h(z)}_{L^2(e)}. \label{b5h_aux2}
\end{align}

Following arguments in Subsection \ref{subsection:boundedness}, Theorem \ref{error_DG}, and Lemma 4.1 in \cite{ruas2019OptimalLA}, we bound \eqref{b5h_aux1} as
\begin{align}\label{b5h_aux3}
    &2 \sum_{e \in \partial \Omega_h} \norm{h_e^{-1/2} \llbracket u-u_h \rrbracket}_{L^2(e)} \norm{h_e^{1/2} \nabla e_h(z)}_{L^2(e)} \nonumber \\
    &\le 2 \opnorm{u-u_h} C_T \sqrt{6} C_\Omega h \abs{z}_{H^2(\Omega)} 
    \le C'_{b5} h^{N+1} \norm{u}_{H^{N+1+r}(\Omega)} \norm{z}_{H^2(\Omega)},
\end{align}
with $C'_{b5} = 2 \sqrt{6} \mathcal{C} C_T C_\Omega$.

For \eqref{b5h_aux2}, using Lemma 3.1 in \cite{ruas2019OptimalLA}, the fact that the distance between the physical and computational boundaries is of order $\mathcal{O}(h^2)$,  and $\norm{\nabla \Pi_h(z)}_{L^2(\Omega)}^2 \le \tilde{C}_\Omega^2 \norm{z}_{H^2(\Omega)}^2$, with $\tilde{C}_\Omega = \sqrt{2 + 2 C_\Omega^2 h_0^2}$, we obtain
\begin{align}\label{b5h_aux4}
    &2 \sum_{e \in \partial \Omega_h} \norm{h_e^{-1/2} \llbracket u-u_h \rrbracket}_{L^2(e)} \norm{h_e^{1/2} \nabla \Pi_h(z)}_{L^2(e)} \nonumber \\
    &\le 2 \mathcal{C} C_t C_{\partial \Omega}^{1/2} C_J h^{N+1} \norm{u}_{H^{N+1+r}(\Omega)} \norm{\nabla \Pi_h(z)}_{L^2(\Omega)}  \nonumber \\
    &\le  C''_{b5} h^{N+1} \norm{u}_{H^{N+1+r}(\Omega)} \norm{z}_{H^2(\Omega)}, \qquad C''_{b5} = 2 \mathcal{C} C_t C_{\partial \Omega}^{1/2} C_J \tilde{C}_\Omega.
\end{align}

Combining \eqref{b5h_aux3} and \eqref{b5h_aux4} yields
\begin{align*}
    b_{5h}(u-u_h, z) \le C_{b5} h^{N+1} \norm{u}_{H^{N+1+r}(\Omega)} \norm{z}_{H^2(\Omega)},\qquad C_{b5} = C'_{b5} + C''_{b5}.
\end{align*}

Similarly, for $b_{6h}$ and $b_{7h}$, using $\Pi_h(z) = 0$ on $\partial \Omega_h$, the Cauchy–Schwarz inequality, and Lemma 4.1 in \cite{ruas2019OptimalLA},
\begin{align}
    b_{6h}(u-u_h, z) &\le C_{b6} h^{N+1} \norm{u}_{H^{N+1+r}(\Omega)} \norm{z}_{H^2(\Omega)}, \qquad C_{b6} = \eta \mathcal{C} C_{\Omega,z},\\
    b_{7h}(u-u_h, z) &\le C_{b7} h^{N+1} \norm{u}_{H^{N+1+r}(\Omega)} \norm{z}_{H^2(\Omega)},\qquad C_{b7} =  \mathcal{C} C_{\Omega,z}.
\end{align}

Hence, all terms $b_{ih}$, $i=1,\dots,7$, satisfy
\begin{align}
    b_{ih}(u-u_h, \cdot) \le C_{bi} h^{N+1} \norm{u}_{H^{N+1+r}(\Omega)} \norm{z}_{H^2(\Omega)}, \quad i=1,\dots,7.
\end{align}

Finally, combining these bounds with \eqref{ah_estimate2} into \eqref{estimativa_erro7} and using $h<1$, we obtain \eqref{T42} with
\[
\mathcal{C}_0 = C(\Omega) \left( C_a + \sum_{i=1}^7 C_{bi} \right).
\]
\end{proof}

\subsection{Non-convex case}

We now consider a non-convex domain $\Omega$. In this setting, the Galerkin orthogonality no longer holds, and thus the residual
\[
R_h(v) = B^{DG}_h(u,v) - (f, v)_{L^2(\Omega_h)}, \quad v \in \mathcal{V}_h,
\]
is non-zero. To handle this, we introduce a fixed smooth domain $\tilde{\Omega}$ 
such that 
\[
\Tilde{\Omega}_h = \Omega \cup \Omega_h \subset \Tilde{\Omega}.
\] The domain $\tilde{\Omega}$ is chosen as a neighbourhood of $\Omega$ so that it provides a geometrically close approximation of the physical domain while remaining independent of the mesh parameter $h$ 
(see Figure \ref{nonConvexDomain}). This construction allows the exact solution $u$ to be extended to a function $\Tilde{u} \in H^{N+1}(\Tilde{\Omega})$ following \cite{STEIN2016252}.
Similarly to the norm defined in \eqref{DG_norm}, for $u \in H^2(\mathcal{T}_h)$ we consider
\begin{align} \label{DG_norm_nonConvex}
    \left(\opnorm{u}'\right)^2&= \sum_{k=1}^K \left(\norm{u}^2_{H^1(T^k \cap \Omega)} +  h_k^2 \abs{u}^2_{H^2(T^k  \cap \Omega)} \right) + \sum_{e \in \mathcal{E}_h} h_e^{-1} \norm{\llbracket u\rrbracket}^2_{L^2(e  \cap \Omega)} \nonumber \\
    &\quad+\sum_{e \in \mathcal{E}_h}\frac{1}{2} \norm{ \abs{\boldsymbol{b}\cdot \boldsymbol{n}_e}^{1/2}\llbracket u\rrbracket}^2_{L^2(e\cap \Omega)}.
\end{align}

\begin{figure}[H]
\centering
\includegraphics[width=12cm]{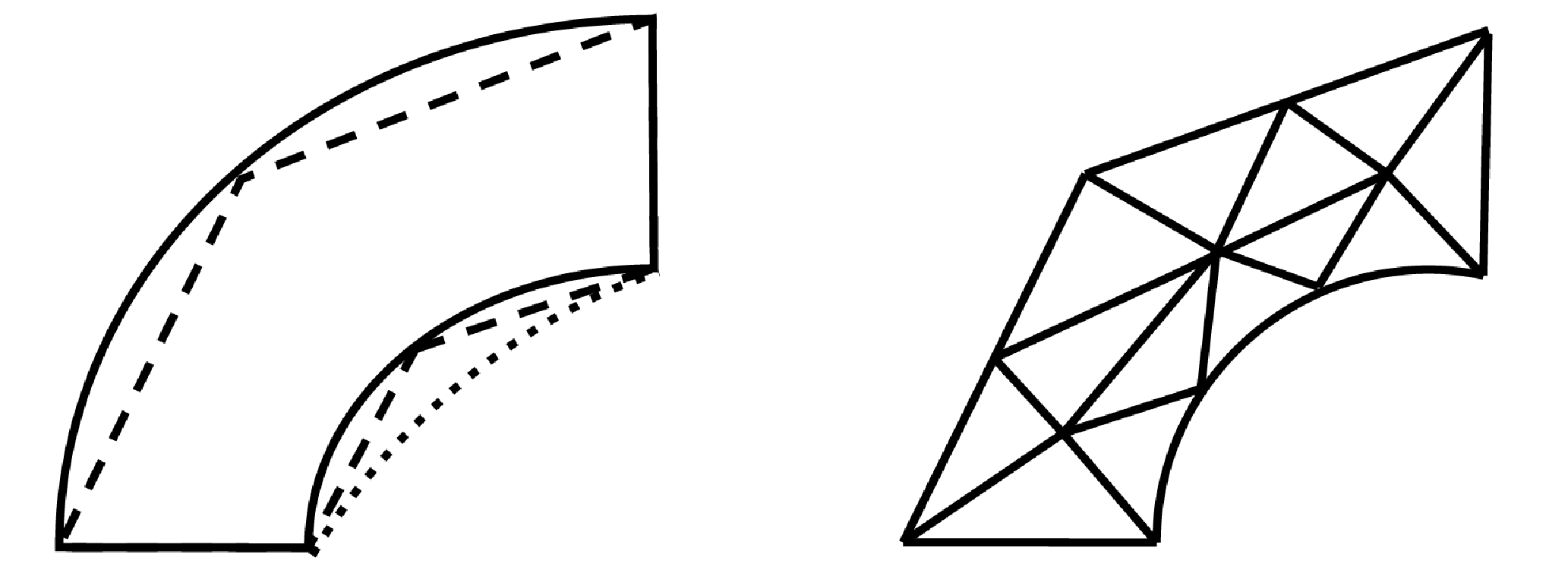}
\caption{\textbf{Left panel:} Example of non-convex domain $\Omega$ (solid), polygonal mesh $\Omega_h$ (dashed) and extended smooth domain $\tilde{\Omega}$ (dotted). \textbf{Right panel:} Example of $\Omega \cap \Omega_h$.}
\label{nonConvexDomain}
\end{figure}

Extend $f$ to $\tilde{\Omega} \setminus \Omega$ so that $f \in H^{N-1}(\tilde{\Omega})$, and denote the extension also by $f$. Assume continuous extensions of $a$, $\boldsymbol{b}$ and $c$ to $\tilde{\Omega} \setminus \Omega$, and let $f$ vanish in $\Omega_h \setminus \Omega$. Denote by $\tilde{u}$ the regular extension of $u$ to $\tilde{\Omega}$ with $\tilde{u} \in H^{N+1}(\tilde{\Omega})$ and $\tilde{u}_{|_\Omega} = u$ \cite{ruas2019OptimalLA, STEIN2016252}. The next theorems estimate the residual $B^{DG}_h(\tilde{u},v) - (f,v)_{L^2(\Omega_h)}$ via two approaches.

\begin{theorem}[DG-norm estimate - $L^2$ residual]\label{nonConvex_errorDG_comL2}
Let $u \in H^{N+1}(\Omega)$ solve \eqref{problem_E}--\eqref{problem_B}. Then, for $h<h_0$, where $h_0$ is the threshold defined in the proof of Theorem 1, 
there exist constants $\Tilde{\mathcal{C}}_1$ and $\widehat{C}_0$, independent of $h$, such that
\begin{equation}\label{nonConvex_errorDG_estimateL2}
\opnorm{ u-u_h}' \leq \Tilde{\mathcal{C}}_1 h^{N} \abs{\Tilde{u}}_{H^{N+1}(\Tilde{\Omega})} + \widehat{C}_0 h^{5/2} \norm{R(\Tilde{u})}_{L^2(\Tilde{\Omega})},
\end{equation}
where $\Tilde{u}$ is the regular extension of $u$ to $\Tilde{\Omega}$ and $
R(\Tilde{u}) = -\Delta \Tilde{u} + \nabla\cdot(\boldsymbol{b} \Tilde{u}) + c \Tilde{u}.
$
\end{theorem}
%

\begin{proof} This proof follows the arguments in the proof of Theorem 4.6 in \cite{ruas2019OptimalLA}. 
Let $I_h(\Tilde{u}) \in \mathcal{V}_h$ denote the DG interpolant of $\Tilde{u}$ over $\mathcal{T}_h$. Using the triangle inequality, we have
\[
\opnorm{u - u_h}' \leq \opnorm{\Tilde{u} - I_h(\Tilde{u})}' + \opnorm{u_h - I_h(\Tilde{u})}'.
\]
By standard DG interpolation theory for $\Tilde{u} \in H^{N+1}(\Tilde{\Omega})$, there exists a constant $\Tilde{\mathcal{C}}_a$ independent of $h$ such that
\[
\opnorm{\Tilde{u} - I_h(\Tilde{u})}' \leq \Tilde{\mathcal{C}}_a h^N \abs{\Tilde{u}}_{H^{N+1}(\Tilde{\Omega})}.
\]
Applying the discrete inf-sup condition for $B^{DG}_h$, we obtain
\begin{align*}
\opnorm{u_h - I_h(\Tilde{u})}' & \leq \frac{1}{\alpha} \sup_{v \in \mathcal{V}_h \setminus \{0\}} \frac{B^{DG}_h(u_h - I_h(\Tilde{u}), v)}{\opnorm{v}'} \\ &= \frac{1}{\alpha} \sup_{v \in \mathcal{V}_h \setminus \{0\}} \frac{B^{DG}_h(\Tilde{u} - I_h(\Tilde{u}), v) + B^{DG}_h(u_h - \Tilde{u}, v)}{\opnorm{v}'}.
\end{align*}
The first term is bounded by the continuity of $B^{DG}_h$:
\[
\abs{B^{DG}_h(\Tilde{u} - I_h(\Tilde{u}), v)} \leq C_b \opnorm{\Tilde{u} - I_h(\Tilde{u})}' \opnorm{v}'.
\]
For the second term, using the DG formulation and the definition of the residual $R(\Tilde{u}) = -\Delta \Tilde{u}+ \nabla \cdot (\boldsymbol{b} \Tilde{u}) + c \Tilde{u}$, we have
\[
B^{DG}_h(u_h - \Tilde{u}, v) = (f, v)_{L^2(\Omega_h)} - B^{DG}_h(\Tilde{u}, v) = - \sum_{k \in \mathcal{Q}^B} (R(\Tilde{u}), v)_{L^2(\Delta_k)},
\]
where $\Delta_k = T^k \cap (\Omega_h \setminus \Omega)$. Using Cauchy-Schwarz and trace inequalities on $\Delta_k$ (see \eqref{distance_boundaries} and Lemma 3.1 in \cite{ruas2019OptimalLA}), we get
\begin{align*}
    \sum_{k \in \mathcal{Q}^B} \abs{(R(\Tilde{u}), v)_{L^2(\Delta_k)}} &\leq \sum_{k \in \mathcal{Q}^B} \norm{R(\Tilde{u})}_{L^2(\Delta_k)} \norm{v}_{L^2(\Delta_k)} \\
    &\leq C^{3/2}_{\partial \Omega}C_J h^{5/2} \norm{R(\Tilde{u})}_{L^2(\Tilde{\Omega})} \opnorm{v}'.
\end{align*}

Combining the above bounds yields
\[
\opnorm{u_h - I_h(\Tilde{u})}' \leq \frac{\Tilde{\mathcal{C}}_a C_b}{\alpha} h^N \abs{\Tilde{u}}_{H^{N+1}(\Tilde{\Omega})} + \frac{C^{3/2}_{\partial \Omega}C_J}{\alpha} h^{5/2} \norm{R(\Tilde{u})}_{L^2(\Tilde{\Omega})}.
\]
Finally, we conclude
\[
\opnorm{u - u_h}' \leq \Tilde{\mathcal{C}}_1 h^N \abs{\Tilde{u}}_{H^{N+1}(\Tilde{\Omega})} + \widehat{C}_0 h^{5/2} \norm{R(\Tilde{u})}_{L^2(\Tilde{\Omega})},
\]
with  $ \Tilde{\mathcal{C}}_1 = \Tilde{\mathcal{C}}_a \left(1 +C_b/\alpha \right) $ and $\widehat{C}_0 = C^{3/2}_{\partial \Omega} C_J/\alpha $, which completes the proof.
\end{proof}

\begin{theorem}[DG-norm estimate - $L^\infty$ residual]\label{nonConvex_errorDG_L_infty}
For $N \geq 3$, let $u \in H^{N+1}(\Omega)$ be the solution of \eqref{problem_E}--\eqref{problem_B}. Then, for $h<h_0$, where $h_0$ is the threshold defined in the proof of Theorem 1, 
there exist mesh-independent constants $\Tilde{\mathcal{C}}_1$ and $C'_0$ such that the DG solution $u_h$ satisfies
\begin{equation}\label{nonConvex_errorDG_estimateLinfty}
\opnorm{ u- u_h}' \leq \Tilde{\mathcal{C}}_1 h^{N} \abs{\Tilde{u}}_{H^{N+1}(\Tilde{\Omega})} + C'_0 h^{7/2} \norm{R(\Tilde{u})}_{L^{\infty}(\Tilde{\Omega})},
\end{equation}
where $\Tilde{u}$ is a regular extension of $u$ to $\Tilde{\Omega}$ such that $\Tilde{u} \in H^{N+1}(\Tilde{\Omega})$ and $
R(\Tilde{u}) = -\Delta \Tilde{u} + \nabla\cdot(\boldsymbol{b} \Tilde{u}) + c \Tilde{u}.
$
\end{theorem}

\begin{proof}
The proof follows the same lines as that of Theorem~\ref{nonConvex_errorDG_comL2}. 
The only additional ingredient is the Sobolev embedding
$
H^{N+1}(\tilde{\Omega}) \hookrightarrow W^{2,\infty}(\tilde{\Omega}),$ $N >2$,
which implies that $\Delta \tilde{u} \in L^{\infty}(\tilde{\Omega})$. Applying the Cauchy-Schwarz inequality
\begin{align*}
&\abs{\sum_{k \in \mathcal{Q}^B} (R(\Tilde{u}), v)_{L^2(\Delta_k)}} \\
&\leq \sum_{k \in \mathcal{Q}^B} \norm{R(\Tilde{u})}_{L^2(\Delta_k)} \norm{v}_{L^2(\Delta_k)}
\leq C_{\partial\Omega}^2 C_J C(\partial \Omega) h^{7/2} \norm{R(\Tilde{u})}_{L^\infty(\Tilde{\Omega})} \opnorm{v},
\end{align*}
assuming that exists a mesh-independent constant $C(\partial \Omega)$ such that $\sum_{k \in \mathcal{Q}^B} h_k \leq C^2(\partial \Omega)$. Hence, \eqref{nonConvex_errorDG_estimateLinfty} holds with
\[
\Tilde{\mathcal{C}}_1 = \Tilde{\mathcal{C}}_a \Big(1 + \frac{C_b}{\alpha} \Big), \quad
C'_0 = \frac{C_{\partial\Omega}^2 C_J C(\partial \Omega)}{\alpha}.
\]
\end{proof}

For $N=2$ and $N=3$, we have
\begin{align}
\opnorm{u-u_h}' &\leq \Tilde{\mathcal{C}}_2 h^2 \Big(\abs{\Tilde{u}}_{H^3(\Tilde{\Omega})} + h^{1/2} \norm{R(\Tilde{u})}_{L^2(\Tilde{\Omega})}\Big), \label{NonConvex_N2} \\
\opnorm{u-u_h}' &\leq \Tilde{\mathcal{C}}_3 h^3 \Big(\abs{\Tilde{u}}_{H^4(\Tilde{\Omega})} + h^{1/2} \norm{R(\Tilde{u})}_{L^\infty(\Tilde{\Omega})}\Big). \label{NonConvex_N3}
\end{align}

We now establish error estimates in the $L^2$-norm for the case of a non-convex domain $\Omega$, requiring additional regularity from the solution $u$. Optimal convergence can be achieved not only when $u$ is more regular but also when the computational domain $\Omega_h$ better approximates the physical domain $\Omega$, i.e., when $\Omega_h \setminus \Omega$ is of order $h^q$ with $q>2$ \cite{Lew2011}. With our definition of $\Omega_h$, optimality is not attained for $N>2$, see \eqref{distance_boundaries}.

\begin{theorem}[$L^2$-norm estimate]\label{errorNonConvex_L2}
Let $N=2$, assume that $\Omega$ is a non-convex domain with piecewise $C^{N+1}$ boundary $\partial \Omega$, and $u \in H^{3+r}(\Omega)$ with $r = 1/2 + \epsilon$, $\epsilon > 0$. Then for $h<h_0$, where $h_0$ is the threshold defined in the proof of Theorem 1,  
\begin{align}\label{errorNonConvex_L2_eq}
\norm{u-u_h}_{L^2(\Omega \cap \Omega_h)}\!&\leq\!\Tilde{\mathcal{C}}_0 h^3 \Big(\!\norm{\Tilde{u}}_{H^{3+r}(\Tilde{\Omega})}\!+\!\abs{\Tilde{u}}_{H^3(\Tilde{\Omega})}\!+\!h^{1/2} \norm{R(\Tilde{u})}_{L^2(\Tilde{\Omega})}\!\Big),
\end{align}
where $\Tilde{u}$ is the regular extension of $u$ to $\Tilde{\Omega}$  and $
R(\Tilde{u}) = -\Delta \Tilde{u} + \nabla\cdot(\boldsymbol{b} \Tilde{u}) + c \Tilde{u}.
$
\end{theorem}

\begin{proof}
Recalling the proof of Theorem \ref{error_L2} and Theorem 4.8 in \cite{ruas2019OptimalLA}, let  $z \in H^1_0(\Omega)$ be the solution of
 \begin{align*}
-\Delta z\left({\boldsymbol x}\right) -  \boldsymbol{b}\left({\boldsymbol x}\right)\cdot \nabla  z\left({\boldsymbol x}\right) + c\left({\boldsymbol x}\right) z\left({\boldsymbol x}\right)  &= u\left({\boldsymbol x}\right)  - u_h \left({\boldsymbol x}\right), \quad {\boldsymbol x} \in \Omega, \\
   z \left({\boldsymbol x}\right)&=0, \quad {\boldsymbol x} \in \partial \Omega. 
\end{align*} Considering $\Omega= \left(\Omega \cap \Omega_h\right) \cup \Delta_h$ and using integration by parts, we obtain
\begin{align} 
     \norm{u - u_h}_{L^2(\Omega)} &\leq C(\Omega) \frac{(u - u_h, -\Delta z -\boldsymbol{b}\cdot \nabla z+ cz)_{L^2(\Omega)}}{\norm{z}_{H^2(\Omega)}} \nonumber \\
    &\leq C(\Omega)\!\frac{ B^{DG'}_h(u-u_h,z)+B^{DG}_{\Delta h}(u-u_h, z)+b_{1h}(u-u_h,z)}{\norm{z}_{H^2(\Omega)}} \nonumber \\
    &\quad -\!C(\Omega)\!\frac{2\int_{\partial \Omega_h \cap \Omega} (u\!-\!u_h)\frac{\partial z}{\partial n_h} \diff s +\int_{\partial \Omega_h \cap \Omega} \frac{\eta}{h_e}  \llbracket u\!-\!u_h\rrbracket\!\cdot\!\llbracket z\rrbracket \dif s }{\norm{z}_{H^2(\Omega)}} \nonumber \\
    &\quad - C(\Omega) \frac{\int_{\partial \Omega_h \cap \Omega} \frac{\abs{\boldsymbol{b}\cdot \boldsymbol{n_e}}}{2}  \llbracket u-u_h\rrbracket \cdot  \llbracket z\rrbracket \dif s}{\norm{z}_{H^2(\Omega)}},
\end{align} where $B^{DG}_{\Delta h}$ and $b_{1h}$ are defined in \eqref{a_deltah_expressao} and \eqref{equacao_b1h}, respectively and
\begin{align*}
   &B^{DG'}_h(u-u_h,z) = (\nabla_h (u-u_h), \nabla z)_{L^2(\Omega \cap \Omega_h)} - (\boldsymbol{b} (u-u_h), \nabla z)_{L^2(\Omega \cap \Omega_h)} \\
     &+\!(u\!-\!u_h, cz)_{L^2(\Omega \cap \Omega_h)} - \int_{\Gamma_0} \llbracket z\rrbracket \cdot  \{\!\{ \nabla_h (u-u_h)\}\!\}\dif s\!-\!\int_{\Gamma_0} \llbracket u\!-\!u_h \rrbracket\!\cdot\!\{\!\{ \nabla z\}\!\}\dif s\\
     &+\!\int_{\Gamma'} \frac{\eta}{h_e}  \llbracket u\!-\!u_h\rrbracket\!\cdot\!\llbracket z\rrbracket \dif s\!+\!\int_{\Gamma'} \frac{\abs{\boldsymbol{b}\cdot \boldsymbol{n_e}}}{2}  \llbracket u\!-\!u_h\rrbracket\!\cdot\!\llbracket z\rrbracket \dif s +  \int_{\Gamma_0} \boldsymbol{b} \cdot  \llbracket z\rrbracket \{\!\{u\!-\!u_h\}\!\} \dif s,
\end{align*} with $\Gamma' = \Gamma_0 \cup (\partial \Omega_h \cap \Omega)$. Since $\norm{u - u_h}_{L^2(\Omega  \cap \Omega_h)} \leq  \norm{u - u_h}_{L^2(\Omega)}$, we have 
\begin{align}\label{desigualdade_nonConvex1}
    \norm{u\!-\!u_h}_{L^2(\Omega  \cap \Omega_h)}\!&\leq\!C(\Omega) \frac{B^{DG'}_h(u\!-\!u_h,z) + B^{DG}_{\Delta h}(u\!-\!u_h, z)\!+\!b_{1h}(u\!-\!u_h,z)}{\norm{z}_{H^2(\Omega)}} \nonumber \\
    & +\!C(\Omega)\!\frac{ b_{8h}(\!u\!-\!u_h,z\!)\!+\!b_{9h}(\!u\!-\!u_h,z\!)\!+\!b_{10h}(\!u\!-\!u_h,z\!)}{\norm{z}_{H^2(\Omega)}},
\end{align} where 
\begin{align}
b_{8h}(u-u_h,z) &= - 2\int_{\partial \Omega_h \cap \Omega} (u-u_h)\frac{\partial z}{\partial n_h} \diff s, \label{equacao_b8h_nonConvex} \\
b_{9h}(u-u_h,z) &= -\int_{\partial \Omega_h \cap \Omega} \frac{\eta}{h_e}  \llbracket u-u_h\rrbracket \cdot  \llbracket z\rrbracket \dif s, \label{equacao_b9h_nonConvex} \\
b_{10h}(u-u_h,z) &= -\int_{\partial \Omega_h \cap \Omega} \frac{\abs{\boldsymbol{b}\cdot \boldsymbol{n_e}}}{2}  \llbracket u-u_h\rrbracket \cdot  \llbracket z\rrbracket \dif s. \label{equacao_b10h_nonConvex}
\end{align}

Following similar arguments as in \cite{ruas2019OptimalLA}, we write
\begin{align}\label{estimateExpression_nonconvex}
&\norm{u\!-\!u_h}_{L^2(\Omega  \cap \Omega_h)}\!\leq C(\Omega) \frac{ B^{DG'}_h(\!u\!-\!u_h,e_h(z)\!)\!+\!b_{1h}(\!u\!-\!u_h,z\!)\!+\!b_{2h}(\!u\!-\!u_h, \Pi_h(z)\!) }{\norm{z}_{H^2(\Omega)}} \nonumber\\
 &\quad + C(\Omega) \frac{ b_{3h}(u - u_h,\Pi_h(z)) + b_{4h}(u-u_h, e_h(z)) +  b_{8h}(u - u_h,z) }{\norm{z}_{H^2(\Omega)}} \nonumber \\
 &\quad +  C(\Omega)  \frac{ b_{9h}(u - u_h,z)  + +  b_{10h}(u - u_h,z) + b_{11h}(u_h,\Pi_h(z)) }{\norm{z}_{H^2(\Omega)}},
\end{align} where \begin{equation}\label{exp_b11h}
    b_{11h}(u_h,\Pi_h) = - \sum_{k \in \mathcal{Q}^B} \int_{\Delta_k} -\Delta u_h \Pi_h - \boldsymbol{b} u_h \cdot \nabla \Pi_h+ cu_h \Pi_h \dif \boldsymbol{x}.
\end{equation}

We now estimate upper bounds for $B^{DG'}_h$, $b_{8h}$, $b_{9h}$, $b_{10h}$ and $b_{11h}$. Using the boundedness of the bilinear form and applying the Theorem \ref{nonConvex_errorDG_comL2} with $N=2$ (see inequality \eqref{NonConvex_N2}), we first note  
\begin{align}\label{ah_estimate_nonconvex}
     B^{DG'}_h(\!u\!-\!u_h,e_h(z)\!)\!&\leq\!C_b\Tilde{\mathcal{C}_2} h^{2}  \left(\!\abs{\Tilde{u}}_{H^{3}(\Tilde{\Omega})}\!+\!h^{1/2}\norm{R(\tilde{u})}_{L^2(\Tilde{\Omega})}\!\right) \opnorm{e_h(z)}'.
\end{align} Moreover, similarly to \eqref{ehz_estimation}, we have 
\begin{align*}
    \opnorm{e_h(z)}'\leq  C'_{\Omega,z} h \abs{z}_{H^2(\Omega)},
\end{align*} with $C'_{\Omega,z}$ a mesh-independent constant. Thus, 
\begin{align}\label{ah_estimate2_nonconvex}
     B^{DG'}_h(\!u\!-\!u_h,e_h(z)\!)\!&\leq\!C'_{a} h^{3} \left(\!\abs{\Tilde{u}}_{H^{3}(\Tilde{\Omega})}\!+\!h^{1/2}  \norm{R(\tilde{u})}_{L^2(\Tilde{\Omega})}\!\right)\norm{z}_{H^2(\Omega)},
\end{align} with $C'_a = C_b \Tilde{\mathcal{C}_2} C'_{\Omega,z}$. 

Following similar arguments as in the estimates for $b_{5h}, b_{6h}$ and $b_{7h}$, and using \eqref{NonConvex_N2}, we obtain, for $i=8,9,10$,
\begin{align}\label{est_b8h}
   &b_{ih}(u\!-\!u_h,z)\leq \Tilde{C}_{bi} h^{3} \left(\norm{\Tilde{u}}_{H^{3+r}(\Tilde{\Omega})} + \abs{\Tilde{u}}_{H^{3}(\Tilde{\Omega})} + h^{1/2} \norm{R(\tilde{u})}_{L^2(\Tilde{\Omega})}\!\right) \norm{z}_{H^2(\Omega)}.
\end{align}

Similarly, for $b_{11h}$, following Theorem 4.8 in \cite{ruas2019OptimalLA}
\begin{align}
    b_{11h}(u_h,\Pi_h(z))\!&\leq\!\Tilde{C}_{b11} h^{3}  \left(\!\norm{\Tilde{u}}_{H^{3+r}(\Tilde{\Omega})}\!+\!\abs{\Tilde{u}}_{H^{3}(\Tilde{\Omega})}\!+\!h^{1/2} \norm{R(\tilde{u})}_{L^2(\Tilde{\Omega})}\!\right) \norm{z}_{H^2(\Omega)}. \label{est_b11h}
\end{align}

Estimates for $b_{ih}$, with $i=1,2,3,4$, follow as in the proof of Theorem \ref{error_L2}, using $N=2$ and replacing $\abs{u}_{H^{3}(\Omega)}$ by $ \abs{\Tilde{u}}_{H^{3}(\Tilde{\Omega})} + h^{1/2} \norm{R(\tilde{u})}_{L^2(\Tilde{\Omega})}$.

Finally, combining \eqref{ah_estimate2_nonconvex}--\eqref{est_b11h} and the estimates for $b_{ih}$, with $i=1,2,3,4$, into \eqref{estimateExpression_nonconvex}, and noting that $h<1$, we obtain \eqref{errorNonConvex_L2_eq}
with \[\Tilde{\mathcal{C}}_0 =  C(\Omega)\left(C'_a + \Tilde{C}_{b1} + \Tilde{C}_{b2} + \Tilde{C}_{b3}+ \Tilde{C}_{b4} +  \Tilde{C}_{b8} + \Tilde{C}_{b9} + \Tilde{C}_{b10} + \Tilde{C}_{b11}\right),\] where $ \Tilde{C}_{bi}$ is the constant in the estimate for $b_{ih}$. \end{proof}

\section{Numerical Results}\label{NumResults}

Let denote by $u_{h}$ an approximation of the solution $u$ for a given mesh $\mathcal T_h$ and $\left\Vert u-u_h\right\Vert$ the norm of the error. The method is of convergence order $p$ if one has asymptotically
\[
\left\Vert u-u_h\right\Vert \leq C h^p,
\]
with $C$ a real constant independent of $h$. The errors are assessed at the node points of the elements, $T^k\in\mathcal T_h$, $k=1,\ldots,K$. We compute the  $L^2$-errors
\begin{align*}
E_2\left(\mathcal T_h\right) &= \norm{u-u_h}_{L^2(\Omega_h)} = \sqrt{\sum_{k=1}^{K} \norm{u-u_h}^2_{L^2(T^k)} }.
\end{align*}

Recall that
\begin{align*}
\norm{u_h}^2_{L^2(T^k)}\!=\!\left(u^k_h, u^k_h\right)_{L^2(T^k)}= \int_{T^k} \sum_{i=1}^{N_p} \sum_{j=1}^{N_p} u^k_i u_j^k \ell_i^k\left(\boldsymbol{x}\right) \ell_j^k\left(\!\boldsymbol{x}\!\right) \dif \boldsymbol{x}\!=\!(\boldsymbol{u}^k)^{\textrm{T}} M^k \boldsymbol{u}^k,
\end{align*} where $M_{ij}^k = \left(\ell_i^k,\ell_j^k \right)_{L^2\left(T^k\right)}$.

Consider two different meshes, denoted as $\mathcal T_{h_1}$ and $\mathcal T_{h_2}$, whose the corresponding numerical solutions are denoted as $u_{h_1}$ and $u_{h_2}$, respectively. Then, the convergence order between two successively finer meshes is determined as
\[
\mathcal O_2 \left(\mathcal T_{h_1},\mathcal T_{h_2}\right)= \frac{\log \Big(E_{2}\left(\mathcal T_{h_1}\right)/E_{2}\left(\mathcal T_{h_2}\right)\Big)}{\log\left(h_1/h_2\right)}.
\]

The meshes are generated independently by Gmsh with different mesh size parameters. In all numerical experiments, the penalisation parameter $\eta$ defined in \eqref{penaltyConst} is chosen sufficiently large, namely, we consider $C_{\eta}=200$.

We now present a benchmark for the advection-diffusion-reaction problem on several curved boundary domains, namely a disk, an annulus, and a rose-shaped domain. For additional numerical results concerning diffusion-reaction problems, we refer to \cite{DGROD}. In that work, the DG--ROD method was constructed through an iterative combination of the DG method and polynomial reconstruction, and the authors reported both $L^{\infty}$- and $L^1$-errors. They also explored different strategies to define the nodes of $\mathcal{W}_h$ on $\partial \Omega$, including an intuitive approach in which nodes are placed along the normals to the edges of $\partial \Omega_h$. In this work, we adopt the construction described earlier in Subsection~\ref{subsec:spacediscfunc} (item~\ref{item4Def_Wh}) for the boundary nodes used in our numerical experiments.

\subsection{Disk domain}\label{diskDomain_section}
Consider the advection-diffusion-reaction equation on a disk of radius $R=1$ with a homogeneous Dirichlet boundary condition. An analytical solution is manufactured for problem~\eqref{problem_E}--\eqref{problem_B} and is given as $u\left(x,y\right) = x \sin{(1-x^2-y^2)}$, from which the corresponding source term is deduced. Simulations are carried out with successively finer meshes generated by Gmsh (version~4.13.1)~\cite{gmsh} (see Figure~\ref{diskDomain}).

\begin{figure}
  \begin{subfigure}[b]{.45\textwidth} 
  \centering {\includegraphics[width=5cm]{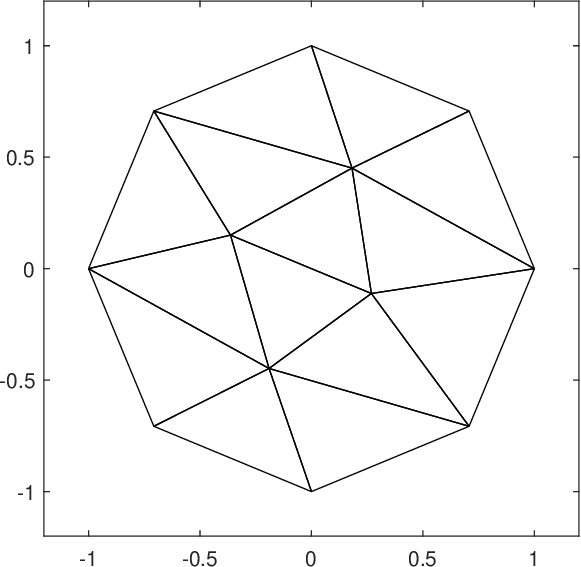}} \end{subfigure}
  \begin{subfigure}[b]{.45\textwidth}  \centering {\includegraphics[width=5cm]{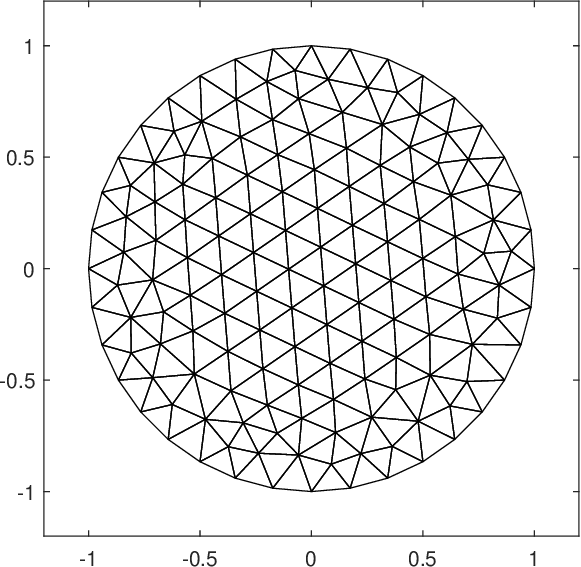}} \end{subfigure}
\caption[]{Unstructured meshes generated for the disk domain. Mesh with $K = 14$ and $h = 9.34\text{E$-$}01$ (left panel) and mesh with $K = 262$ and $h = 2.34\text{E$-$}01$ (right panel).}
\label{diskDomain}
\end{figure}

Simulations are performed using the classical DG method, in which homogeneous Dirichlet boundary conditions are prescribed at the computational boundary nodes. Each computational boundary node is mapped to a corresponding point on the physical boundary where the Dirichlet condition is prescribed. The boundary value evaluated at the physical boundary point is then imposed at the associated computational node. The results for the classical DG method are compared to the results obtained using the DG--ROD method. Tables~\ref{tabela_xsin_b1_c1}, \ref{tabela_xsin_bexp_c05exp}, \ref{tabela_xsin_b_c_polinom} and \ref{tabela_xsin_b50poly_cexp} report the $L^2-$errors and convergence orders for the classical DG method and for DG--ROD method taking $N=2,3,4$ and considering $\boldsymbol{b}(\boldsymbol{x})=(1,1)^{\textrm{T}}$ and $c(\boldsymbol{x})=1$, $\boldsymbol{b}(\boldsymbol{x})=(\exp{(x)},0)^{\textrm{T}}$ and $c(\boldsymbol{x})=\frac{1}{2}\exp{(x)}$, $\boldsymbol{b}(\boldsymbol{x})=(2-y^2,2-x)^{\textrm{T}}$ and $c(\boldsymbol{x})= 1 + (1+x)(1+y)^2$, and $\boldsymbol{b}(\boldsymbol{x})=50(-y,x)^{\textrm{T}}$ and $c(\boldsymbol{x})=\exp{(x)}$, respectively. The results for the classical DG method demonstrate the accuracy deterioration from such a geometrical mismatch without any specific treatment for curved boundaries, and the error convergence is limited to the second order. On the other hand, the quality of the approximations obtained with the DG--ROD method is in good agreement with the theoretical results.

\begin{table}[h!]
\caption{Errors and convergence orders for the classical DG method and for the DG--ROD method in the disk domain with the Dirichlet boundary conditions, 
$\boldsymbol{b}(\boldsymbol{x})=(1,1)^{\textrm{T}}$ and $c(\boldsymbol{x})=1$.}
\label{tabela_xsin_b1_c1}
\centering
\begin{adjustbox}{width=1\textwidth}
\begin{tabular}{@{}rc@{}r@{}cccc@{}r@{}cccc@{}r@{}cccc@{}}
\toprule
\multirow{3}{*}{$K$} & \multirow{3}{*}{$h$} & \phantom{aaaa} 
& \multicolumn{4}{c}{$N=2$} 
& \phantom{aaaa} 
& \multicolumn{4}{c}{$N=3$} 
& \phantom{aaaa} 
& \multicolumn{4}{c}{$N=4$} \\
\cmidrule{4-7}\cmidrule{9-12}\cmidrule{14-17}
& & 
& \multicolumn{2}{c}{DG} & \multicolumn{2}{c}{DG--ROD}
& & \multicolumn{2}{c}{DG} & \multicolumn{2}{c}{DG--ROD}
& & \multicolumn{2}{c}{DG} & \multicolumn{2}{c}{DG--ROD} \\
\cmidrule{4-5}\cmidrule{6-7}
\cmidrule{9-10}\cmidrule{11-12}
\cmidrule{14-15}\cmidrule{16-17}
& & 
& $E_2$ & $O_2$ & $E_2$ & $O_2$
& & $E_2$ & $O_2$ & $E_2$ & $O_2$
& & $E_2$ & $O_2$ & $E_2$ & $O_2$ \\
\midrule
16   & 9.28E$-$01 && 9.35E$-$02 &  --- & 1.57E$-$02 & --- && 9.14E$-$02 & --- & 8.59E$-$04 & --- && 8.96E$-$02 & --- & 7.72E$-$04 & --- \\ 
64   & 4.48E$-$01 && 2.53E$-$02 & 1.8  & 8.08E$-$04 & 4.1 && 2.49E$-$02 & 1.8 & 7.44E$-$05 & 3.4 && 2.46E$-$02 & 1.8 & 1.83E$-$05 & 5.1 \\ 
270  & 2.38E$-$01 && 5.02E$-$03 & 2.6  & 8.73E$-$05 & 3.5 && 4.95E$-$03 & 2.6 & 6.78E$-$06 & 3.8 && 4.92E$-$03 & 2.6 & 5.15E$-$07 & 5.7 \\ 
1082 & 1.20E$-$01 && 1.24E$-$03 & 2.1  & 9.52E$-$06 & 3.2 && 1.23E$-$03 & 2.0 & 3.19E$-$07 & 4.5 && 1.22E$-$03 & 2.0 & 1.10E$-$08 & 5.6 \\ 
4316 & 6.45E$-$02 && 3.06E$-$04 & 2.2  & 1.13E$-$06 & 3.4 && 3.04E$-$04 & 2.2 & 2.74E$-$08 & 3.9 && 3.04E$-$04 & 2.2 & 4.57E$-$10 & 5.1 \\
\bottomrule
\end{tabular}
\end{adjustbox}
\end{table}

\begin{table}[h!]
\caption{Errors and convergence orders for the classical DG method and for the DG--ROD method in the disk domain with the Dirichlet boundary conditions, $\boldsymbol{b}(\boldsymbol{x})=(\exp{(x)},0)^{\textrm{T}}$ and $c(\boldsymbol{x})=\frac{1}{2}\exp{(x)}$.}
\label{tabela_xsin_bexp_c05exp}
\centering
\begin{adjustbox}{width=1\textwidth}
\begin{tabular}{@{}rc@{}r@{}cccc@{}r@{}cccc@{}r@{}cccc@{}}
\toprule
\multirow{3}{*}{$K$} & \multirow{3}{*}{$h$} & \phantom{aaaa} 
& \multicolumn{4}{c}{$N=2$} 
& \phantom{aaaa} 
& \multicolumn{4}{c}{$N=3$} 
& \phantom{aaaa} 
& \multicolumn{4}{c}{$N=4$} \\
\cmidrule{4-7}\cmidrule{9-12}\cmidrule{14-17}
& & 
& \multicolumn{2}{c}{DG} & \multicolumn{2}{c}{DG--ROD}
& & \multicolumn{2}{c}{DG} & \multicolumn{2}{c}{DG--ROD}
& & \multicolumn{2}{c}{DG} & \multicolumn{2}{c}{DG--ROD} \\
\cmidrule{4-5}\cmidrule{6-7}
\cmidrule{9-10}\cmidrule{11-12}
\cmidrule{14-15}\cmidrule{16-17}
& & 
& $E_2$ & $O_2$ & $E_2$ & $O_2$
& & $E_2$ & $O_2$ & $E_2$ & $O_2$
& & $E_2$ & $O_2$ & $E_2$ & $O_2$ \\
\midrule
16   & 9.28E$-$01 && 1.08E$-$01 & --- & 2.71E$-$02 & --- && 1.01E$-$01 & --- & 8.97E$-$04 & --- && 9.90E$-$02 & --- & 8.56E$-$04 & --- \\ 
64   & 4.48E$-$01 && 2.86E$-$02 & 1.8 & 9.15E$-$04 & 4.7 && 2.79E$-$02 & 1.8 & 7.70E$-$05 & 3.4 && 2.77E$-$02 & 1.9 & 2.04E$-$05 & 5.1 \\ 
270  & 2.38E$-$01 && 5.68E$-$03 & 2.6 & 8.83E$-$05 & 3.7 && 5.58E$-$03 & 2.5 & 6.91E$-$06 & 3.8 && 5.55E$-$03 & 2.5 & 5.31E$-$07 & 5.8 \\ 
1082 & 1.20E$-$01 && 1.40E$-$03 & 2.1 & 9.44E$-$06 & 3.3 && 1.38E$-$03 & 2.0 & 3.26E$-$07 & 4.5 && 1.38E$-$03 & 2.0 & 1.11E$-$08 & 5.7 \\ 
4316 & 6.45E$-$02 && 3.46E$-$04 & 2.2 & 1.13E$-$06 & 3.4 && 3.44E$-$04 & 2.2 & 2.87E$-$08 & 3.9  && 3.44E$-$04 & 2.2 & 4.58E$-$10 & 5.1 \\
\bottomrule
\end{tabular}
\end{adjustbox}
\end{table}

\begin{table}[h!]
\caption{Errors and convergence orders for the classical DG method and for the DG--ROD method in the disk domain with the Dirichlet boundary conditions, $\boldsymbol{b}(\boldsymbol{x})=(2-y^2,2-x)^{\textrm{T}}$ and $c(\boldsymbol{x})= 1 + (1+x)(1+y)^2$.}
\label{tabela_xsin_b_c_polinom}
\centering
\begin{adjustbox}{width=1\textwidth}
\begin{tabular}{@{}rc@{}r@{}cccc@{}r@{}cccc@{}r@{}cccc@{}}
\toprule
\multirow{3}{*}{$K$} & \multirow{3}{*}{$h$} & \phantom{aaaa} 
& \multicolumn{4}{c}{$N=2$} 
& \phantom{aaaa} 
& \multicolumn{4}{c}{$N=3$} 
& \phantom{aaaa} 
& \multicolumn{4}{c}{$N=4$} \\
\cmidrule{4-7}\cmidrule{9-12}\cmidrule{14-17}
& & 
& \multicolumn{2}{c}{DG} & \multicolumn{2}{c}{DG--ROD}
& & \multicolumn{2}{c}{DG} & \multicolumn{2}{c}{DG--ROD}
& & \multicolumn{2}{c}{DG} & \multicolumn{2}{c}{DG--ROD} \\
\cmidrule{4-5}\cmidrule{6-7}
\cmidrule{9-10}\cmidrule{11-12}
\cmidrule{14-15}\cmidrule{16-17}
& & 
& $E_2$ & $O_2$ & $E_2$ & $O_2$
& & $E_2$ & $O_2$ & $E_2$ & $O_2$
& & $E_2$ & $O_2$ & $E_2$ & $O_2$ \\
\midrule
16   & 9.28E$-$01 && 1.10E$-$01 & --- & 2.08E$-$02 & --- && 1.07E$-$01 & --- & 8.99E$-$04 & --- && 1.05E$-$01 & --- & 1.07E$-$03 & --- \\ 
64   & 4.48E$-$01 && 3.01E$-$02 & 1.8 & 8.84E$-$04 & 4.3 && 2.98E$-$02 & 1.8 & 8.22E$-$05 & 3.3 && 2.96E$-$02 & 1.7 & 2.10E$-$05 & 5.4 \\ 
270  & 2.38E$-$01 && 6.04E$-$03 & 2.5 & 9.58E$-$05 & 3.5 && 5.97E$-$03 & 2.5 & 6.94E$-$06 & 3.9 && 5.94E$-$03 & 2.5 & 5.25E$-$07 & 5.8 \\ 
1082 & 1.20E$-$01 && 1.50E$-$03 & 2.0 & 9.81E$-$06 & 3.3 && 1.48E$-$03 & 2.0 & 3.39E$-$07 & 4.4 && 1.48E$-$03 & 2.0 & 1.11E$-$08 & 5.7 \\ 
4316 & 6.45E$-$02 && 3.70E$-$04 & 2.2 & 1.15E$-$06 & 3.4 && 3.69E$-$04 & 2.2 & 2.92E$-$08 & 3.9 && 3.68E$-$04 & 2.2 & 4.57E$-$10 & 5.1  \\
\bottomrule
\end{tabular}
\end{adjustbox}
\end{table}

\begin{table}[h!]
\caption{Errors and convergence orders for the classical DG method and for the DG--ROD method in the disk domain with the Dirichlet boundary conditions, $\boldsymbol{b}(\boldsymbol{x})=50(-y,x)^{\textrm{T}}$ and $c(\boldsymbol{x})=\exp{(x)}$.}
\label{tabela_xsin_b50poly_cexp}
\centering
\begin{adjustbox}{width=1\textwidth}
\begin{tabular}{@{}rc@{}r@{}cccc@{}r@{}cccc@{}r@{}cccc@{}}
\toprule
\multirow{3}{*}{$K$} & \multirow{3}{*}{$h$} & \phantom{aaaa} 
& \multicolumn{4}{c}{$N=2$} 
& \phantom{aaaa} 
& \multicolumn{4}{c}{$N=3$} 
& \phantom{aaaa} 
& \multicolumn{4}{c}{$N=4$} \\
\cmidrule{4-7}\cmidrule{9-12}\cmidrule{14-17}
& & 
& \multicolumn{2}{c}{DG} & \multicolumn{2}{c}{DG--ROD}
& & \multicolumn{2}{c}{DG} & \multicolumn{2}{c}{DG--ROD}
& & \multicolumn{2}{c}{DG} & \multicolumn{2}{c}{DG--ROD} \\
\cmidrule{4-5}\cmidrule{6-7}
\cmidrule{9-10}\cmidrule{11-12}
\cmidrule{14-15}\cmidrule{16-17}
& & 
& $E_2$ & $O_2$ & $E_2$ & $O_2$
& & $E_2$ & $O_2$ & $E_2$ & $O_2$
& & $E_2$ & $O_2$ & $E_2$ & $O_2$ \\
\midrule
16   & 9.28E$-$01 && 7.54E$-$02 & --- & 4.76E$-$02 & --- && 5.74E$-$02 & --- & 8.56E$-$03 & --- && 5.65E$-$02 & --- & 4.47E$-$04 & --- \\ 
64   & 4.48E$-$01 && 1.69E$-$02 & 2.1 & 1.76E$-$03 & 4.5 && 1.50E$-$02 & 1.8 & 1.04E$-$04 & 6.1 && 1.49E$-$02 & 1.8 & 1.50E$-$05 & 4.7 \\ 
270  & 2.38E$-$01 && 3.05E$-$03 & 2.7 & 2.01E$-$04 & 3.4 && 2.95E$-$03 & 2.6 & 8.12E$-$06 & 4.0 && 2.92E$-$03 & 2.6 & 5.21E$-$07 & 5.3 \\ 
1082 & 1.20E$-$01 && 7.40E$-$04 & 2.1 & 1.35E$-$05 & 4.0 && 7.26E$-$04 & 2.1 & 3.33E$-$07 & 4.7 && 7.24E$-$04 & 2.0 & 1.09E$-$08 & 5.7 \\ 
4316 & 6.45E$-$02 && 1.82E$-$04 & 2.2 & 1.38E$-$06 & 3.7 && 1.80E$-$04 & 2.2 & 2.59E$-$08 & 4.1  && 1.80E$-$04 & 2.2 & 4.58E$-$10 & 5.1 \\
\bottomrule
\end{tabular}
\end{adjustbox}
\end{table}

\subsection{Annulus domain}

Now, we consider an annulus domain with inner radius $R_I = 0.5$ and outer radius $R_E = 1$ meshed with triangular elements (see Figure~\ref{annulusDomain}). The analytic solution corresponds to the manufactured solution $u(x, y) = \log(x^2 + y^2)$ and the boundaries are prescribed with constant Dirichlet boundary conditions. The numerical simulations are carried out with successively finer meshes generated by Gmsh. As for the previous test case, simulations are performed for the classical DG method and for the DG--ROD method. The results are reported in Tables~\ref{tabela_annulus_b1_c1}, \ref{tabela_annulus_bexp_c05exp}, \ref{tabela_annulus_b_c_polinom} and \ref{tabela_annulus_b50poly_cexp} considering $\boldsymbol{b}(\boldsymbol{x})=(1,1)^{\textrm{T}}$ and $c(\boldsymbol{x})=1$, $\boldsymbol{b}(\boldsymbol{x})=(\exp{(x)},0)^{\textrm{T}}$ and $c(\boldsymbol{x})=\frac{1}{2}\exp{(x)}$,  $\boldsymbol{b}(\boldsymbol{x})=(2-y^2,2-x)^{\textrm{T}}$ and $c(\boldsymbol{x})= 1 + (1+x)(1+y)^2$ and $\boldsymbol{b}(\boldsymbol{x})=50(-y,x)^{\textrm{T}}$ and $c(\boldsymbol{x})=\exp{(x)}$, respectively. The optimal convergence orders are recovered due to the polynomial reconstruction of the boundary conditions.

\begin{remark}\label{remark_nonConvex}
Note that if there exists a function $\Tilde{u}$ defined in $\Tilde{\Omega}$ such that $\Tilde{u} \in H^{N+1}(\Tilde{\Omega})$, $\Tilde{u}$ coincide with $u$ in $\Omega$, $\Tilde{u}$ satisfies \eqref{problem_E} in $\Tilde{\Omega}$ and $\Tilde{u}$ vanishes on $\partial \Omega$ in the sense of trace, then, for $h$ sufficiently small, we may prove that there exists a mesh-independent constant $\mathcal{\Tilde{C}}$ such that
\begin{equation}\label{NonConvex_assumpUTilde}
    \opnorm{u - u_h }' \leq \mathcal{\Tilde{C}} h^N |\Tilde{u}|_{H^{N+1}(\Tilde{\Omega})},
\end{equation}  
where $\opnorm{\cdot}'$ denotes the norm defined in \eqref{DG_norm_nonConvex}. Thus, if the assumptions described above hold, the optimal convergence can be achieved in the $L^2-$ norm. However, given a regular $f$ in $\Tilde{\Omega}$, the existence of an associated $\Tilde{u}$ satisfying the above assumptions is not ensured. In this test case, the solution satisfies the conditions described above. Thus, the method recovers the optimal convergence orders for $N>2$. Specifically, the exact solution $u$ is extended by $\Tilde{u}$ defined by the same expression outside $\Omega$ and the extension of $f$ is defined by $f = -\Delta u + \nabla \cdot (\boldsymbol{b} u) + cu$.
\end{remark}

\begin{figure}
  \begin{subfigure}[b]{.45\textwidth} 
  \centering {\includegraphics[width=5cm]{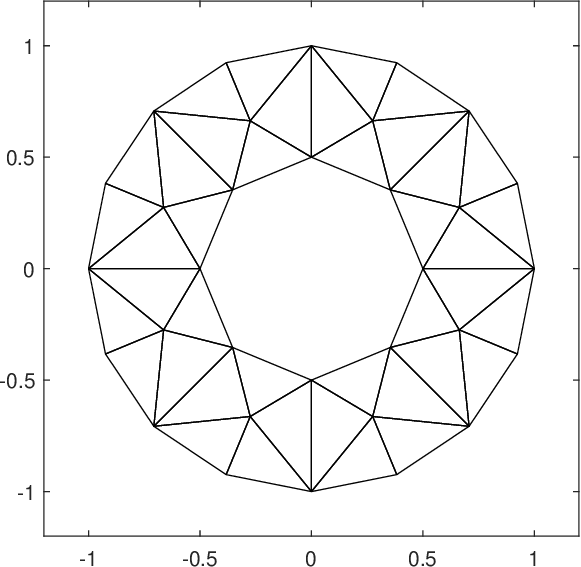}} \end{subfigure}
  \begin{subfigure}[b]{.45\textwidth}  \centering {\includegraphics[width=5cm]{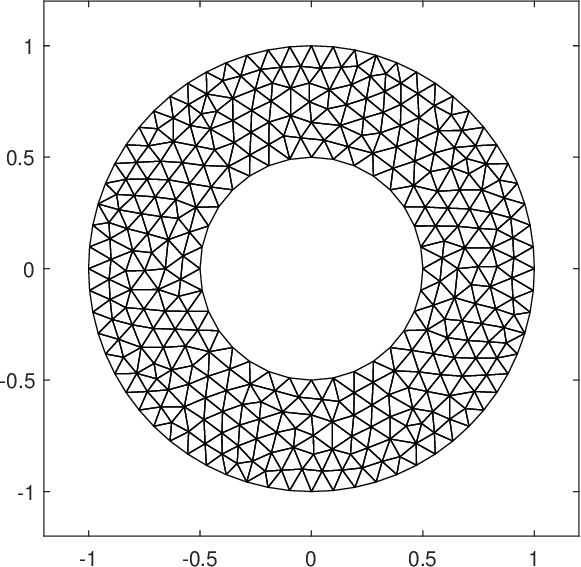}} \end{subfigure}
\caption[]{Unstructured mesh generated for the annulus domain. Mesh with $K = 40$ and $h = 5.00\text{E$-$}01$ (left panel) and mesh with $K = 608$ and $h = 1.31\text{E$-$}01$ (right panel).}
\label{annulusDomain}
\end{figure}

\begin{table}[h!]
\caption{Errors and convergence orders for the classical DG method and for the DG--ROD method in the annulus domain with the Dirichlet boundary conditions, $\boldsymbol{b}(\boldsymbol{x})=(1,1)^{\textrm{T}}$ and $c(\boldsymbol{x})= 1$.}
\label{tabela_annulus_b1_c1}
\centering
\begin{adjustbox}{width=1\textwidth}
\begin{tabular}{@{}rc@{}r@{}cccc@{}r@{}cccc@{}r@{}cccc@{}}
\toprule
\multirow{3}{*}{$K$} & \multirow{3}{*}{$h$} & \phantom{aaaa} 
& \multicolumn{4}{c}{$N=2$} 
& \phantom{aaaa} 
& \multicolumn{4}{c}{$N=3$} 
& \phantom{aaaa} 
& \multicolumn{4}{c}{$N=4$} \\
\cmidrule{4-7}\cmidrule{9-12}\cmidrule{14-17}
& & 
& \multicolumn{2}{c}{DG} & \multicolumn{2}{c}{DG--ROD}
& & \multicolumn{2}{c}{DG} & \multicolumn{2}{c}{DG--ROD}
& & \multicolumn{2}{c}{DG} & \multicolumn{2}{c}{DG--ROD} \\
\cmidrule{4-5}\cmidrule{6-7}
\cmidrule{9-10}\cmidrule{11-12}
\cmidrule{14-15}\cmidrule{16-17}
& & 
& $E_2$ & $O_2$ & $E_2$ & $O_2$
& & $E_2$ & $O_2$ & $E_2$ & $O_2$
& & $E_2$ & $O_2$ & $E_2$ & $O_2$ \\
\midrule
40    & 5.00E$-$01 && 8.90E$-$02 & --- & 4.50E$-$03 & --- && 9.02E$-$02 & --- & 4.45E$-$04 & --- && 9.09E$-$02 & --- & 8.17E$-$05 & --- \\ 
140   & 2.74E$-$01 && 2.41E$-$02 & 2.2 & 5.20E$-$04 & 3.6 && 2.43E$-$02 & 2.2 & 4.86E$-$05 & 3.7 && 2.44E$-$02 & 2.2 & 4.10E$-$06 & 5.0 \\ 
640   & 1.34E$-$01 && 5.72E$-$03 & 2.0 & 4.27E$-$05 & 3.5 && 5.75E$-$03 & 2.0 & 1.54E$-$06 & 4.8 && 5.76E$-$03 & 2.0 & 5.15E$-$08 & 6.1 \\ 
2536  & 6.80E$-$02 && 1.44E$-$03 & 2.0 & 4.11E$-$06 & 3.5 && 1.45E$-$03 & 2.0 & 7.08E$-$08 & 4.5 && 1.45E$-$03 & 2.0 & 1.52E$-$09 & 5.2 \\ 
10208 & 3.66E$-$02 && 3.62E$-$04 & 2.2 & 4.60E$-$07 & 3.5 && 3.63E$-$04 & 2.2 & 4.09E$-$09 & 4.6 && 3.63E$-$04 & 2.2 & 4.70E$-$11 & 5.6 \\
\bottomrule
\end{tabular}
\end{adjustbox}
\end{table}

\begin{table}[h!]
\caption{Errors and convergence orders for the classical DG method and for the DG--ROD method in the annulus domain with the Dirichlet boundary conditions, $\boldsymbol{b}(\boldsymbol{x})=(\exp{(x)},0)^{\textrm{T}}$ and $c(\boldsymbol{x})= \frac{1}{2}\exp{(x)}$.}
\label{tabela_annulus_bexp_c05exp}
\centering
\begin{adjustbox}{width=1\textwidth}
\begin{tabular}{@{}rc@{}r@{}cccc@{}r@{}cccc@{}r@{}cccc@{}}
\toprule
\multirow{3}{*}{$K$} & \multirow{3}{*}{$h$} & \phantom{aaaa} 
& \multicolumn{4}{c}{$N=2$} 
& \phantom{aaaa} 
& \multicolumn{4}{c}{$N=3$} 
& \phantom{aaaa} 
& \multicolumn{4}{c}{$N=4$} \\
\cmidrule{4-7}\cmidrule{9-12}\cmidrule{14-17}
& & 
& \multicolumn{2}{c}{DG} & \multicolumn{2}{c}{DG--ROD}
& & \multicolumn{2}{c}{DG} & \multicolumn{2}{c}{DG--ROD}
& & \multicolumn{2}{c}{DG} & \multicolumn{2}{c}{DG--ROD} \\
\cmidrule{4-5}\cmidrule{6-7}
\cmidrule{9-10}\cmidrule{11-12}
\cmidrule{14-15}\cmidrule{16-17}
& & 
& $E_2$ & $O_2$ & $E_2$ & $O_2$
& & $E_2$ & $O_2$ & $E_2$ & $O_2$
& & $E_2$ & $O_2$ & $E_2$ & $O_2$ \\
\midrule
40    & 5.00E$-$01 && 8.87E$-$02 & --- & 4.50E$-$03 & --- && 8.98E$-$02 & --- & 4.40E$-$04 & --- && 9.04E$-$02 & --- & 8.16E$-$05 & --- \\ 
140   & 2.74E$-$01 && 2.41E$-$02 & 2.2 & 5.14E$-$04 & 3.6 && 2.42E$-$02 & 2.2 & 3.85E$-$05 & 3.8 && 2.43E$-$02 & 2.2 & 4.07E$-$06 & 5.0 \\ 
640   & 1.34E$-$01 && 5.68E$-$03 & 2.0 & 4.25E$-$05 & 3.5 && 5.71E$-$03 & 2.0 & 1.54E$-$06 & 4.8 && 5.72E$-$03 & 2.0 & 5.14E$-$08 & 6.1 \\ 
2536  & 6.80E$-$02 && 1.43E$-$03 & 2.0 & 4.09E$-$06 & 3.5 && 1.44E$-$03 & 2.0 & 7.04E$-$08 & 4.6 && 1.44E$-$03 & 2.0 & 1.15E$-$09 & 5.2 \\ 
10208 & 3.66E$-$02 && 3.60E$-$04 & 2.0 & 4.59E$-$07 & 3.5 && 3.61E$-$04 & 2.2 & 4.07E$-$09 & 4.6 && 3.61E$-$04 & 2.2 & 4.72E$-$11 & 5.6 \\
\bottomrule
\end{tabular}
\end{adjustbox}
\end{table}

\begin{table}[h!]
\caption{Errors and convergence orders for the classical DG method and for the DG--ROD method in the annulus domain with the Dirichlet boundary conditions, $\boldsymbol{b}(\boldsymbol{x})=(2-y^2,2-x)^{\textrm{T}}$ and $c(\boldsymbol{x})= 1 + (1+x)(1+y)^2$.}
\label{tabela_annulus_b_c_polinom}
\centering
\begin{adjustbox}{width=1\textwidth}
\begin{tabular}{@{}rc@{}r@{}cccc@{}r@{}cccc@{}r@{}cccc@{}}
\toprule
\multirow{3}{*}{$K$} & \multirow{3}{*}{$h$} & \phantom{aaaa} 
& \multicolumn{4}{c}{$N=2$} 
& \phantom{aaaa} 
& \multicolumn{4}{c}{$N=3$} 
& \phantom{aaaa} 
& \multicolumn{4}{c}{$N=4$} \\
\cmidrule{4-7}\cmidrule{9-12}\cmidrule{14-17}
& & 
& \multicolumn{2}{c}{DG} & \multicolumn{2}{c}{DG--ROD}
& & \multicolumn{2}{c}{DG} & \multicolumn{2}{c}{DG--ROD}
& & \multicolumn{2}{c}{DG} & \multicolumn{2}{c}{DG--ROD} \\
\cmidrule{4-5}\cmidrule{6-7}
\cmidrule{9-10}\cmidrule{11-12}
\cmidrule{14-15}\cmidrule{16-17}
& & 
& $E_2$ & $O_2$ & $E_2$ & $O_2$
& & $E_2$ & $O_2$ & $E_2$ & $O_2$
& & $E_2$ & $O_2$ & $E_2$ & $O_2$ \\
\midrule
40    & 5.00E$-$01 && 9.10E$-$02 & --- & 4.66E$-$03 & --- && 9.24E$-$02 & --- & 4.48E$-$04 & --- && 9.31E$-$02 & --- & 8.19E$-$05 & --- \\ 
140   & 2.74E$-$01 && 2.47E$-$02 & 2.2 & 5.31E$-$04 & 3.6 && 2.49E$-$02 & 2.2 & 4.87E$-$05 & 3.7 && 2.50E$-$02 & 2.2 & 4.12E$-$06 & 5.0 \\ 
640   & 1.34E$-$01 && 5.85E$-$03 & 2.0 & 4.33E$-$05 & 3.5 && 5.89E$-$03 & 2.0 & 1.55E$-$06 & 4.8 && 5.90E$-$03 & 2.0 & 5.17E$-$08 & 6.1 \\ 
2536  & 6.80E$-$02 && 1.48E$-$03 & 2.1 & 4.12E$-$06 & 3.5 && 1.48E$-$03 & 2.0 & 7.14E$-$08 & 4.5 && 1.48E$-$03 & 2.0 & 1.52E$-$09 & 5.2 \\ 
10208 & 3.66E$-$02 && 3.71E$-$04 & 2.2 & 4.61E$-$07 & 3.5 && 3.72E$-$04 & 2.2 & 4.10E$-$09 & 4.6 && 3.72E$-$04 & 2.2 & 4.76E$-$11 & 5.6  \\
\bottomrule
\end{tabular}
\end{adjustbox}
\end{table}

\begin{table}[h!]
\caption{Errors and convergence orders for the classical DG method and for the DG--ROD method in the annulus domain with the Dirichlet boundary conditions, $\boldsymbol{b}(\boldsymbol{x})=50(-y,x)^{\textrm{T}}$ and $c(\boldsymbol{x})= \exp{(x)}$.}
\label{tabela_annulus_b50poly_cexp}
\centering
\begin{adjustbox}{width=1\textwidth}
\begin{tabular}{@{}rc@{}r@{}cccc@{}r@{}cccc@{}r@{}cccc@{}}
\toprule
\multirow{3}{*}{$K$} & \multirow{3}{*}{$h$} & \phantom{aaaa} 
& \multicolumn{4}{c}{$N=2$} 
& \phantom{aaaa} 
& \multicolumn{4}{c}{$N=3$} 
& \phantom{aaaa} 
& \multicolumn{4}{c}{$N=4$} \\
\cmidrule{4-7}\cmidrule{9-12}\cmidrule{14-17}
& & 
& \multicolumn{2}{c}{DG} & \multicolumn{2}{c}{DG--ROD}
& & \multicolumn{2}{c}{DG} & \multicolumn{2}{c}{DG--ROD}
& & \multicolumn{2}{c}{DG} & \multicolumn{2}{c}{DG--ROD} \\
\cmidrule{4-5}\cmidrule{6-7}
\cmidrule{9-10}\cmidrule{11-12}
\cmidrule{14-15}\cmidrule{16-17}
& & 
& $E_2$ & $O_2$ & $E_2$ & $O_2$
& & $E_2$ & $O_2$ & $E_2$ & $O_2$
& & $E_2$ & $O_2$ & $E_2$ & $O_2$ \\
\midrule
40    & 5.00E$-$01 && 8.54E$-$02 & --- & 5.30E$-$03 & --- && 8.87E$-$02 & --- & 4.96E$-$04 & --- && 8.95E$-$02 & --- & 8.26E$-$05 & --- \\ 
140   & 2.74E$-$01 && 2.43E$-$02 & 2.1 & 7.80E$-$04 & 3.2 && 2.43E$-$02 & 2.1 & 5.69E$-$05 & 3.6 && 2.44E$-$02 & 2.2 & 5.01E$-$06 & 4.6 \\ 
640   & 1.34E$-$01 && 5.73E$-$03 & 2.0 & 6.15E$-$05 & 3.6 && 5.76E$-$03 & 2.0 & 1.70E$-$06 & 4.9 && 5.77E$-$03 & 2.0 & 5.31E$-$08 & 6.4 \\ 
2536  & 6.80E$-$02 && 1.45E$-$03 & 2.0 & 4.72E$-$06 & 3.8 && 1.45E$-$03 & 2.0 & 7.25E$-$08 & 4.7 && 1.45E$-$03 & 2.0 & 1.53E$-$09 & 5.2 \\ 
10208 & 3.66E$-$02 && 3.63E$-$04 & 2.2 & 4.80E$-$07 & 3.7 && 3.64E$-$04 & 2.2 & 4.26E$-$09 & 4.6 && 3.64E$-$04 & 2.2 & 4.70E$-$11 & 5.6 \\
\bottomrule
\end{tabular}
\end{adjustbox}
\end{table}

\subsection{Rose–shaped domain}
Consider a geometry obtained by applying a diffeomorphic transformation to an annular domain $\Omega \textsc{\char13} $, whose physical interior and exterior boundaries have radius $r_{I}$ and $r_{E} $, respectively. This diffeomorphic transformation defines a mapping from $\Omega\textsc{\char13}$ to $\Omega$, where $\Omega$ denotes the resulting rose-shaped domain, described in polar coordinates
\begin{equation}\label{difTransf}
    \Omega \textsc{\char13} \to \Omega: \begin{bmatrix}
        r\textsc{\char13} \\ \theta\textsc{\char13}
    \end{bmatrix} \to \begin{bmatrix}
        r \\ \theta
    \end{bmatrix} = \begin{bmatrix}
        R(r\textsc{\char13}, \theta\textsc{\char13}; \alpha, \beta) \\ \theta
    \end{bmatrix},
\end{equation} where $\alpha$ is the number of petals and function $R(r\textsc{\char13}, \theta\textsc{\char13}):= R(r\textsc{\char13}, \theta\textsc{\char13}; \alpha, \beta)$ corresponds to a periodic radius perturbation of magnitude in $[-\beta, \beta]$, with $\beta \in \mathbb{R}$, given as
\begin{equation}\label{Raio}
     R(r\textsc{\char13}, \theta\textsc{\char13}; \alpha, \beta) = r\textsc{\char13}\left(1 - \beta + \beta  \cos(\alpha \theta\textsc{\char13}) \right).
\end{equation}

Thus, the parametrisations of the interior and exterior physical boundaries are defined as  $R_{I}: = R(r_{I}, \theta)$ and $R_{E}: = R(r_{E}, \theta)$, respectively. The analytical solution is chosen as the manufactured solution $u\left(x,y\right)=\log\left(x^2+y^2\right)$. In this test case, non-constant Dirichlet boundary conditions are imposed on both the interior and exterior boundaries.

We set $r_{I} = 0.5$, $r_{E} = 1$, the number of petals is $\alpha = 8$, and the perturbation magnitude is $\beta = 0.1$. The resulting rose-shaped domain is discretised using triangular elements (see Figure~\ref{roseDomain}). The advection-diffusion-reaction equation is solved, and the approximate solution is compared with the exact solution. The numerical simulations are carried out with successively finer meshes generated by Gmsh. Tables~\ref{tabela_rose_b1_c1}, \ref{tabela_rose_bexp_c05exp}, \ref{tabela_rose_b_c_polinom} and \ref{tabela_rose_b50poly_cexp} report the errors and convergence rates for the classical DG method and for the DG--ROD method considering $\boldsymbol{b}(\boldsymbol{x})=(1,1)^{\textrm{T}}$ and $c(\boldsymbol{x})=1$, $\boldsymbol{b}(\boldsymbol{x})=(\exp{(x)},0)^{\textrm{T}}$ and $c(\boldsymbol{x})=\frac{1}{2}\exp{(x)}$, $\boldsymbol{b}(\boldsymbol{x})=(2-y^2,2-x)^{\textrm{T}}$ and $c(\boldsymbol{x})= 1 + (1+x)(1+y)^2$ and $\boldsymbol{b}(\boldsymbol{x})=50(-y,x)^{\textrm{T}}$ and $c(\boldsymbol{x})=\exp{(x)}$, respectively. The results for the classical DG method confirm the accuracy deterioration due to the lack of specific treatment for curved boundaries, where the error convergence is limited to the second order. On the other hand, the results for the DG--ROD method show that the optimal convergence orders are recovered. The method behaves similarly to the previous test case, where the extensions of $u$ and $f$ are defined as described in Remark \ref{remark_nonConvex}, and the optimal convergence orders can be achieved for $N>2$.

\begin{figure}
  \begin{subfigure}[b]{.45\textwidth} 
  \centering {\includegraphics[width=5cm]{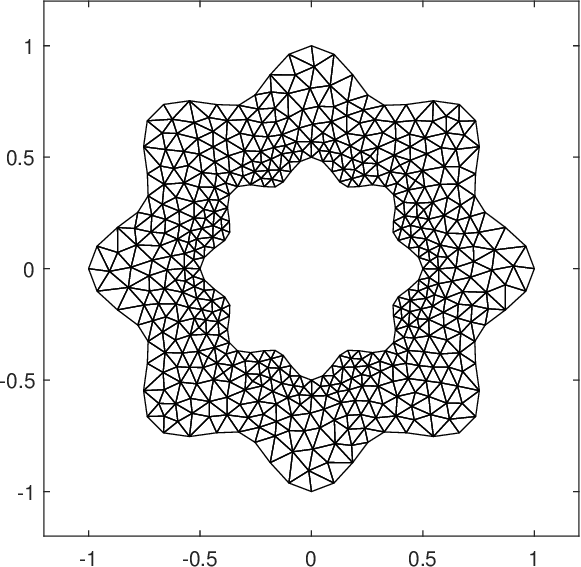}} \end{subfigure}
  \begin{subfigure}[b]{.45\textwidth}  \centering {\includegraphics[width=5cm]{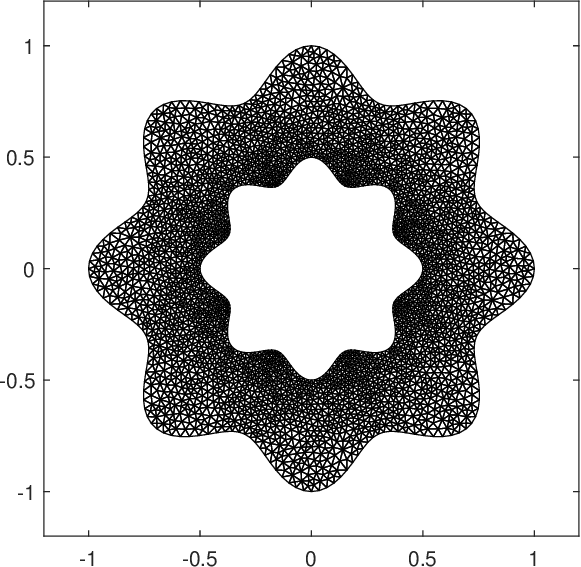}}\end{subfigure}
\caption[]{Unstructured mesh generated for the rose-shaped domain. Mesh with $K = 888$ and $h = 1.41\text{E$-$}01$ (left panel) and mesh with $K = 6912$ and $h = 5.25\text{E$-$}02$ (right panel).}
\label{roseDomain}
\end{figure}

\begin{table}[h!]
\caption{Errors and convergence orders for the classical DG method and for the DG--ROD method in the rose-shaped domain with the Dirichlet boundary conditions, $\boldsymbol{b}(\boldsymbol{x})=(1,1)^{\textrm{T}}$ and $c(\boldsymbol{x})= 1$.}
\label{tabela_rose_b1_c1}
\centering
\begin{adjustbox}{width=1\textwidth}
\begin{tabular}{@{}rc@{}r@{}cccc@{}r@{}cccc@{}r@{}cccc@{}}
\toprule
\multirow{3}{*}{$K$} & \multirow{3}{*}{$h$} & \phantom{aaaa} 
& \multicolumn{4}{c}{$N=2$} 
& \phantom{aaaa} 
& \multicolumn{4}{c}{$N=3$} 
& \phantom{aaaa} 
& \multicolumn{4}{c}{$N=4$} \\
\cmidrule{4-7}\cmidrule{9-12}\cmidrule{14-17}
& & 
& \multicolumn{2}{c}{DG} & \multicolumn{2}{c}{DG--ROD}
& & \multicolumn{2}{c}{DG} & \multicolumn{2}{c}{DG--ROD}
& & \multicolumn{2}{c}{DG} & \multicolumn{2}{c}{DG--ROD} \\
\cmidrule{4-5}\cmidrule{6-7}
\cmidrule{9-10}\cmidrule{11-12}
\cmidrule{14-15}\cmidrule{16-17}
& & 
& $E_2$ & $O_2$ & $E_2$ & $O_2$
& & $E_2$ & $O_2$ & $E_2$ & $O_2$
& & $E_2$ & $O_2$ & $E_2$ & $O_2$ \\
\midrule
240   & 2.52E$-$01 && 2.95E$-$02 & --- & 2.02E$-$04 & --- && 2.73E$-$02 & --- & 1.42E$-$05 & --- && 2.73E$-$02 & --- & 1.10E$-$06 & --- \\ 
888   & 1.41E$-$01 && 7.44E$-$03 & 2.4 & 1.41E$-$05 & 4.6 && 7.20E$-$03 & 2.3 & 4.52E$-$07 & 5.9 && 7.18E$-$03 & 2.3 & 1.68E$-$08 & 7.2 \\ 
3732  & 6.87E$-$02 && 1.53E$-$03 & 2.2 & 1.51E$-$06 & 3.1 && 1.50E$-$03 & 2.2 & 1.90E$-$08 & 4.4 && 1.50E$-$03 & 2.2 & 2.76E$-$10 & 5.7 \\ 
15130 & 3.54E$-$02 && 3.62E$-$04 & 2.2 & 1.80E$-$07 & 3.2 && 3.59E$-$04 & 2.2 & 1.08E$-$09 & 4.3 && 3.59E$-$04 & 2.2 & 9.18E$-$12 & 5.1 \\ 
60472 & 1.83E$-$02 && 8.99E$-$05 & 2.1 & 2.16E$-$08 & 3.2 && 8.96E$-$05 & 2.1 & 6.49E$-$11 & 4.2 && 8.95E$-$05 & 2.1 & --- & ---\\
\bottomrule
\end{tabular}
\end{adjustbox}
\end{table}

\begin{table}[h!]
\caption{Errors and convergence orders for the classical DG method in the rose-shaped domain with the Dirichlet boundary conditions, $\boldsymbol{b}(\boldsymbol{x})=(\exp{(x)},0)^{\textrm{T}}$ and $c(\boldsymbol{x})= \frac{1}{2}\exp{(x)}$.}
\label{tabela_rose_bexp_c05exp}
\centering
\begin{adjustbox}{width=1\textwidth}
\begin{tabular}{@{}rc@{}r@{}cccc@{}r@{}cccc@{}r@{}cccc@{}}
\toprule
\multirow{3}{*}{$K$} & \multirow{3}{*}{$h$} & \phantom{aaaa} 
& \multicolumn{4}{c}{$N=2$} 
& \phantom{aaaa} 
& \multicolumn{4}{c}{$N=3$} 
& \phantom{aaaa} 
& \multicolumn{4}{c}{$N=4$} \\
\cmidrule{4-7}\cmidrule{9-12}\cmidrule{14-17}
& & 
& \multicolumn{2}{c}{DG} & \multicolumn{2}{c}{DG--ROD}
& & \multicolumn{2}{c}{DG} & \multicolumn{2}{c}{DG--ROD}
& & \multicolumn{2}{c}{DG} & \multicolumn{2}{c}{DG--ROD} \\
\cmidrule{4-5}\cmidrule{6-7}
\cmidrule{9-10}\cmidrule{11-12}
\cmidrule{14-15}\cmidrule{16-17}
& & 
& $E_2$ & $O_2$ & $E_2$ & $O_2$
& & $E_2$ & $O_2$ & $E_2$ & $O_2$
& & $E_2$ & $O_2$ & $E_2$ & $O_2$ \\
\midrule
240   & 2.52E$-$01 && 2.97E$-$02 & --- & 2.04E$-$04 & --- && 2.75E$-$02 & --- & 1.42E$-$05 & --- && 2.75E$-$02 & --- & 1.11E$-$06 & --- \\ 
888   & 1.41E$-$01 && 7.50E$-$03 & 2.4 & 1.42E$-$05 & 4.6 && 7.25E$-$03 & 2.3 & 4.54E$-$07 & 5.9 && 7.22E$-$03 & 2.3 & 1.70E$-$08 & 7.2 \\ 
3732  & 6.87E$-$02 && 1.54E$-$03 & 2.2 & 1.51E$-$06 & 3.1 && 1.51E$-$03 & 2.2 & 1.90E$-$08 & 4.4 && 1.51E$-$03 & 2.2 & 2.76E$-$10 & 5.7 \\ 
15130 & 3.54E$-$02 && 3.64E$-$04 & 2.2 & 1.80E$-$07 & 3.2 && 3.61E$-$04 & 2.2 & 1.08E$-$09 & 4.3 && 3.61E$-$04 & 2.2 & 9.97E$-$12 & 5.0  \\ 
60472 & 1.83E$-$02 && 9.04E$-$05 & 2.1 & 2.16E$-$08 & 3.2 && 9.01E$-$05 & 2.1 & 6.49E$-$11 & 4.2 && 9.00E$-$05 & 2.1 & --- & --- \\
\bottomrule
\end{tabular}
\end{adjustbox}
\end{table}

\begin{table}[h!]
\caption{Errors and convergence orders for the classical DG method and for the DG--ROD method in the rose-shaped domain with the Dirichlet boundary conditions, $\boldsymbol{b}(\boldsymbol{x})=(2-y^2,2-x)^{\textrm{T}}$ and $c(\boldsymbol{x})= 1 + (1+x)(1+y)^2$.}
\label{tabela_rose_b_c_polinom}
\centering
\begin{adjustbox}{width=1\textwidth}
\begin{tabular}{@{}rc@{}r@{}cccc@{}r@{}cccc@{}r@{}cccc@{}}
\toprule
\multirow{3}{*}{$K$} & \multirow{3}{*}{$h$} & \phantom{aaaa} 
& \multicolumn{4}{c}{$N=2$} 
& \phantom{aaaa} 
& \multicolumn{4}{c}{$N=3$} 
& \phantom{aaaa} 
& \multicolumn{4}{c}{$N=4$} \\
\cmidrule{4-7}\cmidrule{9-12}\cmidrule{14-17}
& & 
& \multicolumn{2}{c}{DG} & \multicolumn{2}{c}{DG--ROD}
& & \multicolumn{2}{c}{DG} & \multicolumn{2}{c}{DG--ROD}
& & \multicolumn{2}{c}{DG} & \multicolumn{2}{c}{DG--ROD} \\
\cmidrule{4-5}\cmidrule{6-7}
\cmidrule{9-10}\cmidrule{11-12}
\cmidrule{14-15}\cmidrule{16-17}
& & 
& $E_2$ & $O_2$ & $E_2$ & $O_2$
& & $E_2$ & $O_2$ & $E_2$ & $O_2$
& & $E_2$ & $O_2$ & $E_2$ & $O_2$ \\
\midrule
240   & 2.52E$-$01 && 2.97E$-$02 & --- & 2.04E$-$04 & --- && 2.75E$-$02 & --- & 1.44E$-$05 & --- && 2.75E$-$02 & --- & 1.11E$-$06 & --- \\ 
888   & 1.41E$-$01 && 7.49E$-$03 & 2.4 & 1.42E$-$05 & 4.6 && 7.25E$-$03 & 2.3 & 4.53E$-$07 & 5.9 && 7.23E$-$03 & 2.3 & 1.69E$-$08 & 7.2 \\ 
3732  & 6.87E$-$02 && 1.54E$-$03 & 2.2 & 1.51E$-$06 & 3.1 && 1.52E$-$03 & 2.2 & 1.90E$-$08 & 4.4 && 1.51E$-$03 & 2.2 & 2.76E$-$10 & 5.7 \\ 
15130 & 3.54E$-$02 && 3.64E$-$04 & 2.2 & 1.80E$-$07 & 3.2 && 3.62E$-$04 & 2.2 & 1.08E$-$09 & 4.3 && 3.61E$-$04 & 2.2 & 9.39E$-$12 & 5.1 \\ 
60472 & 1.83E$-$02 && 9.06E$-$05 & 2.1 & 2.16E$-$08 & 3.2 && 9.03E$-$05 & 2.1 & 6.50E$-$11 & 4.2 && 9.02E$-$05 & 2.1 & --- & --- \\
\bottomrule
\end{tabular}
\end{adjustbox}
\end{table}

\begin{table}[h!]
\caption{Errors and convergence orders for the classical DG method and for the DG--ROD method in the rose-shaped domain with the Dirichlet boundary conditions, $\boldsymbol{b}(\boldsymbol{x})=50(-y,x)^{\textrm{T}}$ and $c(\boldsymbol{x})= \exp{(x)}$.}
\label{tabela_rose_b50poly_cexp}
\centering
\begin{adjustbox}{width=1\textwidth}
\begin{tabular}{@{}rc@{}r@{}cccc@{}r@{}cccc@{}r@{}cccc@{}}
\toprule
\multirow{3}{*}{$K$} & \multirow{3}{*}{$h$} & \phantom{aaaa} 
& \multicolumn{4}{c}{$N=2$} 
& \phantom{aaaa} 
& \multicolumn{4}{c}{$N=3$} 
& \phantom{aaaa} 
& \multicolumn{4}{c}{$N=4$} \\
\cmidrule{4-7}\cmidrule{9-12}\cmidrule{14-17}
& & 
& \multicolumn{2}{c}{DG} & \multicolumn{2}{c}{DG--ROD}
& & \multicolumn{2}{c}{DG} & \multicolumn{2}{c}{DG--ROD}
& & \multicolumn{2}{c}{DG} & \multicolumn{2}{c}{DG--ROD} \\
\cmidrule{4-5}\cmidrule{6-7}
\cmidrule{9-10}\cmidrule{11-12}
\cmidrule{14-15}\cmidrule{16-17}
& & 
& $E_2$ & $O_2$ & $E_2$ & $O_2$
& & $E_2$ & $O_2$ & $E_2$ & $O_2$
& & $E_2$ & $O_2$ & $E_2$ & $O_2$ \\
\midrule
240   & 2.52E$-$01 && 2.97E$-$02 & --- & 2.03E$-$04 & --- && 2.75E$-$02 & --- & 1.42E$-$05 & --- && 2.75E$-$02 & --- & 1.11E$-$06 & --- \\ 
888   & 1.41E$-$01 && 7.50E$-$03 & 2.4 & 1.42E$-$05 & 4.6 && 7.25E$-$03 & 2.3 & 4.52E$-$07 & 5.9 && 7.22E$-$03 & 2.3 & 1.70E$-$08 & 7.2 \\ 
3732  & 6.87E$-$02 && 1.54E$-$03 & 2.2 & 1.51E$-$06 & 3.1 && 1.51E$-$03 & 2.2 & 1.90E$-$08 & 4.4 && 1.51E$-$03 & 2.2 & 2.76E$-$10 & 5.7 \\ 
15130 & 3.54E$-$02 && 3.64E$-$04 & 2.2 & 1.80E$-$07 & 3.2 && 3.61E$-$04 & 2.2 & 1.08E$-$09 & 4.3 && 3.61E$-$04 & 2.2 & 9.97E$-$12 & 5.0 \\ 
60472 & 1.83E$-$02 && 9.04E$-$05 & 2.1 & 2.16E$-$08 & 3.2 && 9.01E$-$05 & 2.1 & 6.49E$-$11 & 4.2 && 9.00E$-$05 & 2.1 & --- & ---\\
\bottomrule
\end{tabular}
\end{adjustbox}
\end{table}

\section{Conclusions}\label{conclusions}
In this work, we discussed a discontinuous Galerkin scheme on a polygonal computational domain $\Omega_h$ for boundary value problems posed on two-dimensional domains $\Omega$ with curved boundaries. The DG--ROD method compensates for the geometric inconsistency resulting from approximating the curved boundary $\partial \Omega$ by a piecewise linear $\partial \Omega_h$ via a polynomial reconstruction of the boundary data, incorporated directly into the variational formulation. In practice, DG--ROD can be interpreted as an iterative combination of the standard DG method with boundary data reconstruction, enabling the use of simple polygonal meshes while restoring optimal convergence rates.

We presented a study on the existence and uniqueness of the solution for the advection-diffusion-reaction equation with homogeneous Dirichlet boundary conditions. Following and extending the analysis in \cite{ruas2019OptimalLA}, we provided a complete mathematical study of convergence in the natural DG norm, as well as $L^2$-error estimates, for both convex and non-convex domains. For convex domains, we proved that the DG--ROD solution achieves an optimal $\mathcal{O}(h^{N+1})$ convergence rate in the $L^2$-norm when $N$-degree piecewise polynomials are used, under suitable regularity assumptions on the solution. For non-convex domains, optimality is not attained for $N > 2$ unless the solution satisfies stronger regularity conditions and the computational domain $\Omega_h$ approximates the physical domain $\Omega$ with higher accuracy, i.e., when $\Omega \setminus \Omega_h$ is of order $h^q$, with $q > 2$. In other words, the error is affected by the geometric mismatch of order $\mathcal{O}(h^2)$ between the curved physical boundary and its piecewise linear representation. Nevertheless, the numerical experiments reported in Tables~\ref{tabela_annulus_b1_c1}--\ref{tabela_rose_b50poly_cexp} show that the DG--ROD method recovers optimal convergence rates even for $N=3$ and $N=4$. This behaviour is explained by the fact that our benchmark problems employ smooth manufactured solutions, which allow for a regular extension $\Tilde{u}$ that satisfies the governing equation in the mismatch region. In such cases, the geometric residual decays at an optimal rate, effectively bypassing the theoretical worst-case limitation. These results suggest that while the estimates in Theorems~\ref{nonConvex_errorDG_comL2}–\ref{errorNonConvex_L2} provide a rigorous upper bound for general problems where such ideal extensions are not guaranteed, the method remains highly accurate for practical applications with smooth solutions.

The effectiveness of the proposed framework is validated by a series of numerical experiments on both convex and non-convex domains. These benchmarks confirm that the DG--ROD method successfully restores optimal convergence, matching the theoretical predictions for convex cases and demonstrating high robustness in non-convex geometries. It is also important to highlight that while the formal mesh size assumptions are a sufficient condition for the theoretical analysis, the numerical results indicate that accurate approximations are attainable even on relatively coarse meshes. For instance, as shown in Section~\ref{diskDomain_section}, the polynomial reconstruction successfully corrects the geometric error even when the mesh size $h$ is of the same order of magnitude as the radius of the domain. We also note that, in \cite{DGROD}, the authors considered a diffusion--reaction problem and constructed the DG--ROD method through an iterative combination of the DG method and polynomial reconstruction, exploring different strategies for placing the nodes of $\mathcal{W}_h$ on $\partial \Omega$; in their numerical results, they quantified accuracy via the $L^\infty$- and $L^1$-errors. 

Extensions of this work to nonlinear equations and time-dependent problems remain challenging and will be pursued in future research. We also plan to extend the DG--ROD approach to other types of boundary conditions (Neumann, Robin) and derive corresponding error estimates.  

Finally, it is noteworthy that the DG--ROD method can be extended to three-dimensional problems. Indeed, the three-dimensional setting has been investigated for classical Lagrange finite elements in \cite{10.1093/imanum/draa029}, where a closely related technique was analysed. Following the same principles as in the two-dimensional case, both the numerical strategy and the associated error analysis can be adapted to the DG framework. \\

\noindent {\small \textbf{Funding}
The authors acknowledge financial support by the Centre for Mathematics of the University of Coimbra (CMUC, https://doi.org/10.54499/UID/00324/2025) under the Portuguese Foundation for Science and Technology (FCT), Grants UID/00324/2025 and UID/PRR/00324/2025. M. Santos acknowledges FCT for the support under the Ph.D. scholarship UI/BD/153816/2022 (https://doi.org/10.54499/UI/BD/153816/2022). \\

%

\noindent  {\small \textbf{Data Availability} The paper has no associated data.}

\section*{Declarations}
\noindent {\small \textbf{Conflict of interest} The authors declare that they have no conflict of interest.}

\printbibliography 

\end{document}